\documentclass{amsart}
\usepackage[T1]{fontenc} 
\usepackage{amssymb,enumerate}

\usepackage{amsthm}
\usepackage{amsmath}
\usepackage{xcolor}
\usepackage[all]{xy}
\xyoption{arrow}
\usepackage[pagebackref]{hyperref}
\usepackage{theoremref}

\theoremstyle{plain}\newtheorem{Theorem}{Theorem}[section]
\theoremstyle{plain}\newtheorem{Conjecture}[Theorem]{Conjecture}
\theoremstyle{plain}\newtheorem{Corollary}[Theorem]{Corollary}
\theoremstyle{plain}\newtheorem{Lemma}[Theorem]{Lemma}
\theoremstyle{plain}\newtheorem{Proposition}[Theorem]{Proposition}
\theoremstyle{definition}\newtheorem{Definition}[Theorem]{Definition}
\theoremstyle{definition}
\theoremstyle{definition}
\theoremstyle{definition}
\theoremstyle{definition}
\theoremstyle{definition}\newtheorem{Remark}[Theorem]{Remark}
\theoremstyle{definition}

 \def\OH{{\mathcal{O}H}} 
  
\def\CD{{\mathcal{D}}}  
\def\CE{{\mathcal{E}}}

\def\CO{{\mathcal{O}}}

\def\CX{{\mathcal{X}}}
\def\CY{{\mathcal{Y}}}

\def\Aut{\mathrm{Aut}}                    
\def\Br{\mathrm{Br}}

\def\Ext{\mathrm{Ext}}        

\def\GL{\mathrm{GL}}

\def\Ind{\mathrm{Ind}}

\def\Inn{\mathrm{Inn}}

\def\Irr{\mathrm{Irr}}

\def\GU{\mathrm{GU}}
\def \O{\mathrm{O}}
\def\PGL{\mathrm{PGL}}
\def\PO{\mathrm{P\Omega}}
\def\PSL{\mathrm{PSL}}
\def\PSp{\mathrm{PSp}}
\def\PSU{\mathrm{PSU}}
\def\SL{\mathrm{SL}}
\def\SO{\mathrm{SO}}

\def\Spin{\mathrm{Spin}}
\def\SU{\mathrm{SU}}

\def\opp{\mathrm{opp}}  
\def\Out{\mathrm{Out}}

\def\Res{\mathrm{Res}}

\def\Tr{\mathrm{Tr}}             
             
\def\mf{\mathrm{mf}}             
\def\Lin{\mathrm{Lin}}

\def\Stab{\mathrm{Stab}}

\newcommand{\bC}{{\mathbf{C}}}
\newcommand{\bG}{{\mathbf{G}}}
\newcommand{\bH}{{\mathbf{H}}}
\newcommand{\bL}{{\mathbf{L}}}

\newcommand{\bN}{{\mathbf{N}}}
\newcommand{\bP}{{\mathbf{P}}}

\newcommand{\bT}{{\mathbf{T}}}
\newcommand{\bU}{{\mathbf{U}}}

    \newcommand{\type}{\operatorname}

 \newcommand{\wt}{\widetilde}
  
	\newcommand{\tw}[1]{{}^#1\!}

\newcommand{\tA}{\type{A}}
\newcommand{\tB}{\type{B}}
\newcommand{\tC}{\type{C}}
\newcommand{\tD}{\type{D}}
\newcommand{\tE}{\type{E}}
\newcommand{\tF}{\type{F}}
\newcommand{\tG}{\type{G}}
\newcommand{\cent} {C}
\newcommand{\zent} {Z}
\newcommand{\norm} {N}
\newcommand{\sym}{{S}}
\newcommand{\alt}{{A}}
\newcommand{\FF}{\mathbb{F}}

\numberwithin{equation}{section}

\makeindex
\title{On  $2$-blocks with quaternion  defect groups} 

\author[Eaton]{Charles Eaton}
\address[Eaton]{Department of Mathematics, University of Manchester,  Oxford Road, Manchester, M13 9PL}
\email{charles.eaton@manchester.ac.uk}

\author[Eisele]{Florian Eisele}
\address[Eisele]{Department of Mathematics, University of Manchester,  Oxford Road, Manchester, M13 9PL}
\email{florian.eisele@manchester.ac.uk}

\author[Kessar]{Radha Kessar}
\address[Kessar]{Department of Mathematics, University of Manchester,  Oxford Road, Manchester, M13 9PL}
\email{radha.kessar@manchester.ac.uk}
\author[Linckelmann]{Markus Linckelmann}
\address[Linckelmann]{Department of Mathematics, City St George's, University of London EC1V 0HB,
 United Kingdom}
\email{markus.linckelmann.1@city.ac.uk}

\author[Schaeffer Fry]{A. A. Schaeffer Fry}
\address[Schaeffer Fry]{{Department of Mathematics}, {University of Denver}, {Denver, CO 80210, USA}}
\email{mandi.schaefferfry@du.edu}

\date{\today}
\thanks{The second author was supported by the EPSRC grant UKRI2780. The fourth author was supported by EPSRC grant EP/X035328/1. The fifth author was supported by the U.S. National Science Foundation, Award No. DMS-2439897}

\begin{document}

\begin{abstract}
We determine the Morita equivalence classes of $2$-blocks with  quaternion
defect groups of arbitrary $2$-power order, thereby completing the proof of Donovan's 
conjecture for blocks of tame representation type.
\end{abstract}

\maketitle

\section{Introduction.}  \label{Intro-section} 

Nearly four decades ago, Erdmann classified blocks of finite groups of tame representation
type in terms of their quivers and relations, except that some of the quivers were 
not known to arise as blocks, and certain scalars in the relations remained undetermined.
Erdmann's work, in a series of papers leading up to the monograph \cite{er90}, 
came tantalisingly close to proving Donovan's conjecture for all tame
blocks, were it not for the undetermined scalars. Some of these were ruled out
or narrowed down to finitely many possibilities in subsequent work by several authors,
including Eisele \cite{Eisele16}, Holm \cite{Holm99}, and Macgregor
\cite{Macgregor}. 

At the time of writing, Donovan's conjecture  was known to 
hold for all tame blocks except blocks with a  quaternion defect group of order at least $16$
and two isomorphism classes of simple modules. The main purpose of this paper is to prove
Donovan's conjecture for this remaining case, thereby completing the proof of Donovan's
conjecture for all tame blocks. We do this by proving a stronger
result which implies Donovan's conjecture for blocks with a 
quaternion defect group, of arbitrary $2$-power order, 
over a complete discrete valuation ring. The proof, which
uses the classification of finite simple groups, yields 
furthermore  a list  containing all Morita equivalence classes of these  blocks.

In order to state this, let $\CO$ be a complete discrete valuation ring with
an algebraically closed residue field $k$ of characteristic $2$ and a field of
fractions $K$ of characteristic zero. By a quaternion group we mean a (possibly
generalised) quaternion group of order $2^n$ for some integer $n\geq 3$. 

\begin{Theorem}\thlabel{Main1}   
The Morita--Frobenius  number of any block of a finite group algebra over $\CO$ with defect groups isomorphic to a quaternion group $Q_{2^n}$ equals $1$. In particular, Donovan's conjecture holds over $\CO$ and over $k$ for blocks with a quaternion defect group. 
\end{Theorem}

\begin{Corollary} \thlabel{tame-Donovan-Corollary}
Donovan's conjecture holds for all tame blocks of finite group algebras over $k$. 
\end{Corollary}

We expect the above  to hold over $\CO$, but we currently do not have complete information 
about blocks with semidihedral defect groups to conclude this. See \thref{Remark: Donovan over O for dihedral and semidihedral} for a detailed account of what is known.

Theorem  \ref{Main1}  is  a consequence of the  following result.

\begin{Theorem} \thlabel{Main2}  
Any block of a finite group algebra over $\CO$  with defect groups isomorphic to a
quaternion group $Q_{2^n} $ of order $2^n$ for some integer $n\geq 3$
is  either Morita equivalent to a principal  block or a block, over $\CO$, in the  following list.
\begin{enumerate}
\item One of the two nonprincipal blocks of maximal defect of $6.A_7$.
\item  The unique block of $2.\Aut(J_2)$ or $2.S_7$ with $Q_{16}$ defect groups, where $2.S_7$ is the $-$ type double cover of $S_7$.
\item  The unique  block of $2.\type{B}_3(3)$ with $Q_{16}$ defect groups or one of the two  
faithful blocks of $6.\type{B}_3(3)$ with $Q_{16}$ defect groups. 
\end{enumerate}

\end{Theorem}

In \cite{kl20}, Koshitani and Lassueur classified principal blocks of finite groups with  a 
 quaternion Sylow $2$-subgroup up to source algebra equivalence, implying
in particular a classification of these blocks up to Morita equivalence over $\CO$.
We complete these results with a  list of all Morita equivalence classes of blocks of 
finite group algebras over  $\CO$ with a quaternion defect group.

\begin{Corollary}\thlabel{MainCorollary2}  
Let  $n \geq 3 $ and  let $q_1 $ and $q_2$ be odd prime powers satisfying $(q_1-1)_2 = 2^{n-1}$ 
and $(q_2+1)_2  =2^{n-1}$. If $n>3$, let $q_3$ and $q_4$ denote odd prime powers such that 
$(q_3-1)_2 = 2^{n-2}$, $(q_4+1)_2 =2^{n-2}$. Let  $G$ be a  finite group  and let $b $ be a 
block idempotent of $\CO G$. Suppose that the defect groups of $\CO Gb$  are isomorphic to  
$Q_{2^n}$.  Then   $\CO Gb$  is Morita equivalent  to  one of the following blocks over $\CO$.
\begin{enumerate}
    \item  $\CO Q_{2^n}$.
    \item  The principal  block of $\SL_2(q_1)$.
    \item The principal block   of $\SL_2(q_2)$.
    \item The principal block   of $2.\PGL_2(q_3)$ (assuming $n>3$).
    \item The principal block   of $2.\PGL_2(q_4)$ (assuming $n>3$).
\item One of the blocks of maximal defect of $6.A_7$.
\item  The unique block of $2.\Aut(J_2)$ or $2.S_7$ with $Q_{16}$ defect groups.
\item  The unique block of $2.\type{B}_3(3)$ with $Q_{16}$ defect groups or one of the two faithful blocks of $6.\type{B}_3(3)$ with $Q_{16}$ defect groups. 
\end{enumerate}

In the above, $2.\PGL_2(q_i)$ for $i\in\{3,4\}$ denotes the unique central extension of $\PGL_2(q_i)$ by the cyclic group $C_2$ whose Sylow $2$-subgroups are isomorphic to $Q_{2^n}$
\end{Corollary}

We do not know (and do not expect) this list to be irredundant -- that is, there may
be Morita equivalences between some members of this list. In particular, we do not
know whether every block in this list is Morita equivalent to some principal block.
Most of the Morita equivalences encountered in the proofs of the above results
are induced by bimodules with endopermutation sources, suggesting that the
classification up to Morita equivalence should come close to a classification of
the source algebras of these blocks, thereby implying Puig's conjecture for
blocks with a quaternion defect group. The  main obstacle that stands in
the way of a complete classification of these source algebras is the fact that we do 
not know the sources of the bimodules inducing the Bonnaf\'e--Dat--Rouquier 
equivalences. 

 Our results rely on the classification of finite simple groups. In Section \ref{sec:generalities} we gather some general background representation theoretical  results  and  Section  \ref{sec:tamebackground} gives  the required background on blocks with quaternion defect groups.   The main results of Section \ref{socle scalars} are Propositions ~\ref{prop:Q2A} and \ref{prop:Q2B}, which show that bounding Morita--Frobenius numbers of blocks with quaternion defect groups and two simple modules yields  analagous bounds on the socle scalars appearing in Erdmann's descriptions. Section \ref{sec:redthms} contains Clifford theoretic  
 reduction theorems -- the final reduction statement is  Corollary~\ref{main-reduction}. In Section \ref{sec:extendBDR}, we  present  some  general results  extending Bonnaf\'e--Rouquier  and  Bonnaf\'e--Dat--Rouquier Morita equivalences  to overgroups  acting as diagonal and field automorphisms -- these results  may  find  further applications.  Sections \ref{sec:lie-non} and \ref{sec:AD}  contain  an  analysis  of quaternion  blocks   arising from finite groups of Lie type in odd characteristic,  culminating in Theorem~\ref{lietype}. 
  Section \ref{sec:altspor}  deals with  the remaining   
 families of finite simple groups. In particular, when combined with Section \ref{sec:lie-non}, the results of Section \ref{sec:altspor} complete the description of blocks of quasisimple groups with quaternion defect groups, which could be of independent interest.
 Section \ref{sec:mainproofs}  contains the proofs of the theorems stated in this introduction.

\section{Generalities.}\label{sec:generalities} 

\subsection{Basic notation and blocks}

Throughout let $\ell $  be a prime and let $(K, \CO, k) $  be an $\ell$-modular  system  such that $\CO$ is absolutely  unramified  and such that $k \cong \bar{\mathbb {F}}_{\ell}$. Later, we will often have $\ell=2$, but sometimes our results apply in greater generality. For $G$ a finite group, a block of $\CO G$ or $kG$ will mean an indecomposable  $k$-algebra  factor, and will usually be written as, e.g., $\CO Gb$ where $b$ is the block idempotent. We denote by $\ell(kGb)$ the number of isomorphism classes of simple $kGb$-modules. We will use standard notation for
fixed points under group actions and relative traces; see e.g. \cite[\S 2.5]{LiBookI}.

Recall that  if  $N$ is  a normal subgroup of  $G$, $c$ is a central idempotent of $\CO N$  and  $b$ is a central idempotent of $\CO G$, then we  say that $b$ covers $c$  (or that $\CO Gb$ covers $\CO N c$)  if $bc\ne 0 $. We use the following facts about covered blocks freely, and explanation may be found in~\cite[Sec.~6.8]{LiBookII}.
Assume now that $b$ and $c$ are block idempotents. The block idempotents of $\CO N$ covered by $b$ form a single $G$-conjugacy class. If $D$ is a defect group for $\CO Gb$, then $D \cap N$ is a defect group for a block of $\CO N$ covered by $\CO Gb$, and conversely if $Q$ is a defect group for $\CO Nc$, then there is a defect group $D$ for $\CO Gb$ such that $D \cap N = Q$. Writing $H$ for the stabiliser of $c$ in $G$, there exists a block of $\CO G$ covering $\CO Nc$ with defect group $D$ such that $DN/N$ is a Sylow $\ell$-subgroup of $H/N$. If $G/N$ is an $\ell$-group, then there is a unique block of $\CO G$ covering $\CO Nc$, and in particular if $c$ is further $G$-stable, then $G=ND$ and $b=c$ (see~\cite[Thm.~ 6.5.4]{LiBookII}, noting that we have assumed $k$ is algebraically closed). Of particular use are the following:

\begin{Remark}\thlabel{remark:fongcorrespondent}
  Let $N$ be a normal subgroup of $G$ and let $b$ be a block idempotent of $\CO G$ covering a block idempotent $e$ of $\CO N$. Let $T$ denote the stabiliser of $e$ in $G$.
  \begin{enumerate}
    \item There is a unique block idempotent $d$ of $\CO T$ such that $bd=d=ed$ and $b=\Tr_T^G(d)$. Moreover, $\CO G b \cong \mathrm{Mat}_n(\CO T d)$ where $n=[G:T]$. See \cite[Prop.~6.8.3]{LiBookII}.
    \item If $G/N$ is an $\ell$-group, then $b$ is the unique block idempotent covering $e$ and $b=\Tr_T^G(e)$. In particular, $b\in \CO N$ and $e$ is a block idempotent of $\CO T$ (the latter follows from the case where $G=T$). See \cite[Prop.~6.8.11]{LiBookII}.
  \end{enumerate}
\end{Remark}

The following consequence of these facts will be useful later.

\begin{Lemma}  \label{lem-stableblock}
Let $X$ be a finite group and $G$, $N$ normal subgroups of $X$ such that $N\leq G$.
Suppose that $X/G$ has $\ell$-power order.
Let $c$ be a block idempotent of $\CO N$.  Denote by $Y$ the stabiliser of $c$ in $X$,
and suppose that $GY=X$. 
 Then there exists an $X$-stable block idempotent $b$ of $\CO G$ such that $b$ covers $c$
and such that $b$ has a defect group $P$ whose image in $G/N$ is a Sylow 
$p$-subgroup of $(G\cap Y)/N$.
\end{Lemma}

\begin{proof}
Note that $Y$ contains $N$. Set $H=G\cap Y$; that is, $H$ is the stabiliser in $G$ of $c$.
There exists a block idempotent $d$ of $\CO Y$
such that $dc=d$ and such that if $R$ is any defect group of $\CO Yd$, then the
image of $R$ in $Y/N$ is a Sylow $p$-subgroup of $Y/N$. Since $G$ has $\ell$-power
index in $X=GY$, it follows that $X=GR$, so intersecting with $Y$ yields
$Y=HR$. By \cite[Thm.~6.5.4]{LiBookII}, noting that $k$ is algebraically closed, we have that $d$ is contained in $\CO H$,
hence that $d$ belongs to the fixed point algebra  $ (\CO H)^Y$. 
Since $(\CO H)^Y$ is contained in $Z(\CO H)$, the
unique primitive decomposition of $d$ in $Z(\CO H)$ takes the form
$d=\sum_{i=1}^r\ e_i$, where the $e_i$ are the block
idempotents of $\CO H$ satisfying $de_i=e_i$. Since $R$ is a defect group of $\CO Yd$ 
we have $\Br_R(d)\neq 0$. Since $R$ permutes the set $\{e_i\ | \ 1 \leq i\leq r\}$ this
forces that there is one of the $e_i$ which is fixed by $R$; call this block $e$.
Since $Y=HR$ this implies that
$e$ is a $Y$-stable block idempotent of $\CO H$ satisfying $ec=e$.
Set $b=$ $\Tr^G_H(e) =$ $\Tr^G_H(ec)=$ $\Tr^G_H(bce)$. 
Thus $bce\neq 0$, so $bc\neq 0$.
By \cite[Thm.~6.8.3, Lem.~6.8.4]{LiBookII}, $b$ is a block idempotent of $\CO G$.
Moreover, $b$ is both $G$-stable (it belongs to $Z(\CO G)$) and $Y$-stable, because
$Y$ normalises both $G$ and $H=G\cap Y$ and stabilises $e$. Since $X=GY$ it
follows that  $b$ is $X$-stable.
Then $P=H\cap R$ is a defect group of $\CO He$, which by
\cite[Thm.~6.8.3]{LiBookII} remains a defect group of $\CO Gb$. Since $RN/N$ is a 
Sylow $\ell$-subgroup of $Y/N$, intersecting with the normal subgroup $H/N$ of 
$Y/N$ shows that $PN/N$ is a Sylow $\ell$-subgroup of $H/N$. This completes the
proof of the Lemma.
\end{proof}

\begin{Lemma}
\thlabel{solv_quotient_characteristic:lemma}
Let $G$ be a finite group and  $\CO G b$ a block  with defect group $D$. Assume that whenever $N$ is a characteristic subgroup of $G$, the block $\CO G b$ covers a unique block $\CO N c$. Let $N \lhd G$  be a characteristic subgroup such that $G/N$ is solvable and let $\CO N c$ be the unique block covered by $\CO Gb$. Then $ND/N$ is a Sylow $p$-subgroup of $G/N$.
\end{Lemma}
\begin{proof}
  This can be proved just   as  in~\cite[Lem.~2.4]{ar19}.
\end{proof}

\subsection{Donovan's conjecture and Morita--Frobenius numbers}

\begin{Conjecture}[Donovan]
Let $P$ be a finite $\ell$-group. There are only finitely many Morita equivalence classes of blocks of finite groups with defect group isomorphic to $P$. This may be stated over $k$ or $\CO$. 
\end{Conjecture}

The conjecture makes sense over any complete, local and commutative coefficient ring $R$ with algebraically closed residue field of characteristic $\ell$. However, the version over $k$ is the original form of the conjecture, and the version over $\CO$ is the (stronger) form of the conjecture usually considered in recent literature on the topic. We will therefore restrict ourselves to these two coefficient rings. See~\cite{eel20} for a discussion on choice of $\CO$.

By~\cite{ke04} (for $k$-blocks) and~\cite[Cor.~3.11]{eel20} (for $\CO$-blocks), Donovan's conjecture reduces to the bounding of the Cartan invariants and Morita--Frobenius numbers as we briefly describe here.

Denote by $\overline{\sigma}:k \rightarrow k$ the Frobenius automorphism, so $\overline{\sigma}(\lambda)=\lambda^\ell$ for each $\lambda\in k$. Since $\CO$ is absolutely unramified, $\overline{\sigma}$ lifts to a unique automorphism $\sigma: \CO \rightarrow \CO$, which raises every $\ell'$-root of unity to its $\ell$-power. For $G$ a finite group, we may then define ring automorphisms $\overline{\sigma}$ and $\sigma$ of $kG$ and $\CO G$ by $\varphi(\sum_{a \in G} a_gg)=\sum_{a \in G} \varphi(a_g)g$ for $\varphi=\overline{\sigma}$ and $\sigma$. These ring automorphisms permute the corresponding block idempotents of $\CO G$ and $kG$ identically. Writing $\pi:\CO G \rightarrow kG$ for the canonical ring automorphism and $b$ for a block idempotent of $\CO G$, the \emph{Morita--Frobenius number} $\mf_k(kG\pi(b))$ is defined to be the smallest positive $m \in \mathbb{Z}$ such that $kG\pi(b)$ is Morita equivalent to $kG\overline{\sigma}^m(\pi(b))$, and $\mf_\CO(\CO Gb)$ is defined to be the smallest positive $m \in \mathbb{Z}$ such that $\CO G b$ is Morita equivalent to $\CO G\sigma^m(b)$. The Morita--Frobenius number for any block Morita equivalent to a principal block is one.

\begin{Proposition}[~\cite{ke04} and~\cite{eel20}]
\thlabel{Donovan_and_Morita_Frob}
Let $P$ be a finite $\ell$-group. Then Donovan's Conjecture holds for $P$ if and only if there are $m(P)$ and $c(P)$ such that for all blocks of finite groups with defect group isomorphic to $P$, the Morita--Frobenius number is at most $m(P)$ and the largest entry of the Cartan matrix is at most $c(P)$.     
\end{Proposition}

For a finite group $G$, denote by $\Lin(G)$ the group of linear characters, and recall that when $\eta \in \Lin(G)$ has values in $\CO$ (such as when $\eta$ has $\ell'$ order), $\eta$ induces an $\CO$-algebra automorphism of $\CO G$  via 
$$  \eta. \left(\sum_{g \in G} \alpha_g   g\right)   =    \sum_{g\in G} \alpha_g  \eta(g)  g.$$

\begin{Lemma} \label{lem:normalabelian}
 Let $N$  be a normal subgroup of a  finite group $G$  with $G/N $ abelian.  Let $e$ be  a block idempotent of  $\CO N$. If $b_1 $ and $b_2$ are block idempotents of $\CO  G$  covering $e$, then  there exists $\eta \in \mathrm{Lin} (G/N)$ of $\ell'$ order such that  the  action  of $\eta $ on $\CO G$   restricts to  an  $\CO$-algebra  isomorphism  $\CO G b_1 \cong   \CO G b_2$. In particular, the number of block idempotents  of $\CO G$  covering $e$ is prime to ${\ell}$. 
\end{Lemma}

\begin{proof} By, for example, \cite[Thm.~9.4]{N98} for each irreducible character $\theta$ of $N$ in $e$ there are irreducible characters  $\chi_1 $  of $G$ in $b_1 $ and  $\chi_2$ in $b_2 $  covering $\theta$.  Then the first assertion follows by \cite[Lem.~2.9]{FK}. The second assertion is a consequence of the  first and the fact  (also explained in the proof of \cite[Lem.~2.9]{FK}) that multiplying any irreducible character of a block of $\CO G$ by a linear character of $G/N$ of $\ell$-power order results in an irreducible character in the same block.\end{proof}

Recall from above that $\sigma : \CO  \to \CO $ is the ring automorphism  which  raises  every $\ell'$-root  of unity to  its $\ell$-power, and that $\sigma $ is uniquely determined by this property.

\begin{Proposition} \label{prop:frobeniusnormal}   Let $N$
and $H$ be normal subgroups of a  finite group $G$ with $N \leq H$. Suppose that  $H/N$ is  an abelian $\ell'$-group and  $G/H$ is an $\ell$-group.    Let $e$  be a block idempotent of  $\CO N $  with the property that $\,^\sigma  e = e$  and let $b$ be a block idempotent of $\CO G$  covering  $e$. Then $\mf_\CO(\CO G b)=1$.     If   $e$ is the principal block  idempotent of $\CO N$, then  $\CO Gb $  is Morita equivalent to  $\CO  Td$,   for some $T$, $d$  with  $ H\leq  T \leq G$  and   $d$ the principal block idempotent of   $\CO T$.
\end{Proposition}

\begin{proof}    Since $G/H$ is an $\ell$-group Remark~\ref{remark:fongcorrespondent} implies $b \in   \CO  H$  and there is  a  block idempotent $f$ of $\CO H $  with $fe \ne 0 $   and  such that  $b= \Tr_{T}^G  f$, where  $T=  \Stab_G(f)$. Moreover, $f$  is a block idempotent of  $\CO T$  and   $\CO  Gb $  is isomorphic  as an $\CO$-algebra to  $\mathrm{Mat}_n (\CO  T f)$, where $ n= [G : T]$.     Since  $\Stab_G(f)  = \Stab_G (\,^\sigma f) $, we   also have   by the same argument that  
$\,^\sigma f $ is a block idempotent  of $\CO T $ and $\CO G \,^\sigma b $  is isomorphic  as an $\CO$-algebra  to  $\mathrm{Mat}_n (\CO T \,^\sigma f)$. Thus, it suffices to prove that  $\CO  T f  $  and  $\CO  T\,^\sigma f $  are  isomorphic  $\CO$-algebras.   Note that   since  $\,^\sigma e =e$  and $ e f \ne 0$, we also have that  $e \, ^\sigma f \ne 0 $.    
Let $\CX_0$ be the subgroup of  $\mathrm{Lin}(H/N)$  consisting  of  characters $\tau  $  such that $\tau. f =f$ and  let 
 $\CX $  be the subset  of  $\mathrm{Lin}(H/N) $  consisting of  those characters $\theta $   with the property  that $\theta. f =  \,^\sigma  f $.   By Lemma~\ref{lem:normalabelian},  $\CX \ne \emptyset  $  and  hence  $\CX $ is   a left coset  of $\CX_0 $  in $ \mathrm{Lin} (H/N)$.  Since $f$ and $\,^\sigma f $ are $T$-invariant, we also have that  $T/H$ acts by conjugation on   $\CX $.  Since   $|\CX |$  is relatively  prime to $\ell$  and $ T/H$  is  an $\ell$-group, there exists  $\eta \in \CX $  which is $T$-stable.    Again since  $|T/H|$  is relatively prime to  $|H/N|$, it follows that $\eta $  extends to a linear character, say $\tilde \eta $,    of $T/N$ (and this extension again takes values in $\CO$, since all elements of $\ell$-power order can be assumed to lie in its kernel). Then the   automorphism of  $\CO T $  induced by $\tilde \eta $   induces an  isomorphism between $\CO T f$ and  $\CO T \,^\sigma f$.

 Now suppose that $e$  is the principal block idempotent of $\CO N$  and let $T$ and $f$  be  as in the beginning of the proof.  By the argument above, it suffices  to show that  $\CO Tf$ is Morita equivalent to  $\CO T d$, where $d$ is the principal  block idempotent  of  $\CO T$. As earlier, let $\CX_0$ be the subgroup of  $\mathrm{Lin}(H/N)$  consisting  of  characters $\tau  $  such that $\tau. f =f$ and  let 
 $\CY $  be the subset  of  $\mathrm{Lin}(H/N) $  consisting of  those characters $\theta $   with the property  that $\theta. f = d $.   By Lemma~\ref{lem:normalabelian},  $\CY \ne \emptyset  $  and  hence  $\CY $ is   a left coset  of $\CX_0 $  in $ \mathrm{Lin} (H/N)$.  Since $f$ and $d $ are $T$-invariant, we also have that  $T/H$ acts by conjugation on   $\CY$.  Now  we argue as in the previous paragraph.
\end{proof}

\subsection{Some lemmas concerning bimodules inducing Morita equivalences}

Let $\Lambda \in \{K, \CO, k \}$. If $H$ is a common subgroup of groups $G$ and $L$, denote by $\Delta H $  the subgroup $\{(h, h^{-1}):  h \in H \}$
of $G \times L^{\opp}$. 

We recall  Marcus's result on extending equivalences \cite{Marcus}.
\begin{Lemma} \label{lem:marcus} Suppose that $ G_1 $ is  a finite group with normal subgroup $G$   and  $L_1, L $  are   subgroups of $G_1  $ with  $ L= G \cap   L_1 $.    Let $ b$   be $G_1$-stable central idempotent   of $\Lambda G$  and $c $ an $L_1$-stable   central idempotent   of $\Lambda L $ and let $M $ be an $\Lambda(G \times L^{\opp} ) (b \otimes c)$-module  inducing a Morita equivalence between $\Lambda  G b $ and $\Lambda Lc$.  Let  $ \CD$ be the  subgroup  
$$ (G  \times L) \Delta L_1=\{(x, y):  x\in G_1, y \in    L_1^{\opp} :    x G = y^{-1} G \}    \leq   G_1 \times  L_1^{\opp}. $$
If  the $\Lambda (G \times L^{\opp} )$-module  structure on  $M$ extends to  an $\Lambda  \CD  $-module  structure on $M$, then $\Ind_{\CD}^{G_1 \times L_1} M $ induces a Morita equivalence between  $\Lambda  G_1b $ and   $\Lambda L_1 c $.
\end{Lemma}

\begin{Lemma} \label{lem:centraldomination}   Let  $G$  and $H$ be  finite groups  and $b$ and  $c$  be central idempotents of  $\CO G $ and $ \CO H$ respectively. 
Let $M$ be an $\CO Gb  \otimes  (\CO  Hc)^{\opp}$-module  such that $M$ induces a Morita equivalence between $\CO Gb$ and $\CO Hc$. Let $Z$ be a group embedded as a central subgroup of both $G$  and $H$, let  $\bar  G = G/Z$, $\bar H =H/Z$  and let 
$\bar b$  (respectively $\bar c$) be the image of  $b$ (respectively $c$) under the canonical surjection $\CO G \to \CO G/Z$ (respectively $\CO H \to  \CO  H/Z $).

Suppose that  $\bar b \ne 0$ and $(z, z^{-1}).m=m$  for all  $z \in Z$ and $m \in M$.   Then  $\bar c\ne 0 $, $\CO  \otimes_{\CO  Z} M$  is  an  
$\CO \bar G\bar b  \otimes (\CO \bar H \bar c)^{\opp}$-module via  $(\bar g \otimes  \bar h)(1 \otimes m ) =  1\otimes (g \otimes h)m $ for all $g \in G, h\in H, m\in M$  and  $\CO \otimes_{\CO Z} M$  induces a Morita equivalence between $\CO \bar G \bar b$ and $\CO  \bar H \bar c$.
\end{Lemma}

\begin{proof}   This can be proved exactly as in  \cite[Lem.~2.9]{EKKS}. Note that in~\cite{EKKS} $Z$ is  assumed to be an $\ell$-group (so  $\bar b$ and $\bar c $ are both non-zero), but it can be checked that  this assumption  is not needed  in the proof  and that one obtains as a  byproduct that  if $\bar b \ne 0 $, then $\bar c \ne$0.\end{proof}

\section{Background on tame blocks}\label{sec:tamebackground}

We include here for convenience an outline of the classification of Morita equivalence
classes of $k$-blocks with quaternion defect group based on Erdmann's in~\cite{er90}, 
as well as further known results refining this classification and on Morita equivalence of 
$\CO$-blocks. We do not include quivers and relations here since in the case of one 
or three simple modules there are no undecided parameters involved in the relations, 
and the quivers and relations in the case of two simple modules are presented in 
Section \ref{socle scalars}. As before  we denote 
by $Q_{2^n}$ a quaternion group of order $2^n$, where $n\geq 3$.
Similarly, we denote by $D_{2^n}$ a dihedral group of order $2^n$, where
$n\geq 2$, with the convention that $D_4$ is a Klein four group. 

\begin{Proposition}
\thlabel{Erdmann_class_plus}
Let $G$ be a finite group and $kG b$ a block with defect group $Q_{2^n}$, 
where $n \geq 3$. Then $kGb$ is Morita equivalent to one of the following:
\begin{enumerate}[(i)]
\item $kQ_{2^n}$;
\item an algebra $Q(2A)^{2^{n-2},c}$ or $Q(2B)_1^{2,2^{n-2},c}$, where $c \in k$ and $n \geq 4$. Here $\ell(kGb)=2$;
\item an algebra $Q(3A)_2^{2^{n-2}}$, $Q(3B)^{2,4}$, or $Q(3K)^{2,2,2^{n-2}}$. Here $\ell(kGb)=3$.
\end{enumerate}

In particular, $l(kGb) \leq 3$, and $l(kGb)=1$ precisely when $kGb$ is nilpotent.

Further, in the cases where $l(kGb)=1,3$, for each Morita equivalence class of $k$-blocks, there is a corresponding unique equivalence class of $\CO$-blocks, and there is a representative of each Morita equivalence class that is a principal block.
\end{Proposition}

\begin{proof}
This is from~\cite{er90}. Note \thref{Q(2B)_parameters} concerning the possible parameters in type $Q(2B)_1$. By~\cite{Macgregor} the classes $Q(3B)^{2,2^{n-2}}$ do not occur when $n \geq 5$.

The conclusion regarding $\CO$-blocks follows from the theory of nilpotent blocks for $l(kGb)=1$, and by~\cite{Eisele16} for $l(kGb)=3$.
\end{proof}

\begin{Remark}\label{rem:mfbound}
    It follows by \thref{Erdmann_class_plus} that for a given $Q_{2^n}$ the Cartan matrices are known and the entries bounded. Hence by \thref{Donovan_and_Morita_Frob} in order to prove Donovan's conjecture for $Q_{2^n}$ it suffices to bound the Morita--Frobenius numbers.
\end{Remark}

The following fact, stating that blocks with a quaternion defect group
are closed under Morita equivalences, is well-known (and holds more
generally for derived equivalences and stable equivalences of Morita type);
we state here only what we need in this paper, and sketch
a proof for convenience.

\begin{Proposition} \thlabel{quat-Morita}
Let $G$, $H$ be finite groups and $b$, $c$ blocks of $kG$, $kH$,
respectively, such that $kGb$ and $kHc$ are Morita equivalent.
If the defect groups of $kGb$ are isomorphic to $Q_{2^n}$ for some
$n\geq 3$, then so are the defect groups of $kHc$.
\end{Proposition}

\begin{proof}
By \cite[Ch.~5, Thm.~11.6]{NT} the orders of the defect groups
of Morita equivalent blocks are equal since they coincide with the
largest elementary divisors of their Cartan matrices. By
\cite[Cor.~4.3]{Linvar}, the defect groups of Morita equivalent
blocks have the same rank, hence rank $1$ in this case. The defect
groups of $kHc$ cannot be cyclic since the representation type
of $kGb$, hence of $kHc$, is not finite. The result follows.
\end{proof}

We will use the following frequently without reference.

\begin{Lemma}
\thlabel{normal_subgroup_quaternion}
Let $P \cong Q_{2^n}$, where $n \geq 3$. If $Q \leq P$, then either $Q$ is quaternion and satisfies $C_p(Q)\leq Q$, or $Q$ is cyclic. If further $Q \lhd P$ is a noncyclic proper normal subgroup of $P$ (so $n \geq 4$), then $Q \cong Q_{2^{n-1}}$.

Similarly when $P \cong D_{2^n}$: subgroups are either dihedral, in which case $C_P(Q) \leq Q$, or cyclic; normal noncyclic subgroups are isomorphic to $D_{2^{n-1}}$ and are one of two such subgroups. 
\end{Lemma}

\begin{proof}
Write $P=\langle x,y|x^{2^{n-1}}=1, x^{2^{n-2}}=y^2, yx=x^{-1}y \rangle$. Since $P$ has a unique involution, the same is true for subgroups of $P$ and so $Q$ must be cyclic or quaternion. 

Now let $Q \lhd P$ be a proper subgroup and suppose that $Q$ is not cyclic. Then $N \not\leq \langle x \rangle$, so there is some $x^iy^j \in Q$ with $j=1,3$. Without loss of generality we may assume $j=1$. Then $x^2=[x^iy,x^{-1}] \in Q$, so that $Q=\langle x^2,y\rangle$ or $\langle x^2,xy\rangle$. In either case $[P:Q]=2$ and $C_P(Q) \leq Q$. It follows that $C_P(Q) \leq Q$ for any noncyclic subgroup $Q$ of $P$. 

The result for dihedral groups is a consequence of the fact that $D_{2^n} \cong Q_{2^{n+1}}/Z(Q_{2^{n+1}})$. 
\end{proof}

\begin{Corollary}
Let $\CO Gb$ be a block with quaternion defect group $D$ and $N \lhd G$ with $G=ND$ and $[G:N]=2^t$ for $t \geq 2$. Then $\CO Gb$ is nilpotent.      
\end{Corollary}

\begin{proof}
Since $G/N$ is a $2$-group, the condition that $G=ND$ is equivalent to $\CO Gb$ covering a (unique) $G$-stable block $\CO Ne$ of $N$. Now $\CO Ne$ has defect group $D \cap N \lhd D$, which by \thref{normal_subgroup_quaternion} is cyclic, so $\CO Ne$ is nilpotent, and by~\cite[Thm.~8.12.1]{LiBookII} so is $\CO Gb$.
\end{proof}

\begin{Lemma}
\thlabel{G_and_C_G(z)}
Let $kGb$ be a block with quaternion defect group $D$, and let $Z(D)=\langle z \rangle$. Let $kC_G(z)b_z$ be a block with Brauer correspondent $kGb$.  Then $\ell(kGb)=\ell(kC_G(z)b_z)$. In particular $kC_G(z)b_z$ is nilpotent if and only if $kGb$ is.    
\end{Lemma}

\begin{proof}
By~\cite[Prop.~3.2]{ol75} the multiplicity of $2$ as an elementary divisor of the Cartan matrix is equal to $\ell(kC_G(z)b_z)-1$. The Cartan matrices of the algebras in the classes $Q(2 \mathcal{A})$, $Q(2\mathcal{B})_1$, $Q(3\mathcal{A})_2$, $Q(3\mathcal{B})$, $Q(3\mathcal{K})$, which represent all of the possibilities for $kGb$, are given in~\cite{er90}. Computing the elementary divisors we see that $\ell(kGb)=\ell(kC_G(z)b_z)$. The result follows since a block with defect group $D$ is nilpotent if and only if it has an unique simple module.

More intuitively, we could also observe that since $z$ is the unique involution in $D$, then $kGb$ and $kC_G(z)b_z$ induce the same fusion system on $D$.
\end{proof}

\begin{Remark}[Donovan's conjecture over $\CO$ for other types of tame blocks]\thlabel{Remark: Donovan over O for dihedral and semidihedral}
Donovan's conjecture holds for $\CO$-blocks of dihedral defect by known results. This is not explicitly stated in the literature, so we give a proof below. Donovan's conjecture for $\CO$-blocks of semidihedral defect appears to be open. Results of Lange \cite{LangeDiss} imply that Donovan's conjecture holds for $\CO$-blocks of semidihedral defect with three simple modules (see \cite[Thm.~3.2.21]{LangeDiss}, and use the same reasoning we use for blocks of dihedral defect below). However, Lange's corresponding results for blocks with two simple modules (\cite[Thms.~3.2.11,~3.2.14]{LangeDiss}) suggest that a proof in those cases will likely require group-theoretic methods. 

Let us now consider the case of dihedral defect.
Any block of dihedral defect of order $2^n$ over $k$ is Morita-equivalent to $kD_{2^n}$, $D(2A)^{2^{n-2},0}$, $D(2B)^{2^{n-2},0}$, $D(3A)_1^{2^{n-2}}$, $D(3B)_1^{1,2^{n-2}}$ or $D(3K)^{1,1,2^{n-2}}$, using the notation of \cite{Erdmann88} (except that we explicitly specify all parameters in the superscript). See also \cite[Thm. on~p.~180]{Holm97}. The fact that $D(2A)^{2^{n-2},1}$ and $D(2B)^{2^{n-2},1}$ do not occur is from \cite[Cor.~6.9]{Eisele12}. Note that these considerations give an upper bound $c(D_{2^n})$ on the Cartan numbers as required by \thref{Donovan_and_Morita_Frob}. To apply \thref{Donovan_and_Morita_Frob}, we still need an upper bound $m(D_{2^n})$ on $\CO$-Morita--Frobenius numbers.

Any $\CO$-block of dihedral defect of order $2^n$ with one simple module is nilpotent, and therefore isomorphic to $\CO D_{2^n}$. Any $\CO$-block of dihedral defect of order $2^n$ with two or three simple modules is an $\CO$-order whose basic algebra $\Gamma$ satisfies the conditions of either \cite[Thm.~6.8]{Eisele12}, \cite[Cor.~6.10]{Eisele12}, \cite[Thm.~4.15]{EiseleDiss} or \cite[Cor.~4.16]{EiseleDiss}  (depending on the Morita equivalence class of the corresponding block over $k$). In particular, if $\Gamma'$ is the basic algebra of another such block and there is an isomorphism between $Z(\Gamma)$ and $Z(\Gamma')$ which preserves decomposition matrices and character heights (these correspond to the symmetrising form in \cite{Eisele12}), then $\Gamma\cong \Gamma'$. Write $A$ for the $K$-span of $\Gamma$. We can choose a suitable lift $\sigma:\ K \longrightarrow K$ of the Frobenius automorphism of $k$ as in \cite{eel20}, extend $\sigma$ to a ring automorphism of $A$, and then construct the Frobenius twists of $\Gamma$ as $\sigma(\Gamma)$, $\sigma^2(\Gamma)$, and so forth. These all have $K$-span equal to $A$, and they have the same decomposition matrix and symmetrising form as $\Gamma$. Since the centre of an $\CO$-block of a finite group has Morita--Frobenius number $1$, the orders $Z(\Gamma)$, $Z(\sigma(\Gamma))$, $Z(\sigma^2(\Gamma))$, etc. are all isomorphic, and therefore differ by an automorphism of $Z(A)$. Being a commutative semisimple algebra, the automorphism group of $Z(A)$ is finite. 
It follows that $Z(\Gamma)=Z(\sigma^i(\Gamma))$ for some $1\leq i\leq |\Aut_K(Z(A))|$. But then \cite[Thm.~6.8]{Eisele12}, \cite[Cor.~6.10]{Eisele12}, \cite[Thm.~4.15]{EiseleDiss} or \cite[Cor.~4.16]{EiseleDiss} implies that $\Gamma \cong \sigma^i(\Gamma)$. That is, the Morita--Frobenius number of an $\CO$-block of dihedral defect of order $2^n$ with two or three simple modules is bounded by $m(D_{2^n})=|\Aut_K(Z(A))|$. We can now apply \thref{Donovan_and_Morita_Frob} to conclude that there are only finitely many Morita equivalence classes of $\CO$-blocks with defect group $D_{2^n}$. It is likely that this argument can be strengthened  using the results of Cabanes and Picaronny \cite{CabanesPicaronny} to give a classification theorem similar to \cite[Thm.~1.1]{Eisele16}.
\end{Remark}

\section{Erdmann's classification and socle scalars}
\label{socle scalars}

In this section we consider the unknown scalars $c$ in the quiver and relation representations of $Q(2A)^{s,c}$ and $Q(2B)_1^{s,t,c}$, and show that if the Morita--Frobenius number is $1$, then $c \in \{ 0,1\}$.

\begin{Definition}
  \begin{enumerate}
  \item For $2\leq s \in \mathbb N$ and $c\in k$ let $Q(2A)^{s,c}$  be the bound path algebra with quiver 
  $$
    \xymatrix{
    \bullet \ar@<0.6ex>[r]^{\beta} \ar@(ul,dl)[]_{\alpha}
    &
    \bullet \ar@<0.6ex>[l]^{\gamma}
    }
  $$
  with relations
  $$
    \begin{array}{rcl}
    \gamma\beta\gamma &=& (\gamma\alpha\beta)^{s-1}\gamma\alpha \\
    \beta\gamma\beta &=& (\alpha\beta\gamma)^{s-1}\alpha\beta \\
    \alpha^2 &=& (\beta\gamma\alpha)^{s-1}\beta\gamma + c(\beta\gamma\alpha)^s \\
    \alpha^2\beta &=& 0
    \end{array}
  $$
  \item For $2\leq s \in \mathbb N$, $3\leq t \in \mathbb N$ and $c\in k$ let $Q(2B)_1^{s,t,c}$  be the bound path algebra with quiver 
  $$
    \xymatrix{
    \bullet \ar@<0.6ex>[r]^{\beta} \ar@(ul,dl)[]_{\alpha}
    &
    \bullet \ar@<0.6ex>[l]^{\gamma} \ar@(ur,dr)[]^{\eta}
    }
  $$
  with relations
  $$
    \begin{array}{rcl}
    \gamma\beta&=&\eta^{t-1} \\
    \beta\eta&=&(\alpha\beta\gamma)^{s-1}\alpha\beta \\
    \eta\gamma&=&(\gamma\alpha\beta)^{s-1}\gamma\alpha\\
    \alpha^2 &=& (\beta\gamma\alpha)^{s-1}\beta\gamma + c(\beta\gamma\alpha)^s \\
    \alpha^2\beta &=& 0
    \end{array}
  $$
  \end{enumerate}
\end{Definition}

\begin{Remark}
\thlabel{Q(2B)_parameters}
  Note that we follow the notation of \cite{Erdmann88}, except that we renamed the parameters $k$ and $s$ to $s$ and $t$ to avoid clashing notation. Note also that several places in the literature define a version of the algebra $Q(2B)_1^{s,t,c}$ involving an addition parameter $0\neq d\in k$, where the relation $\alpha^2 = (\beta\gamma\alpha)^{s-1}\beta\gamma + c(\beta\gamma\alpha)^s$ is replaced by $\alpha^2 = d(\beta\gamma\alpha)^{s-1}\beta\gamma + c(\beta\gamma\alpha)^s$. The resulting algebra, which we may call $Q(2B)_1^{s,t,c,d}$, is isomorphic to $Q(2B)_1^{s,t,\tilde c}$ for a suitable $\tilde c\in k$. That is, the additional parameter $d$ is redundant.
\end{Remark}

\begin{Proposition}\label{prop:Q2A}
  Let $n\geq 4$ be a natural number and let $c\in k$ be arbitrary. Then the algebra $Q(2A)^{2^{n-2},c}$ has Morita--Frobenius number $1$ if and only if $Q(2A)^{2^{n-2},c} \cong Q(2A)^{2^{n-2},\tilde c}$ with $\tilde c\in \mathbb F_2$.
\end{Proposition}
\begin{proof}
  We will show that for any $c,\tilde c\in k$ there is an isomorphism $Q(2A)^{2^{n-2},c} \cong Q(2A)^{2^{n-2},\tilde c}$ if and only if either $c=\tilde c=0$ or both $c$ and $\tilde c$ are non-zero and $\frac{\tilde c}{c}$ is a $(5\cdot 2^{n-3}-3)$-rd root of unity. Given any $c\neq 0$, the latter part of the assertion means that $Q(2A)^{2^{n-2},c^2} \cong Q(2A)^{2^{n-2},c}$, i.e. $Q(2A)^{2^{n-2},c}$ has Morita--Frobenius number $1$, if and only if $Q(2A)^{2^{n-2},c} \cong Q(2A)^{2^{n-2},1}$. This clearly implies the claim. 

  We prove the ``only if'' direction first. Set $A=Q(2A)^{2^{n-2},c}$ and $\tilde A=Q(2A)^{2^{n-2},\tilde c}$. Assume that $\varphi:\ \tilde A \longrightarrow A$ is an isomorphism of $k$-algebras. To avoid confusion, we denote the generators of $A$ by $e_0$, $e_1$, $\alpha$, $\beta$ and $\gamma$ and those of $\tilde A$ by $\tilde e_0$, $\tilde e_1$, $\tilde \alpha$, $\tilde \beta$ and $\tilde \gamma$. We can assume without loss of generality that $\varphi$ maps $e_i$ to $\tilde e_i$ for $i\in\{0,1\}$, since the Cartan numbers associated to $e_0$ and $e_1$ are different. 
  We define $s=2^{n-2}$ (Erdmann calls this ``$k$'', but this conflicts with our notation), $z=\alpha\beta\gamma+\beta\gamma\alpha+\gamma\alpha\beta$ and $Z=k[z]$.  One checks that $Z\subseteq Z(A)$. Let $\bar{\phantom{f}}: Z\longrightarrow k$ denote the $k$-algebra homomorphism that maps an element of $Z$ to its constant coefficient. 
  
  The relations of $A$ imply that $A$ has a basis consisting of paths in which an $\alpha$ can only be followed by a $\beta$, a $\beta$ can only be followed by a $\gamma$ and a $\gamma$ can only be followed by an $\alpha$. Any path that does not satisfy this lies in the second socle of $A$, and every path of length strictly greater than $3s$ equals zero in $A$. To be more precise, we will use that
  $e_0Ae_1$ has basis $\{ (\beta\gamma\alpha)^i \beta, (\alpha\beta\gamma)^i \alpha\beta \ | \ i=0,\ldots,s-1\}$, 
   $e_1Ae_0$ has basis $\{ (\gamma\alpha\beta)^i \gamma, (\gamma\alpha\beta)^i \gamma\alpha \ | \ i=0,\ldots,s-1\}$ and 
   $e_0Ae_0$ has basis $\{(\alpha\beta\gamma)^i, (\beta\gamma\alpha)^{i+1}, (\alpha\beta\gamma)^i\alpha, (\beta\gamma\alpha)^i\beta\gamma \ | \ i=0,\ldots,s-1\}$. Note that $(\alpha\beta\gamma)^s=(\beta\gamma\alpha)^s$. Also note that the only non-zero path in $A$ having $\alpha^2$ as a proper subpath is $\alpha^3$.

  Given the above, we can assume that
  $$
    \varphi(\tilde \beta)=b\beta + y_1 \alpha\beta \quad\textrm{with $b\in Z^\times$, $y_1\in Z$},
  $$
  $$
    \varphi(\tilde \gamma)=g\gamma + y_2 \gamma\alpha \quad\textrm{with $g\in Z^\times$, $y_2\in Z$},\textrm{ and}
  $$
  $$
    \varphi(\tilde \alpha)=a\alpha + y_3 \beta\gamma+y_4 \alpha\beta\gamma + y_5 \beta\gamma\alpha \quad\textrm{with $a\in Z^\times$, $y_3,y_4,y_5\in Z$}.
  $$
  We get 
  $$
    \varphi(\tilde\beta)\varphi(\tilde\gamma)\varphi(\tilde\beta) = b^2g \beta\gamma\beta + b(by_2+gy_1)\beta\gamma\alpha\beta+y_1(gy_1+by_2)\alpha\beta\gamma\alpha\beta
  $$
  whereas 
  $$
    \varphi(\tilde\beta\tilde\gamma\tilde\beta)= \varphi((\tilde \alpha\tilde\beta\tilde\gamma)^{s-1}\tilde\alpha\tilde\beta)= a^sb^sg^{s-1} (\alpha\beta\gamma)^{s-1}\alpha\beta.
  $$
  It follows that $by_2+gy_1$ annihilates $\beta\gamma\alpha\beta$, which means $by_2+gy_1\in z^{s-1}Z$. But that means $by_2+gy_1$ annihilates $\alpha\beta\gamma\alpha\beta$ as well. It follows that $\varphi(\tilde\beta)\varphi(\tilde\gamma)\varphi(\tilde\beta) = b^2g \beta\gamma\beta$, and therefore 
  \begin{equation} \label{eqn:rels for hom 1}
    \bar b^2 \bar g= \bar a^s \bar b^s \bar g^{s-1}.
  \end{equation} 
  Now compare
  $$
    \varphi(\tilde\alpha)^2 = a^2\alpha^2 + ay_3 \alpha\beta\gamma + ay_5\alpha\beta\gamma\alpha +ay_3\beta\gamma\alpha + y_3y_4 \beta\gamma\alpha\beta\gamma + ay_4 \alpha\beta\gamma\alpha + y_4^2 (\alpha\beta\gamma)^2 + y_3y_5 \beta\gamma\alpha\beta\gamma + y_5^2 (\beta\gamma\alpha)^2 $$
  and 
  $$ \varphi(\tilde \alpha^2)= b^sg^sa^{s-1}(\beta\gamma\alpha)^{s-1}\beta\gamma+\tilde c b^sg^sa^s(\beta\gamma\alpha)^{s},$$
  where we again used $by_2+gy_1\in z^{s-1}Z$ to obtain the latter equation. It follows that $ay_3 \alpha\beta\gamma$ is a scalar multiple of $(\alpha\beta\gamma)^s$, which forces $y_3\in z^{s-1}Z$. But then $ay_3 \alpha\beta\gamma = ay_3\beta\gamma\alpha$ since $(\alpha\beta\gamma)^s=(\beta\gamma\alpha)^s$. That is, $ay_3 \alpha\beta\gamma + ay_3\beta\gamma\alpha =0$ as $k$ has characteristic $2$. Similarly, we see that $a(y_4+y_5)\alpha\beta\gamma\alpha$ must be zero, which implies $y_4+y_5\in z^{s-1}Z$. But then $y_3(y_4+y_5)\beta\gamma\alpha\beta\gamma=0$ and $(y_4+y_5)^2(\beta\gamma\alpha)^2=0$, so 
  $$\varphi(\tilde\alpha)^2 = a^2\alpha^2 = a^2(\beta\gamma\alpha)^{s-1}\beta\gamma+a^2c(\beta\gamma\alpha)^{s}.$$ 
  This gives us 
  \begin{equation} \label{eqn:rels for hom 2}
    \bar a^2 = \bar b^s \bar g^s\bar a^{s-1} \quad \textrm{and} \quad \bar a^2 c= \tilde c \bar b^s \bar g^s \bar a^s.
  \end{equation}
  The last equation implies that $c$ and $\tilde c$ are either both zero or both non-zero. If they are both zero, then we are done. So assume that $c$ and $\tilde c$ are both non-zero. 
  In that case, equations~\eqref{eqn:rels for hom 1}~and~\eqref{eqn:rels for hom 2} imply 
  $$
    \frac{c}{\tilde c} = \bar a \quad\textrm{and} \quad \bar b \bar g =\bar a^{\frac{3-s}{s}} \quad \textrm{and} \quad \bar a^{\frac{5s-6}{s}}=1.
  $$
  Note that $s$ is a power of two, so $s$-th roots are unique in $k$. In particular, $\bar a$ and therefore $\frac{c}{\tilde c}$ is a $(5s-6)$-th root of unity and therefore also a $(5\cdot 2^{n-3}-3)$-rd root of unity, as claimed.

  For the converse assume that $\frac{c}{\tilde c}$ is a $(5\cdot 2^{n-3}-3)$-rd root of unity (there is nothing to show in the case $c=\tilde c=0$). Define $a=\frac{c}{\tilde c}$, $b=a^{\frac{3-s}{s}}$ and $g=1$. Then one checks that the assignment
  $$
    \tilde \alpha \mapsto a \alpha, \ \tilde \beta \mapsto b \beta, \ \tilde \gamma \mapsto g \gamma
  $$
  extends to a $k$-algebra isomorphism between $\tilde A$ and $A$, proving the claim.
\end{proof}

\begin{Proposition}\label{prop:Q2B}
  Let $n\geq 4$ be a natural number and let $c\in k$ be arbitrary. Then the algebra $Q(2B)_1^{2,2^{n-2},c}$ has Morita--Frobenius number $1$ if and only if $Q(2B)_1^{2,2^{n-2},c} \cong Q(2B)_1^{2,2^{n-2},\tilde c}$ with $\tilde c\in \mathbb F_2$.
\end{Proposition}
\begin{proof}
    We prove this the same way as Proposition~\ref{prop:Q2A}. Set $A=Q(2B)_1^{2,2^{n-2},c}$ and $\tilde A = Q(2B)_1^{2,2^{n-2},\tilde c}$. We can again set $z=\alpha\beta\gamma+\beta\gamma\alpha+\gamma\alpha\beta$ and $Z=k[z]$, and $Z$ is again central in $A$. A basis of $e_0 A e_1$ is given by $\{\ z^i\beta, \ z^i\alpha\beta \ | \ i=0,1 \}$, a basis of $e_1 A e_0$ is given by 
    $\{ z^i\gamma, \ z^i\gamma\alpha \ | \ i=0,1 \}$, a basis of 
    $e_0 A e_0$ is given by $\{e_0, \alpha\beta\gamma, z^i\beta\gamma\alpha, z^i\alpha, z^i\beta\gamma \ | \ i=0,1\}$ and a basis of 
    $e_1 A e_1$ is given by $\{e_1, \gamma\alpha\beta, \eta^i \ | \ i=1,\ldots,s\}$. 
      We will extend the ``$\bar{\phantom f}$''-notation to $k[\eta]$. We get 
      $$
        \varphi(\tilde \beta)=b\beta + y_1 \alpha\beta \quad\textrm{with $b\in Z^\times$, $y_1\in Z$},
      $$
      $$
        \varphi(\tilde \gamma)=g\gamma + y_2 \gamma\alpha \quad\textrm{with $g\in Z^\times$, $y_2\in Z$},
      $$
      $$
        \varphi(\tilde \alpha)=a\alpha + y_3 \beta\gamma+y_4 \alpha\beta\gamma + y_5 \beta\gamma\alpha \quad\textrm{with $a\in Z^\times$, $y_3,y_4,y_5\in Z$}\textrm{, and}
      $$
      $$
        \varphi(\tilde \eta) = e \eta + y_6 \gamma\alpha\beta \quad\textrm{with $e\in k[\eta]^\times$, $y_6\in k$.}
      $$

      Comparing $\varphi(\tilde \eta \tilde \gamma)$ and $\varphi(\tilde \eta)\varphi(\tilde \gamma)$ gives $y_6 \gamma\alpha\beta=0$, since $\gamma\alpha\beta\gamma$ has coefficient $0$ in $\varphi(\tilde \eta \tilde \gamma)$. It follows that $\varphi(\tilde \beta)\varphi(\tilde \eta)= be\beta\eta$, which we can compare to $\varphi(\tilde \beta \tilde \eta)= a^2b^2g\alpha\beta\gamma\alpha\beta$. We conclude $\bar a^2\bar g\bar b=\bar e$.
      
      We now compare
      $$
        \varphi(\tilde \gamma)\varphi(\tilde\beta)=gb\gamma\beta+(gy_1+by_2)\gamma\alpha\beta \textrm{ and } \varphi(\tilde \gamma\tilde\beta)=e^{s-1}\eta^{s-1},
      $$
      and conclude that $gy_1+by_2$ annihilates $\gamma\alpha\beta$, which means $gy_1+by_2\in z^2k[z]$. We also conclude $\bar g\bar b= \bar e^{s-1}$, which gives us $\bar a^2 = \bar e^{2-s}$.

      Lastly, consider 
      $$
        \varphi(\tilde \alpha)^2 = a^2 \alpha^2 + ay_3(\alpha\beta\gamma+\beta\gamma\alpha)+a(y_4+y_5)\alpha\beta\gamma\alpha + y_3(y_4+y_5)\beta\gamma\alpha\beta\gamma + (y_4+y_5)^2(\alpha\beta\gamma)^2
      $$
      and 
      $$
        \varphi(\tilde\alpha^2)= b^2g^2a \beta\gamma\alpha\beta\gamma + b^2g^2a^2 \tilde c(\alpha\beta\gamma)^2,
      $$
      where we have used $gy_1+by_2\in z^2k[z]$. Just as in the proof of Proposition~\ref{prop:Q2A}, most terms in the expression for $\varphi(\tilde \alpha)^2$ must vanish, leaving us with $\varphi(\tilde \alpha)^2 = a^2 \alpha^2$.
      We conclude $\bar a=\bar b^2\bar g^2$ and $c=\bar a \tilde c$. This yields $\bar a=\bar e^{2s-2}$, and therefore $\bar a^{2-s}=\bar e^{2(s-1)(2-s)}=\bar a^{4(s-1)}$, which is equivalent to $\bar a^{5s-6}=1$. It follows that  if $A\cong\tilde A$, then either $\frac{c}{\tilde c}$ is a $(5s-6)$-th root of unity or $c=\tilde c=0$. 

      Conversely, if  $\frac{c}{\tilde c}$ is a $(5s-6)$-th root of unity, we can just define a $\varphi$ as above by setting $y_1,\ldots,y_6=0$, $a=\frac{c}{\tilde c}$, $e=a^{3/s}$, $b=1$ and $g=a^{1/2}$. One checks that this defines an isomorphism of $k$-algebras.
\end{proof}

\begin{Remark}
  Note that the proofs of Propositions~\ref{prop:Q2A}~and~\ref{prop:Q2B} classify the algebras in the families $Q(2A)^{2^{n-2},c}$ and $Q(2B)_1^{2,2^{n-2},c}$ for $n\geq 4$ up to Morita equivalence in terms of the socle scalar $c$. Namely, two socle scalars $c$ and $\tilde c$ give rise to Morita equivalent algebras if and only if $c/\tilde c$ is a $(5\cdot 2^{n-3}-3)$-rd root of unity.
\end{Remark}  

\begin{Corollary} \label{scalar}
    Let $B$ be a block of quaternion defect $Q_{2^n}$ with two simple modules and Morita--Frobenius number $1$. Then $B$ is Morita equivalent to either $Q(2A)^{2^{n-2},c}$ or $Q(2B)_1^{2,2^{n-2},c}$ with $c\in\{0,1\}$.
\end{Corollary}
\begin{proof}
    As is remarked explicitly in the beginning of \cite[Sec.~4.1]{Holm97}, any block of quaternion defect with two simple modules is Morita equivalent to either $Q(2A)^{2^{n-2},c}$ or $Q(2B)_1^{2,2^{n-2},c}$ for some $c\in k$. It is also well-known that in this case $n\geq 4$. The claim now follows directly from Propositions~\ref{prop:Q2A}~and~\ref{prop:Q2B}.
\end{proof}

\section{Reduction theorems for blocks with quaternion defect groups}\label{sec:redthms}

The purpose in this section is to reduce the study of Morita equivalence classes of blocks with quaternion defect groups to quasisimple groups and groups containing a quasisimple subgroup with index two.

\begin{Proposition}
\thlabel{prop cover nilpotent}
    Let $G$ be a finite group with a normal subgroup $N$  such that $G/N$ is solvable. If $\CO G b$ is a block of $\CO G$ covering a $G$-stable nilpotent block of $\CO N$, then $\CO G b$ is Morita equivalent to a block of a solvable group. 
\end{Proposition}
\begin{proof}
  This is a consequence of \cite[Prop.~2.2]{EatonDefectThree}.
\end{proof}

\begin{Proposition}
\thlabel{2-nilpotent_quotient}
Let $G$ be a finite group and $\CO G b$ be a block of $\CO G$ with defect group $D \cong Q_{2^n}$ for some $n \geq 3$. Suppose $M \unlhd G$ with $G/M$ supersolvable and for every $M \leq L \unlhd G$, the block $\CO Gb$ covers a (unique) $G$-stable block of $\CO L$. Then either $[G:M]_2=[MD:M]$ divides $2$ and there is a block idempotent $c\in Z(\CO MD)$ such that $\CO Gb$ is Morita equivalent to $\CO MDc$, or $\CO G b$ is Morita equivalent to a block of a solvable group. In either case the Morita equivalent block has defect group $Q_{2^n}$.

\end{Proposition}

\begin{proof}
Let $b_0 \in Z(\CO M)$ denote a block idempotent covered by $b$, which is unique by assumption. Then $\CO Mb_0$ has a defect group $G$-conjugate to $D\cap M$. By \thref{normal_subgroup_quaternion} applied to $D \cap M \unlhd D$ we have one of: (i) $D \cap M$ is cyclic; (ii) $D \leq M$; (iii) $D \cap M \cong Q_{2^{n-1}}$. In case (i) $\CO Mb_0$ is nilpotent (since cyclic defect blocks in characteristic $2$ are nilpotent), so by
\thref{prop cover nilpotent} $\CO Gb$ is Morita equivalent to a block of a solvable group.

Hence we may suppose that $D \leq M$ or $D \cap M \cong Q_{2^{n-1}}$. Now by~\cite[Lem.~2.4]{ar19} we have $[DM:M]=[G:M]_2=1$ or $2$ and so $G/M$ is $2$-nilpotent. Write $N$ for the preimage of $O_{2'}(G/M)$ in $G$, so $[G:N]=1$ or $2$ and $G=ND$.

We may choose a filtration $M=N_0\leq N_1 \leq \ldots\leq N_r = N$, where each $N_i$ is normal in $G$ and $N_i/N_{i-1}$ is cyclic of odd prime order. For each $i$ we let $b_i\in Z(\CO N_i)$ denote the unique block idempotent covered by $b$. Then $\CO N_ib_i$ has a defect group $G$-conjugate to $D\cap N_i=D\cap M$. Again recall that $\CO N_ib_i$ has one, two or three simple modules (up to isomorphism), with one simple module precisely when it is nilpotent. If any $\CO N_ib_i$ is nilpotent, then \thref{prop cover nilpotent} implies that $\CO Gb$ is Morita equivalent to a block of a solvable group. So we may assume that each $\CO N_ib_i$ has two or three simple modules. Finally, note that by \thref{remark:fongcorrespondent} it follows that each $b_i$ is also a block idempotent of $\CO N_i D$.
  
Now fix $0< i \leq r$, and assume by way of induction that $\CO G b$ is Morita equivalent to $\CO N_i Db_i$, which we noted above is a block. Label the simple modules of $\CO N_i b_i$ by $S_1,\ldots, S_t$, where $t\in \{2,3\}$. Note that $N_i/N_{i-1}$ acts on the isomorphism classes of simple $\CO N_{i-1}b_{i-1}$-modules, and that there are either two or three such simple modules. Since $[N_i : N_{i-1}]$ is an odd prime, either $N_i/N_{i-1}$ stabilises all isomorphism classes of simple $\CO N_{i-1}b_{i-1}$-modules, or it permutes the isomorphism classes of simple $\CO N_{i-1}b_{i-1}$-modules transitively.
 
In the latter case the induced module $\Ind_{N_{i-1}}^{N_i} T$ is simple for all simple $\CO N_{i-1}b_{i-1}$-modules $T$. Moreover any two such simple induced modules are isomorphic, since Frobenius reciprocity guarantees the existence of non-zero homomorphisms between them. Any simple $\CO N_i b_i$-module arises as a composition factor of some $\Ind_{N_{i-1}}^{N_i} T$, so $\CO N_i b_i$ has just one simple module, a contradiction. Hence $N_i/N_{i-1}$ stabilises all isomorphism classes of simple $\CO N_{i-1}b_{i-1}$-modules.

Now assume that $\CO N_{i-1} b_{i-1}$ has $t'$ isomorphism classes of simple modules, where $t'\in\{2,3\}$, and call them $T_1, \ldots, T_{t'}$.
Since $N_i/N_{i-1}$ is cyclic of odd order we can extend each $T_j$ to a simple $\CO N_i$-module~$\tilde T_j$ (see \cite[Thm.~8.12]{N98}). It follows that $\Ind_{N_i}^{N_{i-1}} T_j \cong \bigoplus_{\eta \in \Irr(N_i/N_{i-1})} \CO_\eta \otimes _\CO \tilde T_j$, where $\CO_{\eta}$ denotes $\CO$ with $N_i$ acting via $\eta$. As mentioned above, any simple $\CO N_ib_i$-module occurs as a composition factor of some $\Ind_{N_i}^{N_{i-1}} T_j$, which means that $S_1,\ldots,S_t$ are all of the form $\CO_\eta \otimes _\CO \tilde T_j$ for suitable $\eta$ and $j$.  Note that $\Ext^1(\tilde T_j, \Ind_{N_{i-1}}^{N_i}T_l)\cong \Ext^1(T_j,T_l)$ for all $1\leq j,l\leq t'$. So we can assume without loss of generality that, for all $1\leq j \leq t'$, $\Ext^1(\tilde T_j, \tilde T_l)$ is non-zero for at least one $l\neq j$ such that $\Ext^1(T_j, T_l)$ is non-zero (this is by choosing the $\tilde T_j$ differently). In particular, given that we have $t'\leq 3$, we can assume that the $\tilde T_j$ all lie in the same block. It then follows that the $\CO_\eta \otimes_\CO \tilde T_j$ all lie in the same block for any fixed $\eta\in \Irr(N_i/N_{i-1})$. This block must be equal to $\CO N_i b_i$ for exactly one choice of $\eta$, since otherwise $\CO N_i b_i$ 
would have at least $2t'\geq 4$ isomorphism classes of simple modules (to be more precise, it would follow that the number of simple modules is $[N_i:N_{i-1}]t'$). In particular, $t=t'$ and we can  assume without loss of generality that $\eta=1$ and $S_j\cong \tilde T_j$ for all $j$. This means that $\Res^{N_i}_{N_{i-1}} S_j\cong T_j$ and $b_i\Ind_{N_{i-1}}^{N_i} T_j\cong S_j$ for all $j$. 
  
It follows that $$b_{i-1}\CO N_i b_i \otimes_{\CO N_{i-1} b_{i-1}} T_j \cong T_j \cong \CO N_{i-1} b_{i-1} \otimes_{\CO N_{i-1} b_{i-1}} T_j$$ as left $\CO N_{i-1}b_{i-1}$-modules (where we have used $b_{i-1} b_i=b_i$, as otherwise the tensor product on the left would be ill-defined). Since $b_{i-1}\CO N_ib_i$ is clearly projective as a right $\CO N_{i-1}$-module, it follows from the above isomorphisms that $b_{i-1}\CO N_ib_i\cong \CO N_{i-1}b_{i-1}$ as right $\CO N_{i-1}b_{i-1}$-modules. This implies 
$$b_{i-1}\CO N_i b_{i}= \CO N_{i-1}b_{i}\textrm{ and } \CO N_{i-1} b_{i-1}\stackrel{\sim}{\longrightarrow} \CO N_{i-1}b_i: \ x \mapsto xb_i.$$
Given that $b_i$ and $b_{i-1}$ commute with $D$ it follows that
  $$b_{i-1}\CO N_iD b_i = \CO N_{i-1}D b_i.$$
  We also get a surjective ring homomorphism 
  $$
  \CO N_{i-1} D b_{i-1}\stackrel{}{\longrightarrow} \CO N_{i-1} D b_i: \ x \mapsto xb_i,
  $$
  which must be an isomorphism since domain and range are free of rank two over $\CO N_{i-1} b_{i-1}$ and $\CO N_{i-1}b_i$, respectively, ensuring that both have the same rank over $\CO$. 
  In conclusion, we get an isomorphism
  $$
  \CO N_{i-1} D b_{i-1}\stackrel{\sim}{\longrightarrow} b_{i-1}\CO N_{i} D b_i =b_{i-1}\CO N_{i} D b_i b_{i-1}:\ x \mapsto xb_i 
  $$
  which shows that $\CO N_{i} D b_i$ and $\CO N_{i-1}Db_{i-1}$ are Morita equivalent (this uses that $b_{i-1}$ does not annihilate any simple $\CO N_i b_i$-module, and therefore it does not annihilate any simple $\CO N_{i} D b_i$-module either). This completes the inductive step.

  The final part follows by \thref{quat-Morita}.
\end{proof}

\begin{Lemma}
\thlabel{Out_sub_supersolvable}
Let $L$ be a nonabelian finite simple group. Let $H \leq \Out(L)$ such that $|H|_2$ is $1$ or $2$. Then $H$ is supersolvable.    
\end{Lemma}

\begin{proof}
If $L$ is an alternating or sporadic group, then $\Out(S)$ is cyclic and the result is immediate. If $L$ is of Lie type, then by (for example)~\cite[Thm.~2.5.12]{GLS3} $\Out(L)$ is a semidirect product $T \rtimes (E \times F)$, where $T$ is cyclic or $C_2 \times C_2$ (diagonal automorphisms), and $E$ is cyclic or $S_3$, and $F$ is cyclic (graph and field automorphisms). Then $\Out(L)$ (and so $H$) is supersolvable except possibly when $T \cong C_2 \times C_2$. In the remaining cases, note that our hypothesis that $|H|_2\leq 2$ again yields that $H$ is supersolvable.
\end{proof}

The following is from the first part of the proof of~\cite[Prop.~4.3]{eel20}, noting that the hypothesis that the defect group is abelian is not used in that part.

\begin{Lemma}
\thlabel{reduced_block}
Let $G$ be a finite group and $\CO Gb$ a block of $\CO G$ with defect group $D$. Then there is a finite group $H$ and a block $\CO Hc$ of $\CO H$ such that $\CO Gb$ is basic Morita equivalent to $\CO Hc$, and a defect group $D_H$ of $C$ is isomorphic to $D$ such that:
\begin{enumerate} 
\item[(R1)] $\CO Hc$ is quasiprimitive, that is, if $N \lhd H$, then $\CO Hc$ covers a unique block of $\CO N$;
\item[(R2)] If $N \lhd H$ and $\CO Hc$ covers a nilpotent block of $\CO N$, then $N \leq O_\ell(H)Z(H)$ with $O_{\ell'}(N) \leq [H,H]$ cyclic. In particular $O_{\ell'}(H) \leq Z(H)$;
\item[(R3)] $[H:Z(H)_{\ell'}]\leq  [G:Z(G)_{\ell'}]$.
\end{enumerate}
\end{Lemma}

\begin{Proposition}   
\thlabel{prelim-reduction}
Let $G$ be a finite group and $\CO Gb$ a block of $\CO G$ with defect group $D \cong Q_{2^n}$ for some $n \geq 3$. Then there is a finite group $H$, a block $\CO Hc$ of $\CO H$ such that $\CO Hc$ is basic Morita equivalent to $\CO Gb$, and a defect group $P$ of $\CO Hc$ (isomorphic to $D$) such that $[H:O_{2'}(Z(H))] \leq [G:O_{2'}(Z(G))]$, and $H$, $P$ satisfy one of:
\begin{enumerate}[(i)]
\item $H$ is quasisimple;
\item there is $N \lhd H$ such that $N$ is quasisimple, $[H:N]=2$, $H=NP$ and $N \cap P \cong Q_{2^{n-1}}$;
\item $H$ is solvable.
    
\end{enumerate}
\end{Proposition}

\begin{proof}
By \thref{reduced_block} $\CO Gb$ we may assume that $\CO Gb$ satisfies conditions (R1) and (R2). In particular $O_{2'}(G)\leq Z(G)$. 

Write $F^*(G) = F(G)L_1\cdots L_t$, where $F(G)$ denotes the Fitting subgroup of $G$ and $L_1,\ldots, L_t$ denote the components of $G$ (where $t\in \mathbb Z_{\geq 0}$, so for the moment we are not assuming the existence of components). Note that  $F(G)=O_2(G)Z(G)$. The block $\CO G b$ covers a unique block $\CO F^*(G)f$, which will have defect group conjugate within $G$ to $D\cap F^*(H)$. Without loss of generality we can assume that $D\cap F^*(G)$ itself is a defect group of $\CO F^*(G)f$. It is a classical result that defect groups are intersections of two Sylow subgroups, and the Sylow $2$-subgroups of $F^*(G)$ are of the form $O_2(G)S_1\cdots S_t$, where each $S_i$ is a Sylow $2$-subgroup of $L_i$. The intersection of two such subgroups retains the property that if $h_0h_1\ldots h_t$ is an element (where $h_0\in O_2(G)$ and $h_i\in L_i$ for $1\leq i \leq t$), then so is each $h_i$.  So 
  $$D\cap F^*(G)= D_0D_1\cdots D_t,$$
where $D_0=O_2(G)$ and $D_i=D\cap L_i$ for $1\leq i \leq t$. We claim that there is some $i\in\{0,\ldots, t\}$ such that $D\cap F^*(G)=D_i$. Otherwise there would be $i\neq j$ and elements $x\in D_i$ and $y\in D_j$ such that $x\not\in D_j$ and $y\not\in D_i$. Since the groups $O_2(G), L_1,\ldots, L_t$ are all normal in $F^*(G)$ and commute it follows that $\langle x,y \rangle$ is a non-cyclic abelian subgroup of $D\cap F^*(G)$. But $D$, being quaternion, does not have any such subgroups, a contradiction. Since $F^*(G)$ is a central product of $F(G)$ and the $L_i$ it follows that $D_j\leq Z(G)$ for all $j\neq i$.

Given a subset $J\subseteq \{0,\ldots,t\}$ we define the group $L_J=\prod_{j\in J}L_j\unlhd F^*(G)$, where we set $L_0=F(G)$ for notational convenience. Since conjugation in $L_J$ is the same as conjugation in $F^*(G)$ the block $\CO F^*(G)f$ covers a unique block $\CO L_Jf_J$, and the defect groups of $\CO L_J f_J$ are the conjugates of $D_J =\prod_{j\in J}D_j$. Note that if $L_J$ is normal in $G$ then $\CO L_J f_J$ is covered by $\CO G b$, and quasiprimitivity implies that it is $G$-stable. It then follows (e.g. from \cite[Thm.~6.8.9~(ii)]{LiBookII}) that every $G$-conjugate of $D_J$ is again a defect group of $\CO L_J f_J$, that is, every $G$-conjugate of $D_J$ is an $L_J$-conjugate (or, equivalently, an $F^*(G)$-conjugate).

If we take $i$ as above such that $D\cap F^*(G)=D_i$ then the preceding argument clearly implies that $L_i$ is normal in $G$, simply because otherwise $D_{\{0,\ldots,t\}}$ does not have the property that all of its $G$-conjugates are $F^*(G)$-conjugates. So $\CO G b$ covers $\CO L_{\{i\}}f_{\{i\}}$, which implies that the defect group of this block is non-cyclic (as cyclic defect blocks for prime $2$ are nilpotent), so $D\cap L_i$ is non-cyclic. But then  $L_{\{0,\ldots,t\}\setminus \{i\}}$ is also normal in $G$, and, as we saw above, $D_{\{0,\ldots,t\}\setminus \{i\}}\leq Z(G)$. It follows that $\CO L_{\{0,\ldots,t\}\setminus \{i\}} f_{\{0,\ldots,t\}\setminus \{i\}}$ is nilpotent and covered by $\CO Gb$. By our assumptions that means $L_{\{0,\ldots,t\}\setminus \{i\}} \leq O_2(G)Z(G)$, which is solvable. That can only be true if (a) $t=0$ and $i=0$, or if (b) $t=1$ and $i=1$. In case (a) we have $F^*(G)=O_{2'}(G)O_2(G)$ and $O_{2'}(G)\leq Z(G)$, so $G/C_G(F^*(G))$ is isomorphic to a subgroup of $\Aut(D_0)$, which is solvable since $D_0$ is cyclic or quaternion. Since $C_G(F^*(G))\leq F^*(G)$, which is also solvable, it follows that $G$ is solvable. Hence we may assume we are in case (b), so that $O_2(G)\leq D_1\leq L_1$, and $F^*(G)=O_{2'}(G)L_1$. 

Now $G/F^*(G)$ is isomorphic to a subgroup of $\Out(L_1)$, a solvable group. Hence \thref{solv_quotient_characteristic:lemma} applies, so $DF^*(G)/F^*(G)$ is a Sylow $2$-subgroup of $G/F^*(G)$. Since $D \cap F^*(G) \lhd D$ and is noncyclic, it then follows from \thref{normal_subgroup_quaternion} that either $D \leq F^*(G)$, in which case $[G:F^*(G)]$ is odd, or $[G:F^*(G)]_2=2$. By \thref{Out_sub_supersolvable} $G/F^*(G)$ is then supersolvable and we may apply \thref{2-nilpotent_quotient}, so that we may assume $G=F^*(G)$ or $[G:F^*(G)]=2$ with $G=DF^*(G)$. Possibly applying \thref{reduced_block} again, $O_{p'}(G) \leq Z(G) \cap [G,G]=[L_1,L_1]$ is cyclic. Hence $F^*(G)$ is quasisimple and (noting $O_2(G) \leq L_1$) we may take $L_1=F^*(G)$ as required.
\end{proof}

\begin{Proposition}
\thlabel{solvable_classification}
Let $G$ be a solvable and $\CO G b$ be a block of $\CO G$ with defect group $D \cong Q_{2^n}$ for $n \geq 3$. Then $\CO Gb$ is Morita equivalent to the principal block of one of $D$, $\SL_2(3)$ or $\operatorname{CSU}_2(3)$. 
\end{Proposition}

\begin{proof}
By \thref{reduced_block} we may assume $O_{2'}(G) \leq Z(G) \cap [G,G]$, so since $G$ is solvable we must have $C_G(O_2(G)) \leq O_2(G)Z(G)=O_2(G)O_{2'}(G)$. This implies that $D$ is a Sylow $2$-subgroup of $G$ \cite[Lem.~10.6.5]{LiBookII}. By \thref{normal_subgroup_quaternion} $O_2(G)$ is either cyclic, $D$ or isomorphic to $Q_{2^{n-1}}$. If $O_2(G)$ is cyclic or isomorphic to $Q_{2^m}$ for $m \geq 4$, then $\Aut(O_2(G))$, and so $G/O_2(G)O_{2'}(G)$, is a $2$-group. This implies $G=D \times O_{2'}(G)$ and so $\CO Gb$ is Morita equivalent to  $\CO D$. Hence we may assume that $n \leq 4$, and that if $n=4$, then $O_2(G) \cong Q_8$. 

Now $\Aut(Q_8) \cong S_4$, so if $n=3$, then $O_2(G)=D$ and $[G:DZ(G)]=1$ or $3$. If follows that $G=D$ or $G \cong SL_2(3)$.

If $n=4$ and $O_2(G) \cong Q_8$, then $G/O_2(G)Z(G) \cong S_3$. It follows that $G \cong CSU_2(3)$.
\end{proof}

\begin{Definition}
\thlabel{moritaminimal}
For a  finite group $G$ and  a block idempotent $b$ of $\CO G$, we say that  $(G,b)$ is {\it Morita minimal}  if  whenever
$H$ is a finite group and $c$ is a block idempotent of $\CO  H$ such that $\CO Gb$ is Morita equivalent to  $\CO Hc $, then
$[H:Z(H)_{\ell'}]\geq  [G:Z(G)_{\ell'}] $.
\end{Definition}

Note that a given Morita equivalence class will contain multiple Morita minimal representatives. For example, taking a direct product of a finite group in a Morita minimal pair with an abelian $\ell'$-group gives further Morita minimal pairs. However, it suffices in our situation to consider only certain Morita minimal pairs as follows.

\begin{Corollary}
\thlabel{main-reduction}
Let $G$ be a finite group and $\CO Gb$ a block of $\CO G$ with defect group $D \cong Q_{2^n}$ for some $n \geq 3$. Then there is a Morita minimal pair $(H,c)$, where $H$ is a finite group and $c$ is a block idempotent of $\CO  H$ with defect group $P \cong D$, such that $\CO Gb$ is Morita equivalent to  $\CO Hc$ and $H$ satisfies one of:
\begin{enumerate}[(i)]
\item $H$ is quasisimple;
\item there is $N \lhd H$ such that $N$ is quasisimple, $[H:N]=2$, $H=NP$ and $N \cap P \cong Q_{2^{n-1}}$;
\item $H$ is $P$, $\SL_2(3)$ or $\operatorname{CSU}_2(3)$, in which cases $\CO Hc$ is the principal block.
\end{enumerate}
\end{Corollary}

\begin{proof}  Let $(H,c)$ be a Morita minimal pair such that  $\CO Hc$ is Morita equivalent to $\CO Gb $.  By \thref{quat-Morita}  the defect  groups of $\CO Hc$  are ismorphic to $D$.  Thus, we may assume that  $(H,c)$  are as in the conclusion  of  \thref{prelim-reduction}. Now the result follows from  \thref{solvable_classification}, as it is easy to check that the principal blocks of $D$, $SL_2(3)$ and $CSU_2(3)$ are all Morita minimal.
\end{proof}

\section{Around Ruhstorfer's extensions of  Bonnaf\'e--Dat--Rouquier equivalences.}\label{sec:extendBDR}
Except where otherwise stated, in this section we let $\ell$ be any prime.
  In  \cite{Ruh22}  and \cite{Ru22},  Ruhstorfer has developed several results   extending  Bonnaf\'e--Rouquier and Bonnaf\'e--Dat--Rouquier  equivalences   through  automorphisms.   In this section we  present adaptations  of some  of these results (see Propositions~\ref{prop:diag2} and \ref{prop:field}) and derive some consequences. We are  not always  able to  directly  invoke  Ruhstorfer's statements  as we  do not have  the freedom  to  choose the isomorphism type of  the extension  inducing a given outer automorphism.

 Throughout this  section, we let  $\bG$ be   a  connected reductive group defined over a finite field of characteristic  prime to $\ell$  with  Frobenius $F$. Let $\bG \hookrightarrow \tilde \bG$  be a regular embedding such that $F$   extends to a Frobenius morphism on $\tilde \bG$   and let   $\tilde \bG^*  \twoheadrightarrow \bG^*  $ be  a dual   pair,  with Frobenius $F^*$. (See, for example, \cite[Ch.~15]{CaEnbook}.)

 Recall that  $Z(\bG^F)=Z(\bG)^F $ (see \cite[Prop. 3.6.8]{Carter}).  We  record here a well-known related fact.
\begin{Lemma}\thlabel{lem:GGF-centralisers}  Let $Z \leq  Z(\bG^F)$. Then $C_{\bG}(\bG^F/Z) = Z(\bG)$ and  $C_{\tilde \bG} (\bG^F/Z)=  Z(\tilde \bG)$.
 \end{Lemma}
\begin{proof} The second assertion is a consequence of the first  as $\tilde \bG = Z^{\circ}(\tilde \bG) \bG$.  We prove the first assertion.  Let $p$ be the prime number with $q=p^r$ and  let $x \in C_{\bG} (\bG^F/Z)$. Then  both the semisimple part and  the unipotent part of $x$  are in $C_{\bG} (\bG^F/Z)$.  Hence we may assume that $x$ is either semi-simple or that $x$ is unipotent. 

Let $u \in \bG^F$ be  a  regular unipotent element of $\bG$ (see  \cite[Prop.~5.1.7]{Carter})  and let  $ s \in \bG^F$ be a regular semi-simple element  of  $\bG$. The existence of  $s$  can be seen  from the fact that  the Steinberg character  of $\bG^F$   is of  $p$-defect $0$ (see \cite[Cor. 7.4.4]{DiMi}), hence there exists an (necessarily semi-simple)  element $s$ in $\bG^F$ whose  centraliser in  $\bG^F$ is a $p'$-group (see \cite[Thms.~6.2.1,~6.5.4]{LiBookII}). Since  $|C_{\bG}^{\circ}(s)^F|_p =  q^N$, where $N$ is the number of positive  roots of  $C_{\bG}^{\circ}(s) $ (see \cite[Props.~4.4.1,~4.4.2]{DiMi}  or \cite[Sec.~2.9]{Carter}), it follows that  $C_{\bG}^{\circ} (s)$  is a torus. A different argument for this fact  is given in \cite[Rem.~7.1.5]{DOR}   which   however excludes the groups $\SU_3(2)$.  

First suppose  that $x$ is  semisimple.   We have  $u \in C_{\bG} (xZ)$. Since  $Z$ is a finite central subgroup of $\bG$ (see remark above the statement) of order prime to $p$  and $u$ is a  $p$-element,  we  have that  $u \in C_{\bG}(x)$.  Since $u $ is regular this means that $ x \in Z(\bG)$  (see \cite[Prop.~5.1.5]{Carter})  as required. Suppose next that $x$ is  unipotent. Since $ x \in C_{\bG} (sZ)$, by the same argument as above, we have that  $x \in C_{\bG} (s)$. Since $C_{\bG}(s)/C_{\bG}^{\circ} (s)$ is a $p'$-group (see \cite[Prop.~3.5.3]{DiMi}),  $x  \in C_{\bG}^{\circ}(s) $,  a contradiction to the regularity of $s$.
 \end{proof}

 Recall that a  {\it diagonal}  automorphism  of $\bG^F$ is  an  automorphism  induced by conjugation by an element of  $\tilde \bG^F$, where  $\tilde  \bG $ is  a connected reductive  group    with connected centre  and with  Frobenius morphism  $F$   containing $\bG$  as a closed  subgroup such that $[\tilde  \bG, \tilde \bG] \leq  \bG$.  The set of diagonal automorphisms of $\bG^F$ forms a subgroup of $\Aut(\bG^F)  $ containing   $\Inn(\bG^F) $.      If $Z$ is a central subgroup  of $\bG^F$,  then  $Z$ is centralised by any diagonal automorphism  of $\bG^F$. By a  diagonal  automorphism  of  $\bG^F/Z$   we mean an automorphism which is   induced from a  diagonal automorphism of  $\bG^F$.   Let $\bT$  be  an $F$-stable   maximal torus of   $\bG$ and  let $ T_0$  be the  set of $ t\in \bT $ 
such that $t^{-1} F(t)  \in  Z(\bG^F)$. Then  $T_0$ is  a subgroup of $\bT$    normalising  $\bG^F$. Moreover,  the  image of  the map from $T_0\bG^F $  to $\Aut(\bG^F) $  which sends an element $x$ of 
$T_0\bG^F$ to conjugation by  $x$   is the group of  diagonal automorphisms   of  $\bG^F$.    We record the following  easy consequence.

\begin{Lemma} \label{lem:diagonalLevi}   Let  $\alpha $ be a diagonal automorphism of $\bG^F$.  Let $\bL $ be an $F$-stable Levi subgroup of  $\bG$, and  $\tilde \bL:= Z(\tilde \bG) \bL $.  Then  there exists an $h \in \tilde \bL^F$    and  $ g \in  \bG^F$ such that 
$\alpha (x) =  gh x (gh)^{-1}$  for all $x \in \bG^F$.
\end{Lemma}

\begin{proof}  Let  $\bT $  be an $F$-stable maximal torus of  $\bL$.     As explained above,   there exists  $g \in \bG^F $   and $ t \in T_0 $  such that  $\alpha $  is conjugation by $g t$. Since $Z(\tilde \bG)$ is connected,  and  $ t^{-1}F(t) \in  Z(\bG) $, by the Lang-Steinberg theorem,  
$t^{-1}F(t) = z^{-1} F(z) $   for some $ z \in  Z(\tilde \bG)$.  Then, $tz^{-1} \in  \bT Z(\tilde \bG) \cap  \tilde \bG^F \leq \tilde \bL^F $ and  $t$  acts on $\bG^F$  as  $tz^{-1}$. 
\end{proof}

Let $\tau: \bG_{sc}  \to  [\bG, \bG] $ be a  simply-connected covering    suitably equivariant with respect to  the  relevant Frobenius   morphism  (see  \cite[Sec.~8.1]{CaEnbook}  or \cite[Sec.~3] {FK}).    Then  $\tau (\bG_{sc}^F) $ is  a  normal subgroup of $\bG^F$ and  $ \bG^F =  \bT^F \tau (\bG^F_{sc} )$   for any $F$-stable maximal torus $\bT$ of $\bG$ (see   the text after  \cite[Rem.~8.18]{CaEnbook})  and  thus in particular 
$\bG^F/\tau(\bG_{sc}^F)$  is an  abelian  group.

\begin{Lemma} \label{lem:sccovering}     Let $\alpha $ be  a  diagonal automorphism  of   $\bG^F$.  Then $\tau (\bG_{sc}^F) $  is invariant under   $\alpha $ and   $\alpha  $   acts  trivially on $\bG^F/ \tau(\bG_{sc}^F) $.
\end{Lemma}

\begin{proof} Let  $\bT$  be an $F$-stable  maximal torus  of  $\bG$ and let  $T_0  =\{ t \in \bT  \mid  t^{-1}F(t) \in Z(\bG) \} = N_{\bT} (\bG^F)$.   Since   $\bG^F $  normalises  $\tau (\bG_{sc}^F) $   and  acts trivially on  $\bG^F/  \tau (\bG^F_{sc} )$, it suffices  to show that  $T_0$  normalises  $\tau(\bG_{sc}^F) $ and acts trivially on $\bG^F/  \tau (\bG^F_{sc} )$. Since  $\bG^F  =  \bT^F \tau (\bG^F_{sc} )$  and $T_0$   centralises  $\bT^F$, it suffices to show that  $T_0$  normalises  $\tau(\bG_{sc}^F) $.
Further,  since  $\bT =  Z(\bG)([\bG, \bG] \cap \bT ) $  and   $ Z(\bG)  \leq T_0 $,  it suffices  to show this for  $T_0 \cap  [\bG, \bG]$. Let $t\in \tau^{-1}(T_0 \cap  [\bG, \bG])$. Then, by definition of $T_0$, we have $t^{-1}F(t)\in\tau^{-1}(Z(\bG)\cap [\bG,\bG])\leq Z(\bG_{sc})$,  where the  last inequality follows from the  fact than any  proper closed normal  subgroup of  a  connected  reductive group is central. Hence $t$ normalises $\bG_{sc}^F$   and  it follows that  $T_0 \cap  [\bG, \bG]$  normalises   $\tau (\bG_{sc}^F) $. 
\end{proof}

The next result will be used for dealing with graph automorphisms in type $\type{A}$. 
\begin{Lemma} \label{lem-red-graph2}   Let $Z \leq Z(\bG^F)$ and  set $H = \bG^F/Z$  and  $\tilde H= \tilde \bG^F/Z$. Let $\tau $ be  an automorphism  of $\tilde \bG^F$ of $\ell$-power order  which normalises  $\bG^F$ and $Z$ and denote also by $\tau $  the induced maps on  $H$ and $\tilde H$. Let $H_1=H\langle x \rangle $   be a  group  containing $H$ as  a normal subgroup  and with $x$ an  $\ell$-element acting  on $H$ as $\tau \varphi $ for  some  diagonal  automorphism $\varphi$ of $H$  and let  $c$ be  an $H_1$-stable block idempotent of $\CO H$.
Then there exists  a   $\tau$-stable block idempotent  $b$  of   $\CO \tilde H$ with $bc \ne 0$.
Further, writing $b_1, c_1$ for the unique block idempotent for $\tilde \bG^F, \bG^F$ dominating  $b,c$ respectively, so that $b_1c_1 \neq 0$, the idempotent $b_1$ is $\tau$-stable. 
\end{Lemma}

\begin{proof} Write  $\overline{h}=hZ$ for $h \in \tilde \bG^F$ and for any $m  \in \tilde \bG^F$ write $c_m$ for the  automorphism of $\tilde  H$ induced  by conjugation with $m$. Then $\Inn(\tilde H) =\{c_m \mid m \in \tilde  \bG^F\}$. Let $g\in   \tilde \bG^F$ be such that $\varphi$ is induced by conjugation by $g$   and let $\alpha=  \tau \circ c_g\in \Aut(\tilde H)$ (so  $\alpha $  is  an extension  of conjugation by $x$ on $H$ to  $\tilde H$).  Now  the restriction of 
$\alpha $  to $H$ is  an automorphism  of $\ell$-power order,  and $\tau $  is also  of $\ell$-power order. Let $t \geq 1 $ be such that $\alpha^{\ell^t}|_{H} =1 =\tau^{\ell^t}$.  Since  $\Inn (\tilde H)$ is  a normal subgroup of $\Aut(\tilde H)$,  and $\tau^{\ell^t}=1$, we  have that  $\alpha^{\ell^t} =  c_{g'}$  for some $ g' \in  \tilde \bG^F $.  Since  $\alpha|_{H}^{\ell^t} =1$ it follows that $g'  \in  C_{\tilde \bG^F} (H)= Z(\tilde \bG^F)$ and hence  $\alpha $ also has $\ell$-power order.    By Lemma \ref{lem-stableblock},   applied  with 
$N=H$, $ G=\tilde H $  and  $ X=  \tilde H   \rtimes \langle \alpha \rangle$, we have  that   there 
exists an $\alpha$-stable  block  idempotent $b$  of  $\CO \tilde H $   with $bc \ne 0$. 
Since $ g \in \tilde \bG^F $, 
this means that  $b$ is also  $\tau$-stable. 

The final part is immediate.
\end{proof}

Recall that the  set of ordinary  irreducible characters  of $\bG^F$ is  a disjoint union of  Lusztig series  $\CE(\bG^F, s)$  as $s$ runs over the $\bG^{*F^*}$-classes of  semi-simple  elements  of $\bG^{*F^*}$.  Further, for any block  $\CO \bG^Fb $ of $\CO \bG^F$, there is a unique  $\bG^{*F^*}$-class  of  semi-simple  $\ell'$-elements $s$  of $\bG^{*F^*} $ such that $\CO \bG^Fb $   contains  an ordinary irreducible character in  $\CE(\bG^F, s)$ (see \cite[Thm.~9.12]{CaEnbook}). For  $s$  a semi-simple  $\ell'$-element of  $\bG^{*F^*}$  we  denote by  $e_s^{\bG^F}$ the  central idempotent of   $\CO  \bG^F$  (or $k\bG^F $) corresponding  to  $s$, that is,  the sum of  the block idempotents     $b$  such that  $\CO \bG^Fb $   contains  an ordinary irreducible character in  $\CE(\bG^F, s)$.

 \begin{Lemma}\thlabel{lem-primediagonalstable}  Let $Z \leq  Z(\bG^F)$,  let $H=\bG^F/Z$ and  let $H_1=H\langle x \rangle $   be a  group  containing $H$ as  a normal subgroup  and with $x$ an  $\ell$-element acting as a diagonal automorphism on $H$. Then every block of $\CO H $    is $H_1$-stable.   
\end{Lemma}

\begin{proof}  The conjugation action  of  $x$ on $H$  lifts to a diagonal automorphism of  $\bG^F$ of order a power of $\ell$. Therefore, we may assume that $Z=1$ and  that $H_1 \leq \tilde \bG^F$. Since every  block of $\CO \bG^F$ contains a character in  a Lusztig series $\CE (\bG^{F}, s )$, where $ s \in \bG^{*F}$ is  semisimple of $\ell'$-order,  it suffices to  prove that  the elements of $\CE (\bG^{F}, s )$ are $\langle x \rangle $-stable  for all $s \in \bG^{*F} $ of order prime to $\ell$.    Now the result follows from \cite[Cor. 11.13]{Bon06} and the fact that every element of  
$C_{\bG^*}(s)/C_{\bG^*}^{\circ} (s)$   has  order prime to $\ell$ (see \cite[Rem.~11.2.2]{DiMi}).
    \end{proof}
    
    \begin{Lemma} \label{lem: centralkerne}  Let $ Z\leq Z(\bG^F)$, let $H = \bG^F/Z $  and  for any $a \in \CO \bG^F $, denote by $\bar a $ the image of $a$  under the canonical surjection $\CO \bG^F \to \CO  H $. Let  $s\in  \bG^{*F}$ be a semi-simple element of $\ell'$-order. The following are equivalent.

\begin{enumerate}
\item [(i)]  $\bar e_s^{\bG^F}  \ne 0 $.
\item  [(ii)] $\bar b   \ne 0 $ for some   block idempotent $b$ of $\CO \bG^F$ such that  $b e_s^{\bG^F} \ne 0$.
\item [(iii)] $\bar b   \ne 0 $  for any  block idempotent $b$  of $\CO \bG^F$such that  $b e_s^{\bG^F} \ne 0$.
\item [(iv)]   For any  $F$-stable maximal torus  $\bT$ of $\bG$ and $\theta \in \Irr (\bT^F)$ such that  $(\bT, \theta) $  corresponds to the  class of $s$ via duality,  the restriction of $\theta $ to $Z$ is trivial.
\end{enumerate}
\end{Lemma}

\begin{proof}  Clearly, (i)  and  (ii)  are equivalent  and  (iii) implies (i).   
Let $b $ be a block idempotent  of $\CO \bG^F $  with   $be_s^{\bG^F} \ne 0$. Then  $\bar b \ne 0 $ if and only if $Z_{\ell'} $ is in the kernel   of every ordinary  irreducible  character of $\bG^F$  in $b$ which in turn holds   if and only if   $Z_{\ell'}$  is  in the kernel of  some irreducible  character  of  $\bG^F$ in $b$ (see \cite[Ch.~5, Thm.~8.8]{NT} or \cite[Lem.~2.10]{FK}.)

Let   $\chi \in {\mathcal E}(\bG^F , s)$ be  an irreducible character  which is in $b$ (there is always  such a character by \cite[Thm.~9.12]{CaEnbook}). Let   $\bT $ be  an $F$-stable maximal  torus  of $\bG$ and let $\theta \in \Irr(\bT^F) $ be  such that $(\bT, \theta)$  corresponds to $s$ via duality.
 Then, by \cite[Lem.~2.2]{MaHeight},   $\chi|_Z $ is  a  multiple  of  $\theta|_Z$.   Since $s$  is an $\ell'$-element, so is $\theta $. Hence   $Z_{\ell'} $ is in the kernel of $\chi$ if and only if   $\theta|_Z$  is the trivial character.   Thus (iv) implies (iii)  and (ii) implies  (iv). This proves the equivalence of (i) - (iv). 
 \end{proof}

  \subsection{Bonnaf\'e--Dat--Rouquier equivalences and diagonal automorphisms}\label{subsec:extenddiagonal} We put ourselves in the setting of \cite[Sec.~7]{BDR}.  Let    $s \in \bG^{*F^*}$ be a semi-simple element of  order prime to $\ell $. Let $\bL^* =C_{\bG^*}(Z^{\circ} (C^{\circ}_{\bG^*} (s)) )$  be a  minimal   Levi subgroup  of $\bG$  with respect to containing $C^{\circ}_{\bG^*} (s)$, and let $\bP^*$ be a  parabolic subgroup with $\bL^*$ as Levi complement.  
Let $(\bL, \bP )$  be dual  to $(\bL^*, \bP^* )$  and let  $\bN \leq N_{\bG} (\bL^F) $ be the  dual  to the subgroup $C_{\bG^*}(s) ^{F^*}\cdot\bL^*$  of $\bG^*$ as defined  in \cite[ Sec.~7A]{BDR}.   Denote by $\bU$
 the unipotent radical of   $\bP$, set $\tilde \bL  =  Z(\tilde \bG)  \bL$ and $\tilde \bN =  \tilde \bL \bN$.   Let   $$X  =  (\bG^F \times (\bL^F)^{\opp} )\Delta \tilde \bL^F ,   \text{ and }    Y= (\bG^F  \times  (\bN^{F})^{\opp} ) \Delta \tilde \bN^F,$$
   $$\tilde X=  \tilde\bG^F \times (\tilde \bL^F)^{\opp},   \text{ and }   \tilde  Y= \tilde \bG^F  \times  (\tilde \bN^F)^{\opp}. $$
 
Recall  the  Deligne-Lusztig variety  $$ Y_{\bU}^{\bG}  :=\{ g \bU \, : \,  g^{-1} F(g) \in \bU  F(\bU) \} \subseteq \bG/\bU.$$ 
Then    $Y_{\bU}^{\bG} $  is an  $X$-variety  via  $$ \,^{(y, z) (x, x^{-1})}  g \bU   =  y xg x^{-1}z \bU,    \  \ x\in \tilde \bL^F,   y \in \bG^F, z \in \bL^F $$ By functoriality, this action endows   the  associated  $i$-th  $\ell$-adic cohomology (with compact support)   groups   $H_c^i( Y_{\bU}^{\bG}, \CO)$    with the structure  of an $\CO X$-module.

The idempotent $e_s^{\bL^F}$ is $\tilde \bN^F$-stable and  $e_s^{\bG^F}$ is $\tilde \bG^F$-stable  hence $H_c^i( Y_{\bU}^{\bG}, \CO) e_s^{\bL^F}$  is an  $\CO  X$-module.  Note also that  the subgroup  $\Delta Z(\tilde \bG^F) $   acts  trivially  on $ Y_{\bU}^{\bG} $  and hence   also  on  $H_c^i( Y_{\bU}^{\bG}, \CO) e_s^{\bL^F}$. One of the main results of  \cite{BDR}  is the following theorem.

\begin{Theorem} \label{thm-BDR}  Let $d$ be  the dimension of   $Y_{\bU}^{\bG}$  and set  $M=  H_c^d (Y_{\bU}^{\bG}, \CO) e_s^{\bL^F}$.
\begin{enumerate} 
\item  Suppose that   there is  an   $\CO Y$-module   $M'$  which extends  the $\CO X$-module structure of  $M$. Then the restriction to $\CO (\bG^F \times  (\bN^F)^{\opp}) $  of  $M'$ induces  a Morita equivalence between $\CO \bG^F e_s^{\bG^F} $ and   $\CO \bN^F e_s^{\bL^F} $  and $\Ind_Y ^{\tilde Y} M' $ induces a Morita equivalence between  $\CO \tilde \bG^F e_s^{\bG^F} $ and   $\CO\tilde \bN^F e_s^{\bL^F} $.
\item  Suppose that  $\bN^F/ \bL^F $ is cyclic. Then the   $\CO X $-module  structure of   $M$  extends to  an  $\CO Y$-module structure.    
\end{enumerate}
\end{Theorem}

\begin{proof}  Part (1)  is  contained in the last part of the proof of  \cite[Thm.~7.5]{BDR}.  Part (2) follows from  \cite[Thm.~7.2]{BDR}  combined with \cite[Lem.~10.2.13]{Rou98}. Note that \cite[Thm.~7.2]{BDR} states that the  $\bG^F \times {\bL^F}^{\opp}$-module  $M$ is  $\bN^F$-stable, however the proof of \cite[Thm.~7.5]{BDR} requires  the stronger property  that  $M$ is  $\bN^F$-stable  as an $ X$-module. It can be checked  that  the proof of \cite[Thm.~7.2]{BDR} does indeed yield the stronger property. Note also  that \cite[Thm.~7.5]{BDR} is  stated for  $k$  in place of $\CO$, but the part of the proof  that we are using here works  in the same way  for $\CO$  (see also \cite[Thm.~1.1]{Ruh22}).
\end{proof}

The next Proposition extends  the Bonnaf\'e--Rouquier equivalences through   diagonal automorphisms. We make an elementary observation first that will be used in the proof.

\begin{Lemma} \label{lem:extensions}  Let  $N$  be  a finite group contained as a normal subgroup in finite groups $G =N \langle x\rangle $ and $H = N \langle  y \rangle$  such that that the conjugation action of $x$ on $N$  coincides with that of $y$  on $N$.    Let  $ N \cap \langle x \rangle  = \langle x^i \rangle $, with $ i \mid o(x)$  and set  $Z = \langle  x^i y^{-i}  \rangle \leq H $. Then  $Z$   is a   central subgroup  of $H$ and there is a surjective group homomorphism from   $ G$  to $H/Z$ which sends $n x^j$  to $n  y^j  Z$ for all $ n\in N $ and  all  $ j \in {\mathbb Z} $.
\end{Lemma}

\begin{proof}  
To see that $Z$ is a central subgroup of $H$, note that $x^iy^{-i}$ commutes with all elements of $N$ by assumption. Moreover, since $x^i\in N$ we have $^y(x^i)={^x(x^i)}=x^i$, again by assumption. That is, $y$ commutes with $x^i$, and therefore also with $x^iy^{-i}$. We conclude that $x^iy^{-i}$ commutes with all elements of $N\langle y\rangle=H$, as claimed.

Let us now show the second assertion. Since  $ x^i $ and  $y^i$  commute with each other, we have that $ (x^i y^{-i})^t =  x^{it} y^{-it} $   for all  $t \in {\mathbb Z}$.
Let $n_1, n_2 \in N $  and  $ j_1, j_2 \in {\mathbb Z}$   with  $n_1 x^{j_1} = n_2  x^{j_2} $. Then,   $n_2^{-1} n_1  = x^{j_2 -j_1}  \in  N \cap   \langle x \rangle =  \langle x^ i \rangle  $.  Since $ i \mid o(x)$, this means that $j_2- j_1  = i t $ for some $ t\in {\mathbb Z}$  and hence $n_2^{-1} n_1   \in    y^{j_2- j_1}  Z $. Thus the map from  $G$  to $H/Z$ in the statement of the proposition is well defined  and is clearly surjective.    It is a  homomorphism  since  $\,^{x^t} n  =  \,^{y^t} n $ for all $ t \in {\mathbb Z}$ and all $n \in N$.
\end{proof}

\begin{Proposition} \label{prop:diag2} Suppose that   $\bN^F/\bL^F $ is  cyclic.   Let   $H  =\bG^F/Z$  and  $Q= \bN^F/Z$  for   $Z$  a central subgroup of $\bG^F$  such that $ \bar e_s^{\bG^F} \ne 0  $, where   $\bar e_s^{\bG^F}$  is  the   image of   $e_s^{\bG^F} $  under the canonical surjection $\bG^F \to H $.   Suppose that  $H_1 $ is a  finite group  containing $H$ as a normal subgroup with $H_1= H \langle x\rangle $  for some $x$ acting by  a diagonal   automorphism on   $H$.
 Then $\bar  e_s^{\bG^F} $  is  $H_1$-stable. Moreover,
there exists 
 a  subgroup $Q_1 $  of $H_1 $  with $ Q_1 \cap H  =Q  $,  $Q_1/Q \cong  H_1/H$   and with  $Z(Q_1) \geq  Z(H_1) $   such that $\bar e_s^{\bL^F}$ is $Q_1$-stable  and  there is a Morita equivalence between $\CO H_1\bar e_{s}^{\bG^F}$ and $\CO Q_1 \bar e_{s}^{\bL^F}$.  
  \end{Proposition} 

\begin{proof}   
By Lemma \ref{lem:diagonalLevi},  we may assume that   $H_1 = H \langle  x \rangle   $  where $x$ acts   on $H$ by  some   $  h  \in \tilde \bL^F$.  
 Since $h \in \tilde \bL^F$, and $\bN$  contains $\bL$,  $\,^h \bN  =\bN  $, hence $\,^h \bN^F = \,^h(\bG^F \cap \bN ) = \,^h \bG^F \cap \,^h \bN =  \bG^F \cap \bN = \bN^F $.   Hence $h$  normalises  $Q$, from which it follows that    $x$  also normalises  $Q$.   Set $Q_1 =  Q \langle x \rangle$.   Since $C_{\tilde \bG^F} (H) =  Z(\tilde\bG^F)  \leq Z(\tilde \bL^F)$ (see \thref{lem:GGF-centralisers}) and  $\tilde \bL^F \cap \bG^F= \bL^F$, we  have that  $C_{H_1} (H)  \leq  (\bL^F/Z)\langle x \rangle$  from which it follows that  $Z(H_1) \leq  Z(Q_1) $  and that $x^a \in   H$  if and only if $x^a  \in  Q$,  hence $Q_1/ Q  \cong  H_1/H$.
 Since  $h$  is  acting as a diagonal automorphism on both $\bG^F$ and on $\bL^F$, $e_s^{\bG^F}$ and  $ e_s^{\bL^F}$  are both $\langle h \rangle$-stable and it follows that  $\bar e_s^{\bG^F}$  is $H_1$-stable and  $\bar e_s^{\bL^F} $ is $Q_1$-stable.

Let $\bar h:= h Z$   and let   $$\CD :=  (H \times Q^{\opp} ) \Delta \langle x  \rangle\leq  H_1 \times Q_1^{\opp}, $$ $$\CE:= (H\times Q^{\opp} )\Delta  \langle \bar h \rangle \leq  (H\times Q^{\opp} )\Delta (\tilde \bL^F/Z )\leq  \tilde \bG^F /Z\times \tilde \bN^F/Z. $$ Again,   since $C_{\tilde  \bG^F/Z} (H) =  Z(\tilde \bG^F)/Z  $, we have that  for any  $i$ such that $(x, x^{-1})^i  \in H \times Q$,    
 $ (x, x^{-1}) ^i (\bar h, \bar h^{-1})^{-i}  \in \Delta (Z(\tilde \bG^F )/Z)\cap  \CE $.  Thus  setting  $Z_1=   \langle x^i \bar h^{-i}\rangle  $,   by Lemma \ref{lem:extensions}, applied    with $N=H \times Q^{\opp}$, $G  =\CD$, $H= \CE$, $x = (x, x^{-1}) $,  $y=(\bar h, \bar h^{-1})$,  we have a  homomorphism from $\CD$  to  $\CE/\Delta Z_1$   which   restricts to the  natural  surjection  of $H \times Q $ onto $(H \times Q) \Delta Z_1/\Delta Z_1$.

 By Theorem  \ref{thm-BDR}, the $\CO X$-module   structure of  $M=  H_c^d (Y_{\bU}^{\bG}, \CO) e_s^{\bL^F}$ extends to a $\CO Y$-module structure. Let 
$M_1$ be such  an extension  and denote also by $M_1$   the  restriction to any  subgroup of $Y$.  Since  $  \Delta Z(\tilde \bG^F) $   acts trivially on  $M$, it does so on $M_1$. Hence by  Theorem \ref{thm-BDR}  and  Lemma \ref{lem:centraldomination},  $\CO \otimes_{\CO Z}M_1$  induces a  Morita equivalence between  $\CO H \bar e_s^{\bG^F} $ and  $ \CO Q  \bar  e_s^{\bG^F} $.   Again  by the trivial action of  $\Delta Z(\tilde \bG^F) $ on $M_1$, it follows that   $\CO \otimes_{\CO Z} M_1$ is  the  inflation  to   $\CO \mathcal {E} $  of an  $\CO \mathcal {E}/\Delta Z_1$-module  and we obtain 
via  restriction through the  homomorphism  above  a  $k\CD$-module structure on $\CO \otimes_{\CO Z} M_1$  whose   restriction to    $\CO(H\times Q^{\opp})$  induces a  Morita equivalence between $\CO H \bar  e_{s}^{\bG^F}$  and  $\CO Q\bar e_s^{\bL^F}$. Hence by Lemma \ref{lem:marcus}, $\Ind_{\CD}^{H_1 \times Q_1} \CO \otimes_{\CO Z} M_1$ induces a Morita equivalence between  $\CO H_1 \bar e_{s}^{\bG^F}$  and  $\CO Q_1 \bar e_s^{\bL^F}$.
\end{proof} 

The following extends \cite[Ex.~7.9]{BDR}. 
\begin{Proposition} \label{example2}   Keep  the notation  and assumptions of Proposition \ref{prop:diag2}  and let $Q_1 $ be  as  in the  conclusion of Proposition \ref{prop:diag2}.  Suppose  further that  $C^{\circ}_{\bG^*}(s) =\bL^* $,  and   that  $H_1/H$  is   an $\ell$-group.  Then  there is  a Morita equivalence  between the $\CO$-algebras  $\CO H_1 \bar e_{s}^{\bG^F}$  and  $\CO Q_1 \bar e_1^{\bL^F} $
\end{Proposition} 

\begin{proof}  Let $\tau: \bL_{sc} \to [\bL, \bL]$  be a simply connected  covering. 
The element $s $ defines a linear character,  say $\hat s$, of $\bL^F/ \tau (\bL_{sc}^F)$  such that $e_s^{\bL^F} =\hat s. e_1^{\bL^F}$.  By definition  of $\bN$, $ \hat s $ is $\bN^F$-stable.   Let $ h  \in \tilde \bL^F$  be  as in  the proof of  Proposition \ref{prop:diag2}. By  Lemma \ref{lem:sccovering}, $ \hat s $  is also $\langle h \rangle $-stable.    Since $\bar e_s^{\bG^F}  \ne 0 $, by Proposition  \ref{prop:diag2}, $\bar e_s^{\bL^F} \ne 0 $. Hence by  Lemma \ref{lem: centralkerne}  and  \cite[Lem.~3.2]{FK}, we have that   $Z$  is  in the kernel of  $\hat s$. Thus $\hat s $  may be viewed as  a    $Q_1$-stable  character of $\bL^F/Z$. Further,  
  since $(\bN^F/Z) /(\bL^F/Z) \cong \bN^F/\bL^F $ is cyclic,  there are  $ |\bN^F/\bL^F| $  distinct  extensions of  $\hat s$    to $\bN^F/Z$. Finally, since $Q_1/ (\bN^F/Z) $ is an $\ell$-group,  whereas $\bN^F/\bL^F$  is an $\ell'$-group, one of these  extensions  is $Q_1$-stable. Therefore, since  $Q_1 /(\bN^F$/Z) is cyclic,   there  is a further extension  of $\hat s $  to a linear character, say $\theta$  of  $Q_1$.   Then $\theta $ induces an $\CO$-algebra isomorphism between $\CO  Q_1 \bar  e_{s}^{\bL^F}$ and $\CO  Q_1 \bar  e_1^{\bL^F}$.  The result now follows by Proposition \ref{prop:diag2}.
\end{proof}

If  $\ell=2 $  and  $\bG$   is simple  of   classical type   $\type{A}$, $\type{B}$, $\type{C}$ or $\type{D}$, then the assumptions of Proposition~\ref{example2}  all hold (in fact, unless  $\bG$ is of  type $\type{A}$, we have that $\bL^* = C_{\bG^*}(s) $). We  thus  obtain the following corollary.
\begin{Corollary}  \thlabel{cor:extendclassical2}   Suppose that  $\ell=2 $ and that  $\bG$   is simple  of   classical type   $\type{A}$, $\type{B}$, $\type{C}$, or $\type{D}$.  
Let $Z \leq Z(\bG^F) $, let $H = \bG^F/Z$  and let $H_1 $  be  a   group containing  $H$  as  a normal  subgroup with $H_1/H$ a cyclic $2$-group acting as  diagonal automorphisms on $H$.   Then any block algebra  $\CO H_1b$   is Morita equivalent  to  a  block  algebra  $\CO Q_1c$, where  $Q_1$  is a  subgroup  of $H_1$  containing $Z(H_1) $   and with   normal  subgroups $L \leq  N $  such that  $N/L$ is a  cyclic  $2'$-group, $Q_1/N$  is a cyclic $2$-group  and  $ \CO Q_1c$ covers the principal block   of $\CO L$.   Consequently,  $\CO H_1b$ is Morita equivalent to  a principal block algebra and  the $\CO$-Morita--Frobenius number of   $ \CO H_1 b$   equals  $1$.     
\end{Corollary}

\begin{proof}    The  idempotents  $\bar e_t^{\bG^F}$,  where $t$  runs over a  set of representatives  of  the $\bG^F$-classes of  semi-simple odd-order elements of $\bG^F$,  are central in $kH$  and pairwise orthogonal.  Thus the first assertion follows from Proposition \ref{example2}   and the fact that $\bG$  being classical implies that    the idempotent  $\bar e_1^{\bL^F}$  in   Proposition \ref{example2}  is the principal block idempotent  of $k(\bL^F/Z)$ by \cite[Thm.~21.14]{CaEnbook}.
The  second assertion  follows  from the first by Proposition \ref{prop:frobeniusnormal}.
   \end{proof}

\subsection{Bonnaf\'e--Rouquier equivalences and field automorphisms}\label{subsec:extenfield} For the rest of the section,   we switch notation  and   let $\bL $  be   a minimal   Levi subgroup of $\bG$  such that $\bL^*$  contains $C_{\bG^*}(s)$. (Here $s$ is said to be quasi-isolated in $\bL^\ast$, and the blocks in $\CO \bL^Fe_s^{\bL^F}$ are called quasi-isolated blocks.)   Let $F_0$  be an  endomorphism of $\bG$ such that   $F =F_0^r$  and  assume that  $e_s^{\bG^F} $  is   $F_0$-stable.  By \cite[Lem.~4.6]{Ru22}, we may assume that $\bL$ is  $F_0$-stable.

\begin{Proposition} \thlabel{prop:field}   Let   $H  =\bG^F/Z$,  $Z \leq  Z(\bG^F)$.  Suppose that $\bar e_s^{\bG^F}  \ne 0$, where $\bar e_s^{\bG^F}$  is  the   image of   $e_s^{\bG^F} $  under the canonical surjection $\bG^F \to H $.  Suppose that  $F_0$ acts trivially  on $Z^F$ and   $e_s^{\bG^F} $ is   $F_0$-stable. Let $\bL $ be  an    $F_0$-stable Levi subgroup of $\bG$  such that $\bL^*$  contains $C_{\bG^*}(s)$, and set $Q= \bL^F/Z$. Suppose that  $H_1 $ is a  finite group  containing $H$ as a normal subgroup with $H_1= H \langle x\rangle $  for some $x$ acting by  the  composition of  a diagonal   automorphism  and $F_0$  on  $H$.
 Then $\bar  e_s^{\bG^F} $  is  $H_1$-stable. Moreover,
there exists 
 a  subgroup $Q_1 $  of $H_1 $  with $ Q_1 \cap H  =Q  $,  $Q_1/Q \cong  H_1/H$   and with  $Z(Q_1) \geq  Z(H_1) $   such that $\bar e_s^{\bL^F}$ is $Q_1$-stable  and  there is a Morita equivalence between $\CO H_1\bar e_{s}^{\bG^F}$ and $\CO Q_1 \bar e_{s}^{\bL^F}$.  
  \end{Proposition} 

\begin{proof}   We follow the  proof of Proposition  \ref{prop:diag2}.
   By Lemma \ref{lem:diagonalLevi},  we may assume that   $H_1 = H \langle  x \rangle   $  where $x$ acts   on $H$ by   $F_0 \circ h $, for some    $  h  \in \tilde \bL^F$.  Set $Q_1 =  Q \langle x \rangle$,
   Let $\bar h:= h Z$   and let   $$\CD :=  (H \times Q^{\opp} ) \Delta \langle x  \rangle\leq  H_1 \times Q_1^{\opp}, $$ $$\CE:= (H\times Q^{\opp} )\Delta  \langle \bar h  F \rangle \leq  (H\times Q^{\opp} )\Delta (\tilde \bL^F/Z \langle F_0\rangle ). $$   As  in  the proof of Proposition  \ref{prop:diag2},    we have a  homomorphism from $\CD$  to  $\CE/\Delta Z_1$   which   restricts to the  natural  surjection  of $H \times Q $ onto $(H \times Q) \Delta Z_1/\Delta Z_1$. 
 
By \cite[Prop.~5.7]{Ru22}, the $\CO \bG^F\times \CO {\bL^F}^{\opp}$-module   structure of  $M=  H_c^d (Y_{\bU}^{\bG}, \CO) e_s^{\bL^F}$ extends to a $\CO (\bG^F \times  (\bL^F)^{\opp})\Delta (\tilde \bL^F \langle F_0\rangle )$-module structure. Let 
$M_1$ be such  an extension  and denote also by $M_1$   the  restriction to any  subgroup of $ (\bG^F \times  (\bL^F)^{\opp})\Delta (\tilde \bL^F \langle F_0\rangle )$.  Since  $  \Delta Z(\tilde \bG^F) $   acts trivially on  $M$, it does so on $M_1$. Hence by  Theorem \ref{thm-BDR}  and  Lemma \ref{lem:centraldomination},  $\CO \otimes_{\CO Z}M_1$  induces a  Morita equivalence between  $\CO H \bar e_s^{\bG^F} $ and  $ \CO Q  \bar  e_s^{\bG^F} $.   Again  by the trivial action of  $\Delta Z(\tilde \bG^F) $ on $M_1$, it follows that   $\CO \otimes_{\CO Z} M_1$ is  the  inflation  to   $\CO \mathcal {E} $  of a  $\CO \mathcal {E}/\Delta Z_1$-module  and we obtain 
via  restriction through the  homomorphism  above  a  $k\CD$-module structure on $\CO \otimes_{\CO Z} M_1$  whose   restriction to    $\CO(H\times Q^{\opp})$  induces a  Morita equivalence between $\CO H \bar  e_{s}^{\bG^F}$  and  $\CO Q\bar e_s^{\bL^F}$. Hence by Lemma \ref{lem:marcus}, $\Ind_{\CD}^{H_1 \times Q_1} \CO \otimes_{\CO Z} M_1$ induces a Morita equivalence between  $\CO H_1 \bar e_{s}^{\bG^F}$  and  $\CO Q_1 \bar e_s^{\bL^F}$.

It remains to  show that   $Q_1/Q \cong  H_1/H$   and  $Z(Q_1) \geq  Z(H_1) $.  This  will   follow  as in Proposition~\ref{prop:diag2}, using that $$C_{\tilde \bG^F \langle F_0 \rangle} (\bG^F/Z)  \leq C_{\tilde \bG^F \langle F_0 \rangle} (\bG^{F_0}/Z) \leq  C_{\tilde \bG^F} (\bG^{F_0}/Z) \langle F_0\rangle \leq   Z(\tilde \bG^F) \langle F_0 \rangle \leq  Z(\tilde \bL^F)  \langle F_0 \rangle.$$  Here, note that by \thref{lem:GGF-centralisers}, $C_{\tilde \bG^F} (\bG^{F_0}/Z) \leq  C_{\tilde \bG} (\bG^{F_0}/Z)  \leq  Z(\tilde \bG)$. 
\end{proof}

\begin{Corollary} \thlabel{cor:minisquasi-isolated} Let   $H  =\bG^F/Z$ and  let $H_1$  be  a  finite group  containing $H$ as a normal subgroup with $H_1= H \langle x\rangle $  for some $x$ acting by  the  composition of  a diagonal   automorphism  and $F_0$  on  $H$, where  $F_0$  is an  endomorphism of $\bG$ such that   $F =F_0^r$  (for some $r$).  Suppose that $b$ is a block idempotent of $\CO H_1$ such that $(H_1, b)$ is Morita minimal. Then $\CO H_1b$  covers a quasi-isolated block of $\CO H$.
\end{Corollary}

\begin{proof}  Let $\CO Hc$  be a block of $\CO H $ covered by $\CO H_1b$. By  minimality, $c$ is $H_1$-stable. Now the result follows  from \thref{prop:field}.\end{proof}


\section{Quaternion Defect Groups in quasisimple groups of Lie type in non-defining characteristic}\label{sec:lie-non}

For this section, we return to the situation $\ell=2$. The aim here is to describe the blocks of finite quasisimple groups of Lie type in odd characteristic having quaternion defect groups $Q_{2^n}$ with $n\geq 3$. The case of the remaining quasisimple groups will be addressed in Section \ref{sec:altspor}.

\subsection{The Setup}\label{sec:setup}
Throughout  this section, we let $q$ be a power of an odd prime and let $G:=\bG^F$, where $\bG$ is a simple, simply-connected algebraic group and $F\colon \bG\rightarrow\bG$ is a Frobenius morphism defining $G$ over $\FF_q$. Then, aside from a finite number of cases (see \cite[Thm.~24.17]{MT}), $G/\zent(G)$ is a finite simple group of Lie type. On the other hand, if $S$ is a finite simple group of Lie type with non-exceptional Schur multiplier, then the Schur covering group of $S$ is such a group $G=\bG^F$, with $\zent(G)$ the Schur multiplier. (See e.g.  \cite[Rem.~24.19]{MT}). We remark that quasi-simple groups covering groups of Lie type in odd characteristic not covered here (in particular, covering groups by exceptional Schur multipliers) will be treated in Section \ref{sec:altspor}.

In this situation, the automorphisms of a quasi-simple group $G/Z$ for $Z\leq \zent(G)$ are those induced by the automorphisms of $G$. In particular, a diagonal automorphism of $G/Z$ is induced by a diagonal automorphism of $G$, which we recall are induced by the action of $\wt{G}:=\wt\bG^F$ on $G$, where $\iota\colon\bG\hookrightarrow\wt\bG$ is a regular embedding as discussed in Section \ref{sec:extendBDR}.
Field automorphisms are those induced by maps $F_0$ as before, and graph automorphisms are those coming from symmetries of the corresponding Dynkin diagram. 
By \cite[Thm.~2.5.12]{GLS3}, every automorphism of $G/Z$ is a product of  diagonal (which includes inner automorphisms in our definition), field, and graph automorphisms.

Let $(\bG^\ast, F^\ast)$ be dual to $(\bG, F)$ and write $G^\ast:={\bG^\ast}^{F^\ast}$. Given a semisimple element $s$ of $G^\ast$ with odd order, we write $\mathcal{E}_2(G, s)$ for the union of blocks of $\CO G e_s^{\bG^F}$. 
Note that by \cite[Prop.~2.6]{KM13}, there is some block in $\mathcal{E}_2(G, s)$ with defect group of size $|\cent_{G^\ast}(s)|_2$. Such a block is sometimes called a \emph{maximal} block in $\mathcal{E}_2(G, s)$. If $\pi\colon \bG\rightarrow\bG_{ad}\cong \bG^\ast$ is the adjoint surjection, then as in \cite[Sec.~2.2]{Ruh25}, we have an $F$-stable subgroup $\bG(s):=\pi^{-1}(\cent_{\bG^\ast}(s))\leq \bG$ dual to $\cent_{\bG^\ast}(s)$. We write $\bG^\circ(s)$ for the connected component of $\bG(s)$, which is then dual to $\cent_{\bG^\ast}^\circ(s)$, and we have $\bG(s)/\bG^\circ(s)\cong A(s)=\cent_{\bG^\ast}(s)/\cent_{\bG^\ast}^\circ(s)$. Further, write $G(s):=\bG(s)^F$ and $G^\circ(s):=\bG^\circ(s)^F$. 

Throughout, we also write $d:=d_2(q)$ for the order of $q$ modulo $4$.

\subsection{Quasi-Isolated Blocks}\label{sec:quasiisol}

Keep the situation above. In this subsection and the next, we will  analyze the situation in which $\CO G b$ is a quasi-isolated $2$-block, meaning that $\CO G b$ is in $\mathcal{E}_2(G,s)$ where $s$ is a semisimple $2'$-element such that $\cent_{\bG^\ast}(s)$ is not contained in any proper Levi subgroup  of $\bG^\ast$ (such a semisimple element is called quasi-isolated).

\begin{Lemma}\thlabel{lem:classicalqi}
    Suppose that  $\bG$ is of classical type but not of type $\tA$. Let $D$ be a defect group of a quasi-isolated $2$-block of $\bG^F$. Let $Z\leq \zent(G)$ be a $2$-group. Then $D/Z$ is neither quaternion, cyclic, nor trivial.
\end{Lemma}
\begin{proof}
    By \cite{bon}, the unipotent blocks are the only quasi-isolated blocks of $G$, and by \cite[Thm.~21.14]{CaEnbook}, the only unipotent block of $G$ is the principal block. Then $D$ is a Sylow $2$-subgroup of $G$ and $D/Z$ is a Sylow $2$-subgroup of $G/Z$. Note that by our assumption, $G\neq \operatorname{Sp}_2(q)\cong \SL_2(q)$. The structure of Sylow $2$-subgroups of the classical matrix groups is given in \cite{CarterFong}. From the descriptions there, we see if $G=\operatorname{Sp}_{2n}(q)$ (that is, $\bG$ is type $\tC$), then $D/Z$ is not quaternion since $G\neq\operatorname{Sp}_{2}(q)$. If $\bG$ is type $\tB_n$ or $\tD_n$, then we may assume that $n\geq 3$, respectively $n\geq 4$. By \cite[Thm.~4.10.2]{GLS3}, $D$ contains the Sylow $2$-subgroup $T_2$ of a a Sylow $d$-torus of $(\bG, F)$, and note that $Z\leq T_2\lhd D$. Since $T_2/Z$ is not cyclic for $n\geq 3$, we have $D/Z$ is not quaternion.
    (This could also be seen by \cite[Thm.~1.1]{kl20} by comparing the number of Brauer characters.)
    
    Further, as $G\neq\SL_2(q)$, $D/Z$ is not cyclic (nor trivial), as no other such $\bG^F/Z$ has abelian Sylow $2$-subgroups (see e.g. the discussion after \cite[Prop.~2.2]{Ma14}). 
\end{proof}

We next deal with the exceptional types. For $\epsilon\in\{\pm1\}$, we denote by $\type{E}_6(\epsilon q)$ the untwisted group $\type{E}_6(q)$ when $\epsilon=1$ and the twisted group $\tw{2}\type{E}_6(q)$ when $\epsilon=-1$. Note that the Ree groups $\tw{2}\type{G}_2(q)$ are omitted here and dealt with instead in Section \ref{sec:altspor}.

\begin{Lemma}\thlabel{lem:exceptionalqi}
    Suppose that $G$ is of exceptional type $\tG_2(q)$, $\tw{3}\tD_4(q)$, $\tF_4(q)$, $\tE_6(\epsilon q)$, $\tE_7(q)$, or $\tE_8(q)$. Let $D$ be a defect group of a quasi-isolated $2$-block $\CO G b$ of $\CO G$ and let $Z\leq \zent(G)$ be a $2$-group. Then $D/Z$ is neither quaternion nor a non-trivial cyclic group. 
    
    Further, if $D/Z$ is trivial, then $\CO G b$ is a defect-zero unipotent block such that $\Irr(\CO G b)$ is comprised of a $d$-cuspidal unipotent character of $G$. (These occur in the cases $G=\tG_2(q)$, $\tF_4(q)$, $\tE_6(\epsilon q)$, and $\tE_8(q)$.)
\end{Lemma}
\begin{proof}
As before, let $\CO G b$ be a block in $\mathcal{E}_2(G,s)$ with $s$ a quasi-isolated semisimple element of $G^\ast$ and let $D$ be a defect group for $\CO G b$. Note that $Z$ is trivial unless $G=\tE_7(q)$, in which case $|Z|$ divides $2$. 

By \cite[Thm.~7.12(d)]{KM13}, there is associated to $\CO G b$ a quasi-central $d$-cuspidal pair $(\bL, \lambda)$ such that $\bL=\cent_{\bG}(\zent(\bL)_2^F)$ and $(\zent(\bL)_2^F, b_{\bL^F}(\lambda))$ is a $b$-Brauer pair. In particular, we may assume that $\zent(\bL)_2^F$ is contained in $D$ (by e.g. \cite[Thm.~4.14]{N98}). If $D/Z$ is nontrivial cyclic or quaternion, then $\zent(\bL)_2^F/Z$ must contain a unique involution, and hence be cyclic.

Now, the structure of the possible pairs $(\bL, \lambda)$ are given in the numbered lines of Tables 2-5, 9 of \cite{KM13} if $s\neq 1$ and by the numbered lines of the tables in \cite[Sec.~3.2]{Eng00} if $s=1$. As noted in \cite{KM13, Eng00}, the cases not shown are obtained by an appropriate Ennola duality. Hence, in the following, we may consider only those listed, but recall that the situation for each block discussed also applies to an Ennola dual block. (For example, the situation of Case 8 of $\tE_6(q)$ when $d=2$ discussed below also applies to a block in $\tw{2}\tE_6(q)$ when $d=1$.)

If $s\neq 1$, then we see from this that $\zent(\bL)_2^F/Z$ is nontrivial and is further noncyclic except for the following cases:
\begin{enumerate}[(a)]
    \item Case 2 of \cite[Table 9]{KM13} for $G=\tG_2(q)$;
    \item Case 8 of \cite[Table 3]{KM13} for $G=\tE_6(q)$; and
    \item Case 6 of \cite[Table 5]{KM13} for $G=\tE_8(q)$.
\end{enumerate}
If $s=1$, that is, $\CO G b$ is a unipotent $2$-block, then we see from \cite[Sec.~3.2]{Eng00} that $D/Z$ is trivial only when $\CO G b$ is defect-zero consisting of a $d$-cuspidal unipotent character. Otherwise, we have $\zent(\bL)_2^F/Z$ is noncyclic except in the following case:
\begin{enumerate}[(a)]
\setcounter{enumi}{3}
\item Rows 3, 7 for $G=\tE_7(q)$ in \cite[Sec.~3.2]{Eng00}.
\end{enumerate}

It now suffices to know that in cases (a)-(d), $D/Z$ has at least two involutions.
In case (a), the block is maximal and, as pointed out in \cite[Thm.~6.1]{Ruh25}, $D$ is a Sylow $2$-subgroup of $G(s)\cong \operatorname{SU}_3(q)$, hence is semidihedral. (This is also noted in \cite{hiss}.)

In case (b), $\CO G b$ is again maximal, but this is one of the cases excluded in \cite[Thm.~6.1]{Ruh25}. However, \cite[Cor.~12.12]{Ruh25} yields that $D$ is still isomorphic to a Sylow $2$-subgroup of $G(s)$. Hence $D$ is isomorphic to a Sylow $2$-subgroup of $\SL_3(q^3)$, and is therefore again semidihedral. 

In case (d), the discussion following the relevant table in \cite{Eng00} yields that $D/Z$ is dihedral. 

So, finally consider case (c). From the discussion in the proof of \cite[Prop.~6.4]{KM13} we have $\zent(\bL)^F_2\cong C_{(q-1)_2}$ and $D$ is of the form $C_{2(q-1)_2}.2$. Writing $D_0:=C_{2(q-1)_2}$, we further have an element of $D\setminus D_0$ acts on $\zent(\bL)^F_2$ by inversion. Letting $z\in\zent(\bL)^F_2$ be the involution, we have $\bL^F\leq \cent_G(z)$, $z$ is central in $D$, and \cite[Prop.~4.15(c)]{AlperinBroue} yields that $D$ is also the defect group of some block of $\cent_G(z)$. By \cite[Tab.~4.5.1]{GLS3} we have $\cent_{G}(z)\cong (\SL_2(q)\circ\tE_7(q)).2=\langle \SL_2(q)\circ\tE_7(q), x\rangle$, where the central product identifies the centers of $\SL_2(q)$ and $\tE_7(q)$ and $x$ acts on $\SL_2(q)$ and $\tE_7(q)$ simultaneously as the outer diagonal automorphism. As in the proof of \cite[Prop.~6.4]{KM13}, we have $D_0\leq \SL_2(q)\circ\tE_7(q)$, and $x$ induces the group $D/D_0$. 
From \cite[Tab.~4.5.2]{GLS3} we see $x$ can be chosen as an involution, completing the proof.
\end{proof}

In the remainder of this section we show that blocks with quaternion defect groups of quasisimple groups of type $\tA$ in non-defining characteristic are Morita equivalent to principal blocks, and further describe (when combined with Section \ref{sec:altspor} below) the blocks of quasi-simple groups with quaternion defect groups. While the latter is not explicitly needed for our main results, we believe this could be of independent interest.

\subsection{Quasi-Isolated Blocks in Type $\tA$}\label{sec:typeAqi}
If $\bG$ is type $\tA_{n-1}$, the only isolated blocks are the unipotent blocks by \cite{bon}, and there is a unique unipotent block by \cite[Thm.~21.14]{CaEnbook}. That is, the principal block is the only isolated block. Hence, it will be useful to reduce this case to isolated blocks. For this, we may use the results of Bonnaf{\'e}--Dat--Rouquier \cite{BDR} (see Theorem \ref{thm-BDR}). 

Namely, let $\bH^\ast$ be an $F$-stable Levi subgroup of $\bG^\ast$ containing $\cent_{\bG^\ast}^\circ(s)$ and minimal with respect to this property, so that $s$ is isolated in $\bH^\ast$. Let $\bH$ be an $F$-stable Levi subgroup of $\bG$ dual to $\bH^\ast$. 
Then Theorem \ref{thm-BDR}
yields that $\CO G b$ is Morita equivalent to a block $\CO N c$ of $\CO N$, where $N=\bN^F$ for $\bN:=\norm_{\bG}(\bH, e_{s}^{\bH^F})\leq \norm_{\bG}(\bH)$. 
Further, these blocks have the same defect group, say $D=D(b)=D(c)$, by \cite[Thm.~7.7]{BDR}, \cite[Thm.~2]{Ruh20}.

Now, $\CO N c$ covers an isolated block of $\CO H'$, where $H':=[\bH, \bH]^F$, a product of groups of simple, simply-connected types. Then $D$ contains a defect group of such a block. Let $\CO H' c'$ be an (isolated) block of $\CO H'$ below $\CO N c$, and assume $D$ is quaternion. Then $\CO H' c'=\bigotimes \CO H_i c_i$ for $\CO H_i c_i$ isolated blocks of the direct factors $H_i$ of $H'$. Then at most one of these factors has positive defect (as otherwise $D$ contains multiple involutions). Without loss, say this is $H_1$, with block $\CO H_1c_1$ and defect group $D_1$.

Now, suppose that $s\in G^\ast$ is a quasi-isolated, odd-order semisimple element and let $\CO G b$ be  a block in $\mathcal{E}_2(G,s)$.  Write $\iota\colon \bG\rightarrow \wt\bG$ and $\iota^\ast\colon \wt\bG^\ast\rightarrow \bG^\ast$ for a regular embedding and its dual and let $\wt{s}\in\wt{G}^\ast$ with $\iota^\ast(\wt{s})=s$ and $\CO\wt{G}\wt{b}$  a block in $\mathcal{E}_2(\wt{G},\wt{s})$  above $\CO G b$. Note that in fact $\CO\wt G\wt{b}=\mathcal{E}_2(\wt{G},\wt{s})$ and $\CO G b=\mathcal{E}_2(G,s)$ since $G(s)$ and $\cent_{\wt{G}^\ast}(s)$ have unique unipotent blocks by \cite[Thm.~21.14]{CaEnbook}.  

We have $\cent_{\bG^\ast}^\circ(s)=\iota^\ast(\cent_{\wt\bG^\ast}(\wt{s}))=\cent_{\wt\bG^\ast}(\wt{s})/\zent(\wt\bG^\ast)$ by \cite[(2.2)]{bon}. Note that $G^\circ(s)$ is dual to $\cent_{ G^\ast}^\circ(s) \cong (\cent_{\wt\bG^\ast}(\wt{s})/\zent(\wt\bG^\ast))^F\cong (\cent_{\wt\bG^\ast}(\wt{s}))^F/(\zent(\wt\bG^\ast))^F=\cent_{\wt{G}^\ast}(\wt{s})/\zent(\wt{G}^\ast)$ (where the last two equalities are by \cite[Ex.~1.4.11]{GM20} since $\zent(\wt\bG^\ast)$ is connected). 
Further, in this situation,  $\cent_{\bG^\ast}^\circ(s)$ is a Levi subgroup, so $\bH^\ast=\cent_{\bG^\ast}^\circ(s)$ and $\bH=\bG^\circ(s)$. Note that $\CO N c$ lies above the principal block $\CO H b_0(H)$ of $H=G^\circ(s)$.
From \cite[Prop.~5.2]{bon} we have $A(s)\cong C_m$, where $m:=o(s)$ is odd and from \cite[Prop.~6.1]{Ruh25}, we see that a Sylow $2$-subgroup of $G(s)$ is a defect group for $\CO G b$.

Then we know that $D(b)=D(c)=D(b_0(G^\circ(s)))$. Recall $N/\bH^F=N/G^\circ(s)$ is cyclic here, so $\CO N c\cong\CO N b_0(N)$ via some linear $\beta\in\Irr(N/G^\circ(s))$ by Lemma \ref{lem:normalabelian}. Hence if $D$ is quaternion, then $\CO N b_0(N)$ has quaternion defect groups; that is, $N$ has  quaternion Sylow $2$-subgroups. Then $\CO N b_0(N)$ is Morita equivalent to one in the list in \cite[Thm.~1.1]{kl20}. Then so is $\CO N c$, hence $\CO G b$.

We next further analyze the structure of $\bH^F=G^\circ(s)$ in this situation.

\subsubsection{Strictly Quasi-Isolated Elements in Type $\tA$}\label{sec:strictlyqi}
We continue with the notation above, so $G=\SL_n(\epsilon q)$ and $\wt{G}=\GL_n(\epsilon q)\cong \wt{G}^\ast$. As noted in \cite[Thm.~A]{Ruh22}, the main result of \cite{BR} 
can be applied when $\bH^\ast$ is minimal with respect to the property $\cent_{{\bG^\ast}^F}(s)\cent_{\bG^\ast}^\circ(s)\leq \bH^\ast$. 
In this situation, $s$ is sometimes called ``strictly quasi-isolated" in $\bH^\ast$. The strictly quasi-isolated elements are studied in \cite[Sec.~5.3]{Ruh22} using the description of quasi-isolated elements in \cite{bon}. Namely, we have $o(s)=m$ divides $n$ and   
 $C:=\cent_{\wt{G}^\ast}(\wt{s})\cong \GL_{n/m}((\epsilon q)^a)^{m/a}$ for some $a$ dividing $m$. 
 
 Then if the Sylow $2$-subgroup of $G^\circ(s)/Z$ for some $Z\leq \zent(G)$ has a unique involution (i.e. is quaternion or cyclic), 
we know $C\cong \GL_1((\epsilon q)^n)$ with $o(s)=n$ or $\GL_2((\epsilon q)^{n/2})$ with $o(s)=n/2$. (Recall that $\GL_1((\epsilon q)^{n/2})^2$ cannot occur since $o(s)$ is odd.)

If we are in the above situation and $C=\GL_1((\epsilon q)^n)$, then note that $|C/\zent(\wt G^\ast)|$ is odd and $\CO G b$ is defect zero.

If instead $C=\GL_2((\epsilon q)^{n/2})$, then $|C/\zent(\wt{G}^\ast)|_2=(q^{n/2}-\epsilon^{n/2})_2(q^n-\epsilon^n)_2/(q-\epsilon)_2=(q^n-\epsilon^n)_2=(q^2-1)_2$ since $n/2$ is odd. Here $D$ (which we recall is a Sylow $2$-subgroup of $G^\circ(s)$) is isomorphic to a Sylow $2$-subgroup of $\SL_2(q^{n/2})$, and hence quaternion. 
In particular, the case that $D/Z$ is nontrivial cyclic does not occur, as then $D$ would be abelian.

Note that if $n/2=1$, then $C=\GL_2(\epsilon q)=\wt{G}^\ast$ (i.e. $\wt{s}\in \zent(\wt{G}^\ast)$) and $\Irr(\CO\wt G \wt b) =\Irr(\CO \wt{G} b_0(\wt{G}))\otimes \hat{\wt{s}}$, and this recovers the case that $\CO G b=\CO G b_0(G)$ is the principal block of $\CO\SL_2(q)$.

\subsection{Reducing to Quasi-Isolated Blocks}\label{sec:reducetoqi}
We return to the more general situation of Section \ref{sec:setup}, and now suppose $\bH^\ast$ is an $F$-stable Levi subgroup of $\bG^\ast$, minimal such that it contains $\cent_{{\bG^\ast}^{F^\ast}}(s)\cent_{\bG^\ast}^\circ(s)$. We argue here similar to Subsection \ref{sec:typeAqi}, but now using \cite{BR} to reduce to the case of quasi-isolated blocks (rather than isolated blocks).  Note that by definition, $s$ is then (strictly) quasi-isolated in $\bH^\ast$. Let $\bH$ be an $F$-stable Levi subgroup of $\bG$ dual to $\bH^\ast$. Then the main result of \cite{BR} yields that $\CO G b$ is Morita equivalent to a block $\CO \bH^F c$ of $\bH^F$. Further, these blocks have the same defect group, say $D=D(b)=D(c)$, by \cite[Thm.~7.7]{BDR}, \cite[Thm.~2]{Ruh20}, and \cite[Thm.~7.14]{KM13}. 

Note that $[\bH, \bH]$ is semisimple and simply-connected by \cite[Props.~6.20(c),~12.14]{MT}. Then by \cite[Cor.~1.5.16]{GM20}, $H':=[\bH, \bH]^F$ is isomorphic to a direct product $\prod_{i=1}^k \bH_i^{F_i}$, where each $\bH_i$ is simple, simply connected and $F_i$ is some power of $F$. 

In this situation, $\CO \bH^F c$ covers a quasi-isolated block of $\CO H'$, where we define $H':=[\bH, \bH]^F$. Then $D$ contains a defect group of such a block. Let $\CO H'c'$ be a (quasi-isolated) block of $\CO H'$ below $\CO \bH^F c$, and assume $D(b)=D(c)$ is quaternion. Then $\CO H' c'=\bigotimes \CO H_ic_i$ where each $\CO H_ic_i$ is a quasi-isolated block of the direct factors $H_i:=\bH_i^{F_i}$ of $H'$. Then at most one of these factors has positive defect (as otherwise $D(c')$ and hence $D(c)$  contains multiple involutions). Without loss, say this is $\CO H_1$, with block $\CO H_1c_1$ and defect group $D_1$. Note then that $D_1$ contains a unique involution.

Then by Section \ref{sec:typeAqi} and Lemmas  \ref{lem:classicalqi} and \ref{lem:exceptionalqi}, we have $H_1$ is $\SL_n(\epsilon q)$ for some $q$ and $\CO H_1c_1$ must be as described in Section \ref{sec:strictlyqi}. Further, from these we also see that if $k>1$, then each $H_i$ for $2\leq i\leq k$ is of exceptional type $X(q_i)$ and the corresponding block $\CO H_ic_i$ is a defect-zero block corresponding to a $d_i$-cuspidal unipotent character, where $d_i:=d_2(q_i)$, or is of type $\type{A}_{n-1}(\epsilon q_i)$ with $C=\GL_1(\epsilon q_i^n)$ as discussed in Section \ref{sec:strictlyqi}. Further, if $\bG\neq \bH$, then Lemma \ref{lem:exceptionalqi} gives that either each of these $H_i$ for $2\leq i\leq k$ is of the latter type $\tA$ form, or $X=\tE_6$ or $\tw{2}\tE_6$ and therefore $G\in\{\tE_7(q), \tE_8(q)\}$.

Now, if $K$ is quasi-simple of the  form $G/Z$ for some $Z\leq \zent(G)$ (which we may assume by \cite[Thm.~9.9c]{N98} is a $2$-group), then note that $Z\leq \bH^F$ and that $DZ/Z$ is defect group for a block $\CO[G/Z]\bar b$ dominated by $\CO G b$. As $Z$ is central, further any block has such a defect group (see e.g. \cite[Lem.~17.2]{CaEnbook}). Then again from Lemmas \ref{lem:classicalqi} and \ref{lem:exceptionalqi} and Section \ref{sec:typeAqi}, we have $H_i$ and $\CO H_ic_i$ are of the form in the previous paragraph. From this we see:

\begin{Proposition}\label{prop:structurequasisimplegenquat}
     Let $Z\leq \zent(G)$ be a central $2$-group and let $\CO[G/Z]\bar b$ be a $2$-block of $\CO[G/Z]$ with  quaternion defect groups. Let $\CO[\bH^F/Z]\bar c$ be the block of $\CO[\bH^F/Z]$ (with the notation above) such that the corresponding blocks $\CO G b$ and $\CO \bH^F c$ of $\CO G$ and $\CO \bH^F$ are Bonnaf{\'e}--Rouquier correspondents. Then 
    \begin{enumerate}[(i)]
        \item $[\bH, \bH]$ has exactly one $F$-simple component $\bH_1$ such that a block $\CO H_1 c_1$ of $\CO H_1$, where $H_1=\bH_1^F$, covered by $\CO \bH^F c$ has nontrivial defect groups. Here $\bH_1^F\cong \SL_n(\epsilon q_1)$ for some $n$ such that $n\equiv 2\pmod 4$ and some power $q_1$ of $p$, and $\CO H_1 c_1$ has quaternion defect groups and is as described in Section \ref{sec:strictlyqi}.
        \item If $G\neq \tE_8(q)$, then the remaining $F$-simple components of $[\bH, \bH]$ are again of type $\tA$ and contribute defect-zero blocks as in Section \ref{sec:strictlyqi}.
        \item If $G=\tE_8(q)$, then either (ii) still holds or $\bH$ is type  $\tA_1.\tE_6$. Here the $\tE_6(\epsilon q_2)$-component of $[\bH, \bH]^F$ contributes a defect-zero block containing a $d_2(q_2)$-cuspidal unipotent character of $\tE_6(\epsilon q_2)$.
    \end{enumerate}\end{Proposition}

\section{Groups of Lie type as index-$2$ subgroups}\label{sec:AD} 
Throughout this section,  $q$  will denote an odd prime power, and $\epsilon \in \{\pm 1\} $.   Recall that if $\epsilon =1 $, then   $\GL_n (\epsilon q)$  is   the general linear group of  an $n$-dimensional vector space over ${\mathbb F}_{q} $  and  if $\epsilon =-1$, then $\GL_n(\epsilon q)= \GU_n(q)$ is the  isometry group of a non-degenerate  unitary form on  an $n$-dimensional vector space 
over ${\mathbb F}_{q^2} $. Recall that $\SU_2(q)  \cong  \SL_2(q)$. By an abuse of notation, we  will use  $\SL_2(q)$  to  denote the  commutator subgroup  of  $ \GL_2(\epsilon q) $.

The  aim of this section is to prove \thref{lietype}, which studies the case of  groups of Lie type in the context of \thref{main-reduction}.  Throughout this section,  as in Section~\ref {sec:lie-non}, $\bG$ will denote a  simple, simply-connected algebraic group with a Frobenius morphism $F$  with respect to an ${\mathbb F}_q$-structure.

\subsection{Reducing to types $\type{A}_n$ and $\type{D}_n$}
We first further refine \thref{main-reduction} as a consequence of \thref{lem:exceptionalqi}.
 \begin{Corollary}\thlabel{cor:notB_nC_nF4G2E8}   Let $(H,c)$ be a Morita minimal pair such that $H$ is a finite group and $c$ is a block idempotent of $\CO  H$ with defect group $P \cong Q_{2^n}$ for some $n \geq 3$, and such that $H$ is either quasisimple or there is $N\lhd H$ with $N$ quasisimple, $[H:N]=2$, $H=NP$, and $N\cap P\cong Q_{2^{n-1}}$ (that is, $H$ and  $c$ are as in parts (i) or (ii)  of 
\thref{main-reduction}). In the first case,  set  $ N:=H$.  Suppose that  $N= G/Z$, where $G=\bG^F$  and $Z \leq  Z(G)$.
If $H=N$,  then  $\bG$ is of type $\type{A}_n$, and  $\OH c $ is Morita equivalent to a principal block.  If   $ [H:N]=2 $,  then  $\bG$ is  of type $\type{A}_n$, $\type{D}_n$ or $\type{E}_6$.
\end{Corollary}
\begin{proof} If  $H=N$, then  the first  assertion  is immediate from  \thref{cor:minisquasi-isolated}, \thref{lem:classicalqi}  and \thref{lem:exceptionalqi}. The second assertion then follows from  \thref{cor:extendclassical2}.
Now suppose that    $N$ is of index $2$ in $H$. Let  $H=N \langle  x \rangle $,  for  $x$ a  $2$-power element.   If $\bG$ is  not of type  $\type{A}_n$, $\type{D}_n$ or $\type{E}_6$,  then  $x$  acts   as the composition  of a diagonal automorphism  with $F_0$, where  $F_0$ is as in \thref{prop:field} (see \cite[Thms.~2.5.1,~2.5.12]{GLS3}).  Again, the result follows  from \thref{cor:minisquasi-isolated}, \thref{lem:classicalqi}  and \thref{lem:exceptionalqi}.
 \end{proof}

\begin{Proposition} \thlabel{E6}  Let $G= \bG^F =  \type{E}_{6,\mathrm{sc}}(\epsilon q) $,   $Z \leq  Z(G)$, $ N= G/Z$, and  $H=N\langle x\rangle$, with   $[H:N]=2$.  
Let $c$  be a  block  idempotent  of    $H$ with  a  defect group $Q_{2^n} $, $ n\geq 4 $  and  covering a  stable block of $N$. Then $\CO Hc$ is nilpotent.
\end{Proposition}

\begin{proof} Let $P$ be a defect group of    $\CO Hc $.   Then $P\cap N  \cong Q_{2^{n-1}}$  is a defect group   of $\CO Nc$  and   $P/P\cap N \cong H/N$. In particular, we may assume that $P$ contains  $x $.    
  Let $u \in P\cap N$ be the unique involution  of $P$ and note  that since $x \in P$,  $C_H(u)  =  C_N(u)  \langle x\rangle $. Let $d$ be a block of $C_H(u)$  such that $(\langle u \rangle, d )$ is  a $c$-Brauer pair. Then $P$  is a defect group of $d$.  Let  $e$  be a  block of  $\CO C_N(u) $  covered by $d$.  Since  $ C_H(u) =C_N(u) P$,  $e$ is  $C_H(u)$-stable (and hence $e=d$)  and  $ P \cap  N$  is a defect group of $\CO C_N(u) e$. 
     Since $Z$   is a $3$-group,   $u$  lifts to an involution, say  $\tilde u \in G$    and  $C_N(u)  =C_G(\tilde u)/Z $.
  
 Set $\bC=C_{\bG}(\tilde u)$.  By \cite[Tab.~4.5.2]{GLS3} (see also the tables available at \cite{Luebeck}), there are two $G$-classes of involutions in $G$.
Suppose first that  $\bC= [\bC,\bC] =\bL_1 \bL_2$,  where $\bL_1 $ is of simply-connected type $\type{A}_5$ and $ \bL_2 $  is  simply connected of type $\type{A}_1$, $\bL_1 \cap \bL_2 \cong C_2$,  $ Z \leq  \bL_1$, and  $\bL_1^F\bL_2^F$ is  of index $2$ in $\bC^F$.  We claim that  $\bL_1^F\bL_2^F$ is characteristic  in $\bC^F$.  Indeed, we have $$[\bL_1^F, \bL_1^F]  [\bL_2^F, \bL_2^F] \leq [\bC^F, \bC^F] \leq \bL_1^F \bL_2^F \leq \bC^F. $$   Further, $\bL_1^F$ is perfect and unless  $\bL_2^F \cong \SL_2(3)$ (which  only happens if  $\epsilon q = 3$), so is $\bL_2^F$.  If $\bL_2^F \cong \SL_2(3)$, then $[\bL_2^F, \bL_2^F]$ is of index $3$ in $\bL_2^F$  and consequently,  since  $\bL_1^F $ and $\bL_2^F$ intersect in a group of order $2$,  $[\bL_1^F, \bL_1^F][\bL_2^F, \bL_2^F] =\bL_1^F [\bL_2^F, \bL_2^F]$ is of index $3$  in $\bL_1^F\bL_2^F$.   Thus, either  $ \bL_1^F \bL_2^F =[\bC^F, \bC^F]$ or  $[\bC^F, \bC^F]= \bL_1^F [\bL_2^F, \bL_2^F]  $ is of index $3$  in $\bL_1^F \bL_2^F$. Since groups of order $6$ have  a unique Sylow $3$-subgroup, the claim follows.  Further, since $\bC^F$ acts as  diagonal automorphisms on  each of $\bL_1^F$ and $\bL_2^F$, every $2$-block  of  $\bL_1^F\bL_2^F$  is  $\bC^F$-stable by \thref{lem-primediagonalstable}.   It follows that $\bL_1^F\bL_2^F/Z$ is a characteristic  subgroup of  index $2$ of $\bC^F/Z=C_N(u)$ and  that every block of  $\bL_1^F \bL_2^F/Z$ is $C_N(u)$-stable. 
In particular, $\bL_1^F \bL_2^F$ is  normal  of index $4$ in $C_H(u)$. Let $e_0$ be the unique  block of  $\CO \bL_1^F \bL_2^F/Z$ covered by $e$ (so $e_0=e$). Then  $P \cap \bL_1^F \bL_2^F/Z $ is a  defect group of $\CO (\bL_1^F \bL_2^F)e_0$ and since $e=e_0$,  $P \cap  C_N(u)$ is not contained in $\bL_1^F\bL_2^F/Z$. It follows  that  $P \cap \bL_1^F \bL_2^F/Z $ is  a normal subgroup of index  $4$ of   $P$.  Thus by \thref{normal_subgroup_quaternion}  $P \cap   \bL_1^F \bL_2^F/Z  $  is cyclic. This means that  $\CO (\bL_1^F \bL_2^F/Z )e_0$ is  nilpotent, and since  $C_H(u) /(\bL_1^F \bL_2^F/Z ) $ is   a  $2$-group, it follows that   $\CO C_H(u)d$ is nilpotent, and hence by \thref{G_and_C_G(z)} $ \CO Hc$  is  nilpotent.

Next suppose that $[\bC,\bC]$  is simple  of  simply connected type $\type{D}_5$.  Then  $\bC = [\bC, \bC] Z^{\circ}(\bC)$,     $[\bC, \bC]^F \cong \type{D}_{5,\mathrm{sc}}(q)$, $ Z^{\circ} (\bC)^F  \cong  C_{q-\epsilon}$  and  
$[\bC, \bC]^F \cap Z^{\circ}(\bC)^F \cong  C_{\mathrm{gcd} (q-\epsilon, 4)}  \cong    \bC^F/ ([\bC, \bC]^F Z^{\circ}(\bC)^F)$.
Further, since  the block of $\CO \bC^F$   dominating $\CO C_N(u) e$   has  quaternion defect groups and $Z^\circ(\bC)^F$ is central in $\bC^F$,  it follows that $q\equiv \epsilon 3 \pmod 4$ and hence that   $([\bC^F , \bC^F]Z^{\circ} (\bC)^F )/[\bC, \bC]^F$ is the unique  subgroup of index $2$  in 
$\bC^F/[\bC, \bC]^F$. Since  the simple-connectedness of $[\bC, \bC]$  forces $[\bC, \bC]^F =[\bC^F, \bC^F] $, it follows that     $[\bC, \bC]^F Z^{\circ}(\bC)^F  $  is a characteristic subgroup of $\bC^F$ of index $2$. Again  since $\bC^F$  acts by diagonal automorphisms on  $[\bC, \bC]^F $ every block of  $[\bC, \bC]^F Z^{\circ}(\bC)^F$ is  $\bC^F$-stable, again by \thref{lem-primediagonalstable}.  Now argue as in  the previous case.
\end{proof}

 \thref{cor:notB_nC_nF4G2E8} and \thref{E6}  show that if $H$  and $c$ are as  in  \thref{main-reduction}(ii),  and if  $[H, H]$   is  of Lie type in odd characteristic  and not an  exceptional cover, then  $[H,H]$ is  of type $\type{A}_n$ or $\type{D}_n$. The remainder of this section will  study these  remaining two cases.

\subsection {On quaternion and dihedral subgroups} 
We will make repeated use of \thref{normal_subgroup_quaternion}. 

\begin{Lemma} \label{lem:inbetween}  Let $\SL_2(q) \leq  H \leq  \GL_2(\epsilon q) $, $q$ an odd prime power and let $Z =\langle -I \rangle$. The Sylow  $2$-subgroups  of $H$ are not dihedral and  the Sylow $2$-subgroups of  $H/Z$ are  not quaternion.
If $[H: \SL_2 (q)]_2 \geq 2 $, then $H$ does not have quaternion  Sylow $2$-subgroups  and if   $|H: \SL_2 (q)|_2 \geq 4 $,  then  the  Sylow $2$-subgroups of  $H/\langle -I \rangle $    are not  dihedral.
\end{Lemma}
\begin{proof}   Let $Q$    be a   Sylow $2$-subgroup of  $\SL_2(q) $  and $P$ a Sylow $2$-subgroup of $H$ containing $Q$.  Then   $Q$ is quaternion and  $Q/Z$ is dihedral, so $P$  is not dihedral and $P/Z$ is not quaternion.  
Since  $\GL_2(\epsilon q) $ is a semi-direct product of  $\SL_2(q)$  with  a cyclic group of order  $(q-\epsilon) $, if $[H: \SL_2 (q)]_2 \geq 2 $, then  $P\setminus Q$ contains an involution, whence $P$ is not quaternion.  Since $Q$ is quaternion, $P$ is clearly not dihedral.   Now,   suppose $[H: \SL_2 (q)]_2 \geq 4 $.  Then  $P/Z$ is   of order at least  $16$ and  $Q/Z$ is a normal dihedral  subgroup of  $P/Z$  of   index at least $4$. Hence $P/Z$ is not dihedral.
\end{proof}

\begin{Lemma}\label{lem:noGL2}
Let $ X$ be  a  finite group and $Z$ a  central subgroup of  $X$ with Sylow $2$-subgroup $Z_2$. Suppose that $X/Z$ has quaternion Sylow $2$-subgroups.
\begin{enumerate}
    \item [(i)] Suppose that $H$  is a normal subgroup of  $X$   containing   a subgroup  $H_0$  which  is isomorphic to   $\SL_2(q)$  for some odd prime power $q$. Then $H_0\cap Z_2 =1$  and $|X/HZ|_2\leq 2$.
    \item [(ii)]$X$ does not contain  a  subgroup isomorphic   to $\GL_2(\epsilon q) $ for any odd prime power $q$.
\end{enumerate}
\end{Lemma}

\begin{proof}  
 Suppose that $H_0 \cap Z_2 \ne 1$. Then  $H_0 \cap Z_2 =Z(H_0) $  and  consequently $X/Z$  contains a subgroup  isomorphic to $\PSL_2(q)$, a  contradiction. This proves the first assertion of (i). The second assertion follows  from the fact that  a  Sylow $2$-subgroup of  $HZ/Z$  is non-abelian and  normal  in a Sylow $2$-subgroup of $X/Z$.    

 In order to prove (ii), we may assume that $X=\GL_2(\epsilon q) $. Since $\SL_2(q)$  contains the unique involution  of  $Z(X)$,  by  part (i)  applied with $H=X$, we have that $Z_2 =1$. Then  the  image of   the subgroup of diagonal matrices  of $X$ in $X/Z$ contains a   Klein $4$-group, a contradiction.
 \end{proof}

\begin{Lemma}\label{lem:noSL2} Let $X, Y $ be  finite groups and  let $M$ be a normal  subgroup of  $X\times Y $ with $(X \times Y)/M$  cyclic.  Let $P_X $  be a  Sylow $2$-subgroup  of $X$, $P_Y $  a  Sylow $2$-subgroup  of $Y$, and let $P  =(P_X\times P_Y) \cap M$, a Sylow $2$-subgroup of $M$.  Let $Z \leq  Z(X \times Y) \cap M $  be a cyclic $2$-group.    Set $Q_X=  P_X\cap   P $, $Q_Y = P_Y \cap P$,   let  $Z_X $ be the projection  into   $X$ of $ Z$ and let   $Z_Y $ be the projection  into   $Y$ of $ Z$. 
Suppose that $P/Z$ is dihedral or quaternion  and let $q$ be an odd prime power.   Then the following holds.
\begin{enumerate} [(a)]\item  $M$  does not involve any of  the groups   $\PSL_n(q)$,  $n\geq 3 $,  $\PSU_n(q) $, $n \geq 3 $, $\PSp_{2n}(q)$, $n\geq 2$, $\PO_{2n+1}(q)$, $n \geq  2$, $\PO_{2n}^{\pm}(q)$, $n\geq 3$, or  $\PO_4^{+} (q)$.  
\item      Suppose that $M \cap X $ contains  $\SL_2 (q)$,  $\PGL_2(q)$  or $\PSL_2(q) $.  Then $ |P/Q_XZ| \leq 2 $,   $Q_Y =Q_XZ\cap  Z_Y$,  $P_Y/Q_Y$ is cyclic and $P_Y $ is  abelian of rank at most $2$.    
\item  Suppose that   $X$ contains  a  subgroup $H$  isomorphic to   $\GL_2 (\epsilon q) $  such that 
  $Z \cap  H =1 $.   Then $P/Z$ is quaternion,   $P_Y$ is  cyclic   and $ (q-\epsilon)_2 \leq  [P_X \times  P_Y: P]$. If    $Z=1$, then $P_Y=1 $.   Moreover, $H \cap M  $  contains $ \SL_2(q) $ as  a subgroup of odd index.
  \item   Suppose that  $X$ contains  a  subgroup $H$  isomorphic to   $\GL_2 (\epsilon q) $  such that   $Z \cap  H \ne 1 $. Then  $P/Z$   is dihedral.
 \end{enumerate}
\end{Lemma}

\begin{proof} Since any section of  $P/Z$ is  cyclic, quaternion, or  dihedral, the only non-abelian  composition factors of $M$ are two-dimensional  projective special linear groups or $A_7$. This  rules out all groups listed in (a), except  $\PO_4^{+} (q)$.  Now $\PO_4^{+} (q)$  is isomorphic to $\PSL_2(q) \times \PSL_2(q) $  (see \cite[Prop. 2.9.1]{KlLi}), hence it is ruled out by rank considerations. This proves (a).

The Sylow $2$-subgroups of $\SL_2 (q)$, $\PGL_2(q)$   and  $\PSL_2(q) $  are  non-cyclic. Hence $Q_XZ/Z$ is  non-cyclic and  normal in $P/Z$.  Thus,  the index of     $Q_XZ$  in $P$  is at most $2$  and  $Q_XZ/Z $ is    self-centralising in $P/Z$.  Since   $Q_YZ/Z$ centralises  $Q_XZ/Z$,  $Q_Y \leq  Q_X Z  \cap  Y  = Q_X Z \cap   Z_Y$.   Since $X\times Y/M$ is cyclic, $P_Y/Q_Y $ is  cyclic.    Since  $Q_Y \leq Z_Y$  and $Z_Y$  is  a cyclic  subgroup of    $Z(Y)$, it follows that  $P_Y/Q_Y $ is   abelian of rank at most $2$.  This proves (b).

Now  let $ H \leq  X $ be   isomorphic to $\GL_2(\epsilon q)$  and  suppose   $Z \cap H=1 $.    Since  $H/H\cap N  $  is cyclic,  $H\cap  N$  contains  $[H,H] =\SL_2(q)$. 
Since    $Z \cap H =1 $,  $P/Z$  contains a quaternion subgroup, hence is  itself  quaternion. Since $Z$ is  cyclic, $P$ is of rank at most  $2$ and  since   $(P_X\times P_Y)/P$ is cyclic, the rank of  $P_X \times P_Y$ is at most  $3$.  But since  $X$ contains  $\GL_2(\epsilon q) $, the  rank of $P_X$ is at  least $2$  and by part (b), $P_X $ is abelian.  Hence,  $P_Y$ is  cyclic.  Further, if   $Z=1$, then the rank  of $P_X \times P_Y$ is at most  $2$, hence $P_Y=1 $. Now by Lemma \ref{lem:inbetween},  $ \SL_2(q)$  has odd index in $H\cap M$.    Thus,   $ (q-\epsilon)_2 \leq     [H :H/ H\cap M]_2 \leq  |(X\times Y)/M |_2 =[P_X \times  P_Y: P]$.

Finally, suppose  that   $Z \cap H\neq1 $.   Then   $Z \cap  [H,H]  =Z([H,H])  \ne 1 $. Hence  $P/Z $  contains a dihedral   subgroup  and is therefore   itself dihedral. This proves (d).\end{proof}

\begin{Lemma} \label{lem:nofield} Suppose that $q=q_0^2$ is an odd prime power and  $ X=\GL_n(\bar{\mathbb F}_q)$.   Let $\sigma  $ be  the   automorphism of  $X$  raising every matrix entry to the  $q_0$-th power  and let $\tau $ be  the transpose inverse automorphism of $X$.  Suppose that   $x \in  X$  is  of order $ 2^a:=(q-1)_2 \geq 8$,  let $\Upsilon $ be  the multiset of eigenvalues of $x$  and let  $\Upsilon' $ be the  set of distinct eigenvalues of $x$ of order $2^a$.  
\begin{enumerate} [(a)]  \item If  $\Upsilon' =\{\lambda, \lambda^{-1} \}$  for some  $\lambda $,  then  neither $\sigma (x) $  nor  $\sigma\tau (x)$ is  $X$-conjugate to  $x$.  If  in addition  $1 \in \Upsilon $ but $-1 \notin 
\Upsilon $, then neither $\sigma (x) $  nor  $\sigma\tau (x)$ is  $X$-conjugate to  $-x$.
\item Let $Q$  be a non-cyclic $2$-subgroup of $X$ containing $x$ and let  $P =Q \langle t \rangle $ be a $2$-group containing $Q$  as  a  subgroup of index  $2$. Suppose further that  there exists $ h \in \GL_n(\bar {\mathbb F}_q)$ such that  $t$  acts as  $c_h \circ \sigma $  (respectively  $c_h \circ \sigma \tau$)  on $Q$. If $P$ is  dihedral or quaternion, then   $\sigma (x)  $  (respectively $\sigma \tau (x) $)   is $X$-conjugate to $x$. If  $ -I \in Q$  and   $P/  \langle  -I\rangle $  is dihedral or quaternion, then $\sigma (x)  $  (respectively $\sigma \tau (x) $)   is $X$-conjugate to $x$ or to $-x$.
\end{enumerate} 
\end{Lemma}

\begin{proof}  Suppose that the set of distinct eigenvalues of $x$ of order $2^a$    is   $\{\lambda, \lambda^{-1} \}$.
Then the eigenvalues  of $\sigma (x) $  and of $\sigma \tau (x)$  of order $2^a$ are $\lambda^{q_0} $ and  $\lambda^{-q_0} $.  Thus if either of these is $X$-conjugate to $x$, then either $\lambda^{q_0-1}=1$ or $\lambda^{q_0+1}=1$.  But this is impossible since $(q_0 \pm 1)_2 < (q-1)_2$.  The  second assertion of   (a) is  immediate  from that  fact  $-x$    has   $-1 $  as  an eigenvalue  whereas  $x$  does not. 

We now turn to (b). Suppose if possible that  $P$  is a quaternion or dihedral group and $t$  acts as $c_h \circ \sigma $.   
Since $Q$ is non-cyclic  and of index $2$  in $P$, the exponent of  $Q$ is  less than that of $P$.  
 Further,  $|P|  \geq  32 $, and in particular has a unique cyclic  subgroup of  index $2$. 
Let   $a \in Q $ such that $at$  generates this  cyclic  subgroup.   Any element of $P$ not contained in $\langle  at \rangle $ has order  $4$ if $P$ is quaternion and has order $2$ if $P$ is dihedral. In particular $x \in  \langle  at \rangle$,  hence $at$  centralises $x$.   But this  means that  $\sigma (x)$ and  $x$  are $\GL_n( \bar {\mathbb F}_q)$-conjugate.  The case that  $t$  acts as $c_h \circ \sigma \tau $ is entirely analogous.   
\end{proof}

\subsection{Type $\type{A}$} In this subsection let $G= \SL_n(\epsilon q)$,   $\wt{G}= \GL_n (\epsilon q) $ and let $Z$ be a central subgroup of  $G$.  Let  $s \in \wt{G}$ be a semi-simple element,  let $\Lambda $  be  the  set of the distinct  $\bar{\mathbb F}_q $ eigenvalues of $s$  and let
$\tilde \Lambda $  be the  multiset of eigenvalues  of $s$ counting multiplicities.  Then  $\tilde \Lambda ^{\epsilon q} = \tilde \Lambda $.
 We have \begin{equation}\label{eq:centdecomp} C_{\wt{G}}(s) =  \prod_{i=1}^r  \GL_{n_i}(\epsilon_i q^{d_i}) \end{equation} for  positive  integers   $ n_i $, $ d_i $, $ 1\leq i \leq  r $ with $\sum_{i=1}^r n_id_i = n $.  The factors of  the  direct product decomposition are in bijection with   the orbits  of  $\Lambda $ under the  $\lambda \to \lambda^{\epsilon q} $;
  $d_i$ is the length of the orbit and $n_i$  is the   the multiplicity of  some (and hence any)  element in the orbit. If $\epsilon=1 $, then all $\epsilon_i=1 $ and if $\epsilon=-1 $, then $\epsilon_i=-1 $ if and only if $d_i$ is odd. Further, if $x_i \in  \GL_{n_i}(\epsilon_i q^{d_i})  $ is  semisimple  and $ \Upsilon_i $ is the multiset  of eigenvalues of  $x_i$ as an element of   $\GL_{n_i}( \bar {\mathbb F}_q) $  that are different from $1$, then the  multiset of  eigenvalues of  $x_i$ different from $1$  as an element of $\wt{G}$ is the  union  of   $\Upsilon_i ^{ (\epsilon q)^j } $, $ 0\leq  j \leq d_i -1 $.

\begin{Lemma}  \label{lem-centralSylow2}  Let $s \in \wt{G}$ be  a semi-simple element of odd order,  $S$   a  Sylow  $2$-subgroup  of $C_{\wt{G}}(s)$, $P = G\cap  S$  a  Sylow $2$-subgroup of   $C_{\wt{G}}(s) \cap  G $   and $ A =  PZ/Z$,  a Sylow  $2$-subgroup  of $ (C_{\wt{G}}(s) \cap G )/Z$. Suppose that  $A$  is  a quaternion group.  
  Then  one of the following holds.
\begin{enumerate} \item  [(i)]
$n=2d $  for  $d$ an odd   integer, $ C_{\wt{G}}(s) \cong \GL_2 (\epsilon q^d) $ and $Z_2=1 $.  If  $s^{-1}$ is   $\wt{G}$-conjugate to $s$, then  $n=2$ and $s=1 $. If   $q =q_0^2$, and  $ x \in P $  is  of order $\frac{1}{2}|A|$, then the eigenvalues of  $x$     (as an element of 
$\GL_n(\bar {\mathbb {F}}_q) $) are  $\lambda $ and $\lambda^{-1}$   for some   $\lambda \in \bar {\mathbb {F}}_q^{\times} $,  each occurring with  mutiplicity  $d$.
\item [(ii)]  $|A|\geq 16$,   $ Z_2 =Z(\wt{G})_2 $,  $ |Z_2| \geq 4$,  $C_{\wt{G}}(s)  \cong  \GL_{2}(\epsilon q^{d}) \times \GL_1(q^{2e}) $ 
 for odd  positive integers $d, e $ with  $ n = 2(d + e)$. Moreover,   the full inverse images  in $G$ of the two non-abelian   subgroups of $A$ of  index $2$  have different exponents.  
\end{enumerate}
\end{Lemma}

\begin{proof} First of all note that we may assume without loss that $Z=Z_2$. 
  If $r \geq 2$, then $Z$ intersects each factor  in the decomposition \eqref{eq:centdecomp} trivially. Further, no  non-trivial factor is an odd-order group, and any factor has a cyclic Sylow $2$-subgroup  if and only if the factor itself is cyclic. Hence, by  Lemma~\ref{lem:noSL2} (c),  either  $Z=1 $,  $r=1$ and $n_1=2 $  or  $Z\ne 1,  r=2, n_1=2, n_2=1 $. 
   
 Let us  first consider the case  $Z=1 $,  $r=1$ and $n_1=2 $ and note that $|A|=P$ in this case.   So, $C_{\wt{G}}(s)\cong \GL_2 (\epsilon_1 q^{d_1}) $. By Lemma~\ref{lem:noSL2} (c), applied with  $X= C_{\wt{G}}(s)$, $Y=1 $, and  $C_{\wt{G}} (s) \cap G$ in place of  $M$, $(q^{d_1} -1)_2 \leq  (q -\epsilon)  $, whence $d_1$ is odd and $\epsilon_1=\epsilon$.  Now suppose that  $s^{-1}$ is $\wt{G}$-conjugate to $s$. Then  $\Lambda $ is invariant under taking inverses.   Since $|\Lambda|$  is odd, there exists $\lambda \in \Lambda$  with $\lambda= \lambda^{-1}$.  Then $\lambda^2=1$, and since $s$ is  odd order, this means that $\lambda =1 $.   Since all elements of $\Lambda$ are  powers of $\lambda $, it follows that $\Lambda =\{1\}$,  $s=1 $ and $n=2$.  Let $x \in P $  with   order $\frac{1}{2}  |P|  = \frac{q^2-1}{2} \geq 8$. By Lemma~\ref{lem:noSL2},  $x \in \SL_2(q^{d_1})$. So the  eigenvalues of  $x$ as  an element of  $\GL_2(\bar {\mathbb F}_q)$ are $\lambda, \lambda^{-1}$ each occurring with  multiplicity  $1$.  Since $q=q_0^2 $, $ \frac{q^2-1}{2}  = (q-1)_2  $  and in particular $\{ \lambda, \lambda^{-1} \}   =  
  \{ \lambda, \lambda^{-1} \} ^{\epsilon q}$.   Thus,   (i) holds.

 Now suppose that $Z\ne 1,  r=2, n_1=2, n_2=1 $.  Since  $Z \ne1 $, $n =2d_1 +d_2$ is even, hence $d_2 $ is even and 
 $C_{\wt{G}}(s) \cap G  =  X \times Y $, where  $ \GL_2 (\epsilon_1 q^{d_1}) $  and $Y=  \GL_2 (q^{d_2})$.  Since   $Z$ intersects  $X$ trivially,    again by Lemma~\ref{lem:noSL2},  $d_1 $ is odd, and $\epsilon_1= \epsilon$.  Let $Q_X= P  \cap  \GL_2 (\epsilon_1 q^{d_1})$.  By Lemma~\ref{lem:noSL2}, $Q_X$ is a Sylow $2$-subgroup of  $\SL_2 (q^{d_1})  $. Since $d_1$ is odd,  $|Q_X| =(q^2-1)_2 $. On the other hand, by Lemma~\ref{lem:noSL2},  $[P: Q_XZ] \leq 2 $.    Since $|Z| \leq  (gcd(q-\epsilon,  n) )_2$ and  the index of   $P$  in   $P_X \times P_Y $  is   $(q-\epsilon)_2 $,    the only possibility is  that $d_2  = 2e $ for some odd integer $e$,   $|Z| = (q-\epsilon)_2 =  n_2  $ and $QZ  $ is of index $2$ in $P$.  In particular, $n= 2(d_1+ e) $ is  divisible by $4$ and   $|P/Z| \geq 16 $.
 The projection of $P$ onto each of $P_X$ and $P_Y$ is surjective since the  restriction of the determinant map on $\wt{G}$  to any of the factors is surjective.   In  particular, there  exists $x \in X$ and $y \in Y$  such that $a :=xy \in P$   and $y$ has order $|P_Y| =  (q^{2e}-1)_2 =  (q^2-1)_2$.       On the other hand, the  exponent of  $Q $ is $ \frac{(q^{2}-1)_2}{2} $  and  since   $Z$ is cyclic of order $(q-\epsilon)_2 $, the exponent of $Q (Z_X \times  Z_Y)  $  is  $ \frac{(q^{2}-1)_2}{2} $.   Thus,
 every element of $aQZ $  has order   $ (q^2-1)_2 $  whereas  the exponent   of   $QZ  \leq Q (Z_X \times  Z_Y)$   is $\frac{(q^2-1)_2}{2} $.   Since $QZ/Z$ is one of the normal subgroups of index $2$ of $P/Z$, it follows that the inverse image of the  other one    has exponent  $ (q^{2}-1)_2 $. Thus (ii) holds.
\end{proof}

Henceforth, we identify   $G$ with $\bG^F$,  and $\wt{G} $  with  $\wt{\bG}^F$, where  $\bG$  is a simple, simply-connected group of type $\tA_{n-1}$   over  a field  of odd characteristic and $\iota\colon\bG\hookrightarrow\wt\bG$ is a regular embedding as discussed in Section \ref{sec:extendBDR}.   In particular, as  in Section \ref{sec:extendBDR}  the   diagonal automorphisms  of $G$  are the automorphisms corresponding to conjugation by elements of $\wt{G}$.   Further, as already noted in Section~\ref{sec:lie-non} (assuming  $G$ is quasi-simple, i.e.  $(n,q) \ne (2,3)$), every automorphism  of   $ G/Z$ lifts  uniquely to an automorphism of $G$ (see \cite[Thm.~2.5.14]{GLS3}).   

Let $\tau $  be the transpose-inverse automorphism of  $\wt{G}$  and if  $q=q_0^2$  and $\epsilon =1$, let $\sigma $ be the  field  automorphism   of $\wt{G}$   raising every matrix  entry  to the  $q_0$-th power. Then $\tau $ and  $\sigma$ restrict to  automorphisms of $G$.  

\begin{Corollary} \label{non-stable2}   Keep the notation of Lemma~\ref{lem-centralSylow2}  and set  $N=G/Z$. Assume $N$ is quasisimple.  Let $H = N  \langle t \rangle $ be  a group containing $N$ as a subgroup of index $2$, with  $t$ a $2$-element normalising  $A$ and  let $D= A\langle t \rangle $. Let $\varphi $ be   the automorphism lifting the conjugation  action of $t$  on $ N$.   Suppose that  $A$ and $D$  are quaternion groups. 
Then  $s$  is as in  part (i) of Lemma~\ref{lem-centralSylow2} and 
$\varphi$  is either a diagonal  automorphism  of $G$  or the composition of a diagonal automorphism with $\tau $.
\end{Corollary}

\begin{proof}  Since  $t$  normalises $A$, $\varphi $  normalises  the inverse image of $A$ in $G$.  
Suppose that  $s$  is as in (ii)  of Lemma~\ref{lem-centralSylow2}.    Since $D$ is quaternion, the two  quaternion subgroups of $A$   of index $2$  are  interchanged by $t$, hence their inverse images  in $A$ are interchanged by $\varphi $.  This is impossible  as the two  groups have different exponents. 

So we may assume that $s$ is as in (i) of  Lemma~\ref{lem-centralSylow2}.  By \cite[Thm.~2.5.14]{GLS3}, the image of $\varphi $ in $\Out(G)$  has order at most $2$. Thus, by  \cite[Thms.~2.5.12,~2.5.14]{GLS3}, there exists  $g \in \wt{G}$  such that $\varphi= c_g $ or  $c_g\tau$, or $\epsilon=1$, $q=q_0^2$ and $\varphi =c_g \sigma$ or $c_g \sigma\tau$.   
The result follows by Lemma~\ref{lem:nofield}.
\end{proof} 

Recall that the block idempotents   of $\CO \wt{G} $ are in one to one correspondence with the $\wt{G}$-conjugacy classes of  semi-simple elements of $\wt{G}$ of  odd order, such that if a block  idempotent $b$ corresponds to    the class of $s$, then   the Sylow $2$-subgroups  of $C_{\wt{G}}(s)$  are defect groups of $\CO \wt{G}b$. Further, if $c$  is a block idempotent of   $\CO G $  such that $bc \ne 0$, i.e.  such that $b$ covers $c$,   and $d$ is a block idempotent of $\CO  G/Z $ dominated by $c$, i.e. such that  $\bar c d \ne 0$, where $\bar c$ is the image  $c$ under the canonical surjection of $\CO G\to \CO G/Z$, then (up to replacing $s$ by a $\wt{G}$-conjugate) the Sylow $2$-subgroups  of $C_{\wt{G}}(s) \cap G $ are defect groups of  $\CO G c$ and the Sylow $2$-subgroups of $(C_{\wt{G}}(s) \cap G)/Z $ are  the defect subgroups of $\CO (G/Z) d$.

\begin{Lemma} \label{lem-notaustable}  Let $N=G/Z$  and assume that $N$ is quasisimple.
Suppose that   $b$ is a $\tau$-stable  block idempotent of $\CO \wt{G} $    and $c$  is a block idempotent of   $\CO G $  such that $bc \ne 0$  and $d$ a  block idempotent of $\CO  N $ dominated by  $c$. Let $s \in \wt{G} $  be a semisimple element  of  odd order  whose $\wt{G}$-class corresponds to $b$.  Suppose that   $\CO  N d$  has 
quaternion  defect groups   and  $s$ is as in (i) of Lemma~\ref{lem-centralSylow2}.Then   $n=2$  and $c $ is the principal  block idempotent.
\end{Lemma}

\begin{proof}   Since   $b$   is  $\tau$-stable,  $ s$ and $s^{-1}$  are 
$\wt{G}$-conjugate (see for instance \cite[Lem.~4.2]{SriVin}).  Then the result follows from Lemma~\ref{lem-centralSylow2}.
\end{proof}

\begin{Theorem}  \thlabel{thm:typeA} Let $H$ be  a group containing $ N:=G/Z$ as a subgroup of index $2$, and let $d$ be a $G$-stable block idempotent of $\CO  N $. Let $D$  be  a defect group of  the block $\CO Hd $ and   let  $D=\langle  D\cap  N, t\rangle $. 
Suppose that   both $D$ and $D\cap   N$  are quaternion groups.    Then either  $n=2 $ and $d$ is the principal block idempotent of  $\CO  N $ or $t$ acts by diagonal automorphisms on $ N$. 
\end{Theorem}

\begin{proof}   
Let $\varphi $ be   the automorphism lifting the conjugation  action of $t$  on $ N$.  Suppose that $\varphi$ is not diagonal.  Then  either  $\varphi = c_h  \tau $   or  $q $ is a square, $\epsilon = 1 $   and $\varphi= c_h \sigma $  or $\varphi=  c_h \sigma \tau $ for some $ h \in \wt{G}$.  
Let $d_u$ be the  block idempotent of $\CO G$  dominating $d$  and let $e$  be a block  idempotent of $\CO \wt{G}$   with $ed_u \ne 0 $.    Suppose that   $e$  corresponds to the  class of the  semi-simple element $s$ and let  $P$ be a  Sylow  $2$-subgroup of  $C_{\wt{G}}(s) $.   Then (up to replacing $s$ by a $\wt{G}$-conjugate) $P \cap  G $ is a defect group of $\CO G d_u $ and $P \cap G/Z  $ is a defect group of  $\CO Nd $  and  is consequently quaternion.   Thus by  Corollary  \ref{non-stable2},   $\varphi = c_h  \tau $.  By 
 Lemma \ref{lem-red-graph2}  there exists  a block  idempotent   $e'$  of $\CO \wt{G}$   with $e'd_u \ne 0 $ and such that $e'$ is $\tau$-stable.    Suppose that $e'$ corresponds to the  class of $s'$. Again by  Lemma~\ref{lem-centralSylow2}, $s'$ is as in (i) of Lemma~\ref{lem-centralSylow2}, hence by Lemma~\ref{lem-notaustable},    $n=2 $ and $d$ is the principal block idempotent.  So,   $G= N  = \SL_2 (\epsilon q) $.   In this case the  action of $\tau $  on $N$   is  inner (see \cite[p.~68]{GLS3} or \cite[Prop.~24.23]{MT}),  hence the Sylow $2$-subgroups of $H$ are not quaternion.
\end{proof}

\subsection{Type $\type{D}$} Suppose that $V$ is a  $2n$-dimensional vector space  over  ${\mathbb F}_q$, $q$ odd,  equipped with a non-degenerate bilinear form   with corresponding group of isometries $ \O(V) \leq \GL(V)$, $\SO(V):=\O(V) \cap \SL(V)$,  and  $\Omega (V):= [\O(V), \O(V)]$.    Let $\eta:= \eta(V)  $   equal $1$  if the bilinear form is of maximal Witt index and equal  $-1$   otherwise. Note that  $\SO(V)$ is of index $2$ in $\O(V)$ and  $\Omega (V) $ is index $2$ in $\SO(V)$;   
 $ Z(\O(V) )  =  Z(\SO(V))= \langle -I\rangle  $,  and   $-I  \in  \Omega(V) $ if and only if $q^n\equiv  \eta \pmod 4$ (see the discussion after \cite[Thm.~11.51]{Tay06}).

The centralisers of  semi-simple elements  are described in \cite[Sec.~1]{FoSri-classical}  and we  keep  close to that account.  Note that  the  definition  of $\eta(V)$ is  stated differently in  \cite{FoSri-classical}
but  the two notions coincide  for even-dimensional  spaces. Let $s\in  \SO(V)$  be a semi-simple element of odd order and let  $V_0 \leq  V$  be the $1$-eigenspace of   $s$.
Then $V_0$ is $2n_0$-dimensional  for some  non-negative integer  $n_0$ and 
$$C_{\O(V)}  (s)  =  \O(V_0) \times C_{\O(V_1)} (s_1), $$ where  $V_1 $ is the orthogonal 
complement of    $V_0$ in $V$ and  $s_1$ is the projection   of $s$  in $\GL(V_1) $.  Moreover, 
$$ C_{\SO(V_1)}(s_1) = C_{\O(V_1)} (s_1)  \cong  \prod_{i=1}^r\GL_{n_i}  (\epsilon_iq^{d_i})$$ and
$$ C_{\SO(V)}(s)   \cong   \SO(V_0) \times   \prod_{i=1}^r\GL_{n_i} (\epsilon_iq^{d_i})$$  for positive integers  $n_i$, $d_i$ such that  
\begin{equation} n=  n_0 +  \sum_{i=1}^{r} n_i d_i  {\ \text{ and } \ } \eta    =  \eta (V_0) \prod_{i=1}^r {\epsilon_i}^{n_i}.  \end{equation}
  We  note that  if  $ g  \in  \O(V) \setminus  \SO(V)$, then  the $\SO(V)$-conjugacy class of  $s $  is invariant under  conjugation by $g$ if and only if  $n_0 \ne 0$.  Note  also that for all   $i$ with $1\leq i \leq r$, we have $\SL_{n_i}(\epsilon_i q^{d_i}) \leq C_{\SO(V_i)}(s_i) \cap \Omega (V) $  and $ C_{\SO(V_i)}(s_i) \cap \Omega (V)   $ is  of index at  most $2$ in  $ C_{\SO(V_i)}(s_i)$. 

\begin{Lemma} \label{lem:Dsylow}   Suppose that $n \geq 4 $. Let  $Z \leq Z(\Omega (V) )$, let  $s$  be  a semi-simple element of $\SO(V)$ of odd order and let $P$  be  a Sylow $2$-subgroup of    $C_{\Omega(V)}(s)$.   Suppose that $P/Z$ is one of: quaternion; dihedral; cyclic of order $2$ with $ Z =\langle -I \rangle  $. Then one of the following holds.

\begin{enumerate} [(i)]\item  $P$ is abelian of order at most $8$, and all $n_i\leq 1 $.
\item  $n_0=2$,  $r=1 $, $n_1=1$,   $\eta(V_0)=  -1 $,  $4 \nmid q^{d_1} -\epsilon_1$, $Z=\langle -I \rangle $,  $P$  is neither quaternion nor dihedral  and $P/Z$ is dihedral. Moreover,  either $q$ is not a square or   $\sigma(s)$ is not $\GL_{2n}(\bar {\mathbb F}_q)$-conjugate to $s$. 

\item  $n_0=1$, $r=1$, $n_1=2$, $Z=\langle -I\rangle$, $P/Z$ is  quaternion of order at least $16$  and there exists a  normal subgroup  $Q$ of $P$ such that $QZ$ has  index $2$ in $P$, $Q \cap Z=1 $,  and $Q$ is quaternion   with the property that  no element of  $Q$ of order $4$ is  $\GL_{2n}(\bar {\mathbb F}_q)$-conjugate to  an element of  $P \setminus QZ$.  Moreover, if $q$ is a square, then  
there exists an element $x \in Q$ of order   $(q-1)_2$    such that $\sigma(x)$ is not $\GL_{2n}(\bar {\mathbb F}_q)$-conjugate to $x$ or $-x$.
\item  $n_0=0$, $r=2 $, $n_1=2$, $n_2 =1$, $Z=\langle -I\rangle$, $P/Z$ is  quaternion of order at least $16$   and there exists a  normal subgroup  $Q$ of $P$  such that $QZ$ has  index $2$ in $P$, $Q \cap Z=1 $,  and $Q$ is quaternion   with the property that  no element of  $Q$ of order $4$ is  $\GL_{2n}(\bar {\mathbb F}_q)$-conjugate to  an element of  $P \setminus QZ$.  Moreover, if $q$ is a square, then  
there exists an element $x \in Q$ of order   $(q-1)_2$    such that $\sigma(x)$ is not $\GL_{2n}(\bar {\mathbb F}_q)$-conjugate to $x$ or $-x$.
\item   $n_0=0$, $r=1 $,  $n_1=2$,  $Z=1$, $P$ is quaternion,  and either $q$ is not a square or   $\sigma(s)$ is not $\GL_{2n}(\bar {\mathbb F}_q)$-conjugate to $s$. 

\item $n_0=0$, $r=1 $,  $n_1=2 $,  $Z=\langle -I \rangle$,  $P$  is    not quaternion, $P/Z$  is  dihedral, and either  $q$ is not a square or   $\sigma(s)$ is not $\GL_{2n}(\bar {\mathbb F}_q)$-conjugate to $s$. 
\end{enumerate}
 \end{Lemma}

\begin{proof}   Let $u_0 $ be the central involution of $\SO(V_0) $  (if   $n_0 \ne 0$), and let $u_i$ be the central involution of $\GL_{n_i}(\epsilon_iq^{d_i} ) $. Then $-I =\prod_{i=0}^r u_i $.  Let   $P_0 $ be  a  Sylow $2$-subgroup   of $\SO(V_0) $   and  for  $i \geq 1 $, let    $P_i $  be  a Sylow $2$-subgroup of $\GL_{n_i}(\epsilon_iq^{d_i} ) $ such that $ P =  \Omega(V) \cap   \prod_{i=0}^r (P_i) $. Let  
$Q_i  = P \cap P_i$, $0  \leq i \leq r $,  let $Q_0' = P_0 \cap \Omega (V_0) $ and let $Q_i'=  P_i \cap  \SL_{n_i}(\epsilon_iq^{d_i} ) $, $ 1\leq i \leq r $.  Then $Q_0'$ is  a Sylow   $2$-subgroup of  $ \Omega (V_0) $ and  $Q_i'$ is  a  Sylow $2$-subgroup of    $\SL_{n_i}^{\epsilon_i} (q^{d_i} ) $,  $1\leq i \leq r $. Further,     $\prod_{i=0}^r Q_i'  \leq   P$  and for all $i$, $0\leq i \leq  r$, $Q_i $ is of index  at most  $2$ in $P_i$.

By  Lemma~\ref{lem:noSL2},  $n_0\leq  3  $ and  $n_i \leq 2 $ for all $i $ with $1\leq i \leq r $.  Suppose  first that  $n_0  =2 $. By Lemma~\ref{lem:noSL2}  $\eta(V_0) =-1$,  $Z(\Omega(V_0))=1 $,  $\Omega(V_0) \cong  \PSL_2(q^2)$  (see \cite[Prop.~2.9.1]{KlLi}) and  $\SO(V) = \langle u_0\rangle \times \Omega (V_0)$, and $P_0 =    \langle u_0\rangle \times  Q_0' $.   In particular,   $P_0$    is of rank  $3$, hence  $r \leq 1 $. Since  $n \geq 4 $,  it follows that  $r=1 $  and  by   Lemma~\ref{lem:noSL2}, $n_1 = 1 $ and  $P_1 $ is  non-trivial  cyclic.  Thus  $P_0 \times P_1 $ has  rank $4$  which forces $Z \ne 1 $.  The  rank of $P$ is at least $3$ and the rank of $P/Z$ is at  least  $2$, hence $P$ is  neither  quaternion, nor dihedral and $P/Z$ is not quaternion.   Since $P_1 $ is cyclic, if  $Q_1 \ne 1 $, then $u_1 \in  Q_1 $,  and since $-I =u_0u_1 \in P$, this would imply  that $u_0 \in   Q_0 $, and consequently $P_0 \leq  P$, a contradiction since   the rank of $P_0Z/Z  \cong P_0 $ is  $3$.  Thus $Q_1 =  1 $.  On the other hand, the index of $Q_1 $ in $P_1$ is at most $2$. It follows that     $4\nmid  |\GL_1 (\epsilon_1q^{d_1})| $.    Now suppose that $q$ is a  square.  Then  $\epsilon=-1 $.   Let $\Lambda'$ be  the set of distinct eigenvalues of $s$  different from $1$  and let $\lambda \in \Lambda'$.   Since $\epsilon=-1 $, ${\mathbb F}_q [\lambda ]: {\mathbb F}_q|=2d_1$ is even  and  consequently, $|{\mathbb F}_{q_0} [\lambda ]: {\mathbb F}_{q_0}|=4d_1$. It follows that  $\Lambda' $ is  not invariant  under  $\mu \to  \mu^{q_0} $ and consequently   the $\GL_{2n} (\bar{\mathbb F}_q)$-class of $s$ is not $\sigma$-stable.  Thus Case (ii) holds

 Suppose that $n_0=1 $  and set  $\eta(V_0) =\eta_0 $. Then  $\O (V_0) $ (respectively   $\SO(V_0)$, respectively   $\Omega(V_0) $ is  isomorphic to the dihedral group of order $2 (q-\eta_0)$  (respectively cyclic group of order $q-\eta_0$,  respectively cyclic group of order   $\frac{q-\eta_0}{2}$) (see \cite[Thm.~11.4]{Tay06}).    
In particular,    $P_0$ is   cyclic of order $   (q -\eta_0)_2  > 1 $.  If $n_i \leq 1 $ for all $i$, $1\leq i \leq r$, then  $P$ is abelian   and case (i) holds.  We  may assume from now on  that  $n_1 =2$.    Since $P_0 \ne 1 $    it follows from Lemma~\ref{lem:noSL2} (c) that $r=1 $,  $Z =\langle -I \rangle  $,  and   $ 4 \nmid  q^{d_1}  - \epsilon_1  $  and $P/Z$  is quaternion. 
In particular,  $ n = 1  +2d_1 $ is odd, and $\eta =\eta_0$.
Since $Z \ne1 $, we have that   $q- \eta_0 \equiv     q^{n} - \eta  \equiv 0 \pmod 4$.   Thus  $|P_0 \times P_1| \geq   8(q^{2d_1} -1) _2 \geq 64  $   and  consequently $|P/Z| \geq 16 $.  Since $ Z \cap P_1  =1 $, $ P_1  $ is not contained in   $P$, hence  $Q_1 $  is  isomorphic to a Sylow $2$-subgroup of  $\SL_2(q^{d_1}) ) $   and is of  order  $(q^{2d_1} -1) _2 $. Since $P$ is of index at most $2$ in $P_0 \times P_1 $, we  have  that $Q_1Z \ne P $.
Let $x$ be an element of order $4$ in $Q_1$.   As an element of $\GL_2(q^{d_1})$,  $ x$  has eigenvalues $i$ and $-i $ Then  as an element of $\GL_{2n}(\bar{\mathbb F}_q)$, $x$ has eigenvalues $i$ and $-i$  each with multiplicity  $n-1 \geq  3 $ and  $1$ with multiplicity  $2$.
Let $y$ be  an element of order $4$  of $P$ which is not in $QZ$. We claim that  the multiset of  eigenvalues  of $y$ is not that of $x$.  
Indeed, write  $y = y_0y_1$, $y_i \in P_i$, $ i=0,1$. Then $y_0 \ne 1$  since otherwise $y =y_2 \in P \cap  P_1 = Q_1$.  If $y_0 $  has order $2$, then $-1$  is an eigen-value of $y$. Thus $y_0$  has order $4$ and  since $y_0  \in \SL(V_0)$ and $V_0$ is  $2$-dimensional it follows that  the eigenvalues of  $y_0$ in $\GL(V_0) $ are $i$  and  $-i$, each with mutiplicity  $1$. Thus,  the  multiplicity  of  $1$  as  an eigenvalue of  $y_1$ as an element of $\GL (V_1)$  is   $2$.   This is impossible as  either this multiplicity is $0$ or   at least $n-1 \geq  4$.   Now let $x \in Q_1$ be an element of  order $(q-1)_2$.  Then  the eigenvalues of  $x$ as an element of $\SL_2(q^{d_1})   $ are  $\eta $ and $\eta^{-1}$, hence  as element of   $\GL_{2n} (\bar {\mathbb F}_q)$   the eigenvalues of   $x$  are  
  $\eta $ and $\eta^{-1}$  each with multiplicity  $n-1 $ and  $1$  with multiplicity  $1 $. Then by Lemma~\ref{lem:nofield}, we are in case (iii).

Now suppose that $n_0=0$.  If  all $n_i \leq 1 $, then we are on case  (i).  So we may assume that $n_1=2 $. By Lemma~\ref{lem:noSL2},  $r \leq 2 $.   Suppose first that $r=2 $.  Again by Lemma~\ref{lem:noSL2},  $n_2=1$,  $Z \ne 1 $,   $4\nmid (q^{d_1}-\epsilon_1) $,   $P/Z$ is quaternion,  $Q_1  Z$ is  of index  at  most $2$  in  $P$, and $  Q_1  $ is  of index $2$  in  $P_1 $.
Since  $Z \ne 1$  and $n - d_2 =2d_1$ is  even, we have that   $q^{d_2}- \epsilon_2 \equiv     q^{n} - \eta  \equiv 0 \pmod 4$  from which it follows  that $QZ$ is of index $2$ in  $P$  and that $|P|$  is  of order at least $16$.
Now arguing as in case (iii), we see that Case (iv) holds.

Finally, suppose that  $n_0=0 $  and $r=1 $. In this case, $\eta=1$.  If $Z =1 $, then by  Lemma~\ref{lem:noSL2}  and Lemma~\ref{lem:inbetween} $P=Q_1$ is quaternion and  $Q_1= Q_1' $. Since  $Q_1 $ is of index at most $2$ in $P_1$, it follows that 
$ 4 \nmid  (q^{d_1}-\epsilon_1)$.  Now suppose that   $q$ is  a square.   Then $\epsilon_1 =-1 $.   Arguing   as in Case (ii)   we see that   Case (v) holds.     Now suppose that  $Z \ne 1  $.  By  Lemma~\ref{lem:noSL2} and  Lemma~\ref{lem:inbetween}, $P/Z $ is dihedral  and $Q_1'$ is  of index at most  $2$ in $ Q_1' $. It follows that $ 8 \nmid  (q^{d_1}-\epsilon_1)$  and consequently if $q$ is a square, then  $\epsilon=-1 $.  Arguing as above, we see that Case (vi) holds.\end{proof}

Since special orthogonal groups in   even dimension are self-dual   (and  since centralisers of semi-simple elements  of odd order in type-$\type{D}$ groups  are Levi subgroups)  we have  that the  blocks of $\CO \SO(V)$ are  in  bijection  with $\SO(V)$-conjugacy classes of semi-simple elements  of   odd order; if a  block  idempotent  $e$  corresponds to the class of $s$, then  the   Sylow $2$-subgroups   of $C_{\SO(V)}(s)$  are  defect groups of $\CO \SO(V) e$, and moreover,  for any  (relevant)   automorphism  $\alpha $ of $\SO(V)$,  $\,^\alpha e   $  corresponds to   the class of $\,^\alpha  s$  (see e.g. \cite[Prop.~7.2]{taylor18}).  Further, if $c$  is a  block idempotent  of $\Omega (V)  $ covered by  $e$, then   the Sylow $2$-subgroups of $C_{\SO(V)}(s) \cap \Omega (V)$ are     defect groups of  $\CO \Omega (V) c$.

For the following  result,  note that the spin group $\Spin(V)$   may be identified  with $\bG^F$, where $\bG$ is  a simple, simply-connected group  of type $\type{D}_{n}$ over  $\bar{\mathbb F}_q$, and $F$ is  a Frobenius morphism.  

 \begin{Theorem}  \thlabel{thm:typeD} Let  $G= \Spin (V)$,  $Z \leq  Z(G) $, and  $ N =  G/Z$. Let $H$ be  a group containing $ N$ as a subgroup of index $2$, and let $d$ be a $G$-stable  block idempotent of $\CO  N $  such that $\CO  N d$ is not nilpotent. Let $D$  be  a defect group of  the block $\CO Hd $ and   let  $D=\langle  D\cap  N, t\rangle $. 
Suppose that   both $D$ and $D\cap  N$  are quaternion groups.    Then $t$ acts by diagonal automorphisms on $ N$. 
\end{Theorem}

\begin{proof}   Set  $Z_u =  Z(G)$  and  recall that   $Z_u \cong C_4 $ if  $n$ is odd and $4\mid  (q -\eta)$,
 $Z_u \cong C_2$ if  $n$ is odd and $4\nmid  (q -\eta)$, $Z_u  \cong  C_2 $ if  $n $ is even  and $\eta =-1$,  and 
 $Z_u \cong  C_2  \times C_2 $ if  $n $ is even  and $\eta =1$  (see \cite[Thm.~2.5.12]{GLS3}). Let $Z_0 \leq Z_u $ be  the  subgroup of  order $2$  such  that $ \Omega (V)  =  G/Z_0$.   If  $Z' \leq Z_u$, then 
 any  automorphism  of  $G/Z'$ lifts   to a unique automorphism  of $G$;  the  resulting map $\Aut(G/Z') \to \Aut (G) $   is  injective and has as image the automorphisms of  $G$  which  stabilise $Z'$  (see  \cite[Thm.~2.5.14]{GLS3}). We  use the same symbol to denote  an automorphism of  $G/Z'$  and  its  lift to $G$. 
 Let $g \in \O(V)\setminus  \SO(V) $  (be an element of order $2$) and denote by $\gamma $ (the lift of)  the automorphism of $\Omega (V)$  induced  by  conjugation  by  $g$.  If $q=q_0^2$  is a  square and $\eta=1 $, then denote by $\sigma $  (the lift of) the  automorphism  of $\Omega (V) $ corresponding  to  raising every entry to the $q_0$-th power. 

Let  $\varphi \in \Aut(G)$ be the lift of  the conjugation action of  $t$ on $N$.  
The image of  $\varphi $ in $\Out (G) $   has  order at most  $2$  and $ Z$  is $\varphi$-stable.  Thus, again by \cite[Thms.~2.5.12,~2.5.14]{GLS3} unless $n=4 $ and  $\eta=1 $, there is a diagonal automorphism  $\varphi_0$  such that  either $\varphi=  \varphi_0 $, or   $\varphi =  \varphi_0 \gamma  $ or $q =q_0^2 $,  $\eta  =1 $ and $\varphi=   \varphi_0 \gamma  \sigma $  or $\varphi_0 \sigma $. If $n=4 $ and  $\eta=1 $, then in addition to the above possibilities,   $\varphi$ could also be  the  conjugate of  an automorphism of the above form  by   a power, say  $\tau $,  of  a  triality  automorphism;  in this case, replacing  $N$ by  $\tau (N)$, $\varphi $ with $\tau \varphi \tau^{-1}$,  and $H$  by an overgroup  $H'$   of $\tau(N)$   such that  $\tau $  extends to an isomorphism  from $H$ to $H'$ (note that such a $H'$ exists), we  may assume that  $\tau=1$. In particular, 
 $Z_0$  is $\varphi$-stable.
Further, note that  there exists  $h \in \GL_{2n} (\bar{\mathbb F}_q) $    such that  $\varphi_0  $   is conjugation by $h$  on $\Omega (V) $,  $h$ normalises  $\SO(V)  $   and  $h$ stabilises  the   $\SO(V)$-conjugacy class of   any semi-simple element of odd order. In particular, $\varphi_0 $  (respectively $\varphi $)  extend to automorphisms  of $\SO(V) $   which we denote again by  $\varphi_0$  (respectively $\varphi$) (see \cite[Thm.~2.5.14]{GLS3}). Finally, note that $\gamma $ and $\sigma $  commute.

Let $d_u$  be the  block idempotent of $\CO G$ dominating  $d$, let $ d_a $  be the  block idempotent of  $\CO(G/Z_u) $  dominated by $d$  and let   $d_0$ be the block idempotent of $\CO \Omega (V)  $  dominated by $d_u$. Let  $e$  be the block  idempotent of  $\SO(V)$   covering $d_0$  and let   $s \in \SO(V)$  be the  conjugacy   class corresponding to $e$. Let   $P$   be  a Sylow $2$-subgroup of $C_{\SO(V) } (s)\cap  \Omega (V)$,  a defect group of $\CO \Omega (V)  d_0$. Note that  since $d$ is $\varphi$-stable, and  $d_u$, $d_a $,  $d_0$ and $e$  are all unquely determined by $d$, $d_u$, $d_a $,  $d_0$ and $e$  are  also all $\varphi$-stable.  Consequently  the  $\SO(V)$-class of $s$ is  $\varphi$-stable. 

Suppose  if possible that    $\varphi $ is not diagonal.  So, either  $\varphi =  \varphi_0 \gamma  $ or $q =q_0^2 $,  $\eta  =1 $ and $\varphi=   \varphi_0 \gamma  \sigma $  or $\varphi_0 \sigma $.   Consider first the case  that   $Z =Z_0 $. Then   $N=  \Omega (V)$    and  we may assume that  $P= D\cap N  $ is quaternion.   Then, we 
 are  in Case (v) of Lemma~\ref{lem:Dsylow}. Since $n_0=0 $, this means  that  $\gamma $  does not stabilise   the $\SO(V)$-class of $s$.  Hence   $q$  is a  square  and either  $\varphi=   \varphi_0  \sigma  $  or  $\varphi=   \varphi_0\gamma\sigma  $.   Since $\varphi $  stabilises  the $\SO(V)$-class of $s$, $\varphi_0= c_h $ and $\gamma =c_g $, it follows that  $\sigma  (s)  $ is  $\GL_n(\bar {\mathbb F}_q) $-conjugate  to $s$,  a  contradiction to  Lemma~\ref{lem:Dsylow}.

Next suppose that  $Z \ne Z_0 $ but    $N$ is  a   quotient  of $\Omega(V)$.  Then  $Z_u =Z$ is  cyclic  of order $4$, $Z(\Omega(V)) \ne 1 $, and  $N =  \Omega (V)/Z (\Omega (V) ) $. Further,  we may assume  that  $ P/Z (\Omega (V) )= D\cap N  $  is quaternion.
Thus   we are in case  (iii) or  (iv) of Lemma~\ref{lem:Dsylow}.  The case that  
 $q =q_0^2 $,  $\eta  =1 $ and $\varphi=   \varphi_0 \gamma  \sigma $ or $ \varphi_0 \sigma $  is ruled out  by Lemma~\ref{lem:nofield}.   So,  $\varphi=   \varphi_0 \gamma = c_{hg} $.   Since  the  $\SO(V)$-class of $s$ is  also $\varphi$-stable, we  are in fact in  Case  (iv) of 
Lemma~\ref{lem:Dsylow}. Let  $Q \leq  P$  as in Case (iv).  Then  $QZ(\Omega(V)) / Z(\Omega(V))   $  is  a quaternion subgroup of  index  $4$  in  $D$, hence non-normal.  Let  $y \in Q$  be an element of order $4$  such that $y Z(\Omega (V))  $  is not in the cyclic  subgroup of index $2$ of $P/ Z(\Omega (V))  $. Then  $tyZ(\Omega (V)) t^{-1} \notin QZ(\Omega (V)) t^{-1}$, hence  $\, ^{hg} y \notin  QZ$. But this is a  contradiction   to Lemma~\ref{lem:Dsylow}.

Now suppose  that   $N$ is not a  quotient of  $\Omega (V)  $  and   $Z_u$ is  cyclic.  In this   case,   $Z_0= 1  $ and $N =G$.  
Since $N$ has  a quaternion block,   $Z_u$  has order $2$  whence $ \Omega(V)   =  N/  Z_u $,  $P =  D \cap N/Z_u $  is   dihedral  and $Z(\Omega(V) )=1 $.  Thus  $s$ is as in Case (i)  of Lemma~\ref{lem:Dsylow},  which means  that $\CO SO(V) e$   and consequently   $\CO  \Omega (V) d_0$ and   $\CO N d$   are nilpotent,  a contradiction.

So, we  may assume  that  $N$ is not a  quotient of  $\Omega (V)  $  and   $Z_u$ is  not  cyclic. In this case, $Z_u = Z_0 \times Z  \cong C_2 \times C_2 $,   $n $ is even, $\eta =1 $  and  $Z(\Omega (V)) \ne 1 $. Now $\langle \gamma \rangle$ acts  faithfully on $Z_u$  and  stabilises  $Z_0$, thus  does not stabilise $Z$, whereas   $\sigma $ and $h$  do stabilise  $Z$ (see again \cite[Thms.~2.5.12,~2.5.14]{GLS3}).  Therefore,  since  $Z$ is $\varphi$-stable,  we have that  $q$  is a square and   $\varphi= c_h \sigma  $. 
Moreover, we may  assume that $P/Z(\Omega(V))   $     is isomorphic  to   the quotient of   $ D \cap N   $ by the  unique central subgroup of $N$ of  order $2$   (both are isomorphic to defect  groups of $\CO (G/Z_u) d_a$)  and in particular  $P/Z(\Omega(V))   $   is   dihedral.     Thus, we are in Case (i), (ii) or (vi) of Lemma~\ref{lem:Dsylow}.   Case (i) yields  that $d$ is nilpotent.   Case (ii) and Case (vi)  yield that  the $\SO(V)$-class of  $s$ is not $\varphi$-stable,  a contradiction.
\end{proof} 

With this, we are now ready to show that in the situation of \thref{main-reduction}(i, ii) for groups of Lie type, the block is Morita equivalent to a principal block: 
\begin{Theorem} \thlabel{lietype} 
Let $(H,c)$ be a Morita minimal pair such that $H$ is a finite group and $c$ is a block idempotent of $\CO  H$ with defect group $P \cong Q_{2^n}$ for some $n \geq 3$, and such that $H$ is either quasisimple or there is $N\lhd H$ with $N$ quasisimple, $[H:N]=2$, $H=NP$, and $N\cap P\cong Q_{2^{n-1}}$. (That is, $H$ and  $c$ are as in parts (i) or (ii)  of 
\thref{main-reduction}.) In the first case,  set  $ N=H$.  
Suppose that $N= \bG^F/Z$, where  $\bG$  is a simple, simply-connected  group and $F$   is a  Frobenius morphism on $\bG$ with respect to an ${\mathbb F}_q$-structure. Then  $\CO Hc $ is Morita equivalent to a principal block.
\end{Theorem}

\begin{proof}  By \thref{cor:notB_nC_nF4G2E8}, we may assume that  $H \ne N $ and  $\bG$ is  of type $\type{E}_6$, $\type{A}_n$  or $\type{D}_n$. In the first case, the result follows  from \thref{E6}. In the latter two cases the result  follows  from \thref{thm:typeA}, \thref{thm:typeD}, and \thref{cor:extendclassical2}.
\end{proof}

We end this section by remarking that the work in this section could also be applied to the classical groups of types $\type{B}$ and $\type{C}$, although these groups are dealt with by different methods here. Also,  following through the reductions, one can prove that  \thref{lietype}  holds without the assumption on Morita minimality.

\section{The remaining (almost) quasi-simple groups}\label{sec:altspor}

In this section, we consider the remaining quasi-simple groups (namely, the alternating groups, sporadic simple groups, groups with exceptional Schur multiplier, Ree groups and groups of Lie type in defining characteristic) and almost quasi-simple groups with these as index-2 subgroups. 

\subsection{Alternating Groups} \label{alt}
We refer the reader to \cite[Secs.~11--13]{olsson93} for background on blocks of symmetric and alternating groups and their double covers.
In particular, to each block of the symmetric group $\sym_n$, there is associated a weight $w$, and the defect group of a block $\CO\sym_n\bar{c}$ of $\sym_n$ of weight $w$ is a Sylow $2$-subgroup of $\sym_{2w}$. (See e.g. \cite[Prop.~11.3]{olsson93}.)  
If $\CO\alt_n\bar{b}$ is a block of $\CO\alt_n$ covered by $\CO\sym_n\bar c$, then the defect group of $\CO\alt_n\bar{b}$ contains a Sylow $2$-subgroup of $\alt_{2w}$. The latter contains more than one involution for $w\geq 2$. It follows that no block of $\CO\alt_n$ (or $\CO\sym_n$) has quaternion defect groups. Then we may turn our attention to the double covers.

We denote by $\wt\sym_n$ and $\wt\alt_n$ a double cover of $\sym_n$, respectively $\alt_n$.

\begin{Proposition}\label{prop:alternatinggenquat}
    Let $\CO  \wt\alt_nb$ be a $2$-block  of $\wt\alt_n$. Then $\CO\wt\alt_n b$ has  quaternion defect groups if and only if the block $\CO \wt\sym_n c$ of $\CO \wt\sym_n$ covering $\CO\wt\alt_n b$ dominates a block $\CO \sym_n\bar c$ of $\CO\sym_n$ of weight $w=2$ or $3$. (In these cases, the defect groups are $Q_8$ and $Q_{16}$, respectively.)  Consequently, if $\CO\wt\alt_n b$ has  quaternion defect groups, then  $\CO\wt\alt_n b$  has  three simple modules.  The blocks of maximal defect of $6.\alt_6$ and $6.\alt_7$ have $Q_{16}$ defect groups. The blocks of maximal defect of $6.\alt_6$ and the principal block of $6.\alt_7$ have three simple modules. The two non-principal blocks of maximal defect of $6.\alt_7$ have two simple modules. Further $6.\alt_7$ has a unique block with defect group $Q_8$.
\end{Proposition}
\begin{proof}
Let $Z=\zent(\wt\sym_n)$.
Since $\sym_n=\wt\sym_n/Z$ and $Z$ has size $2$, the blocks of $\wt\sym_{n}$ are in one-to-one correspondence with those of $\sym_n$, which can be seen by \cite[Thm.~9.9(b)]{N98} and the discussion before. Given a $2$-block idempotent $b$ of $\CO\wt\alt_n$, there is a unique $2$-block $\CO\wt\sym_n c$ covering $\CO\wt\alt_n b$ and unique $2$-blocks $\CO\alt_n\bar b$, $\CO\sym_n\bar c$ of $\alt_n$, respectively $\sym_n$  contained in $\CO\wt\alt_n b$, respectively $\CO\wt\sym_n c$.
The defect groups of $\CO\wt\alt_n b$ are then isomorphic to Sylow $2$-subgroups of $\wt{\alt}_{2w}$, where $w$ is the weight of $\CO\sym_n\bar{c}$. Further, $\CO\alt_n\bar b$, respectively $\CO\sym_n \bar c$ has the same number of  simple modules as $\CO\wt\alt_n b$, respectively $\CO\wt\sym_n c$.

Now, suppose that a defect group $D$ of $\CO\wt\alt_n b$ is quaternion of size $2^k$. Then $\zent(D)$ has size two, which forces $Z=\zent(D)$. Further, the defect group $D/Z=D/\zent(D)$ of $\CO\alt_n\bar b$ is dihedral of size $2^{k-1}$. In particular, $D/Z$ contains an element of order $2^{k-2}$. Note that $D/Z$ is a Sylow $2$-subgroup of $\alt_{2w}$. 

Now, the largest order of a $2$-power element of $\sym_{2w}$ is $2^t$, where $2w=2^t+a_{t-1}2^{t-1}+\cdots+a_1 2+a_0$ with $a_i\in\{0,1\}$ is the $2$-adic expansion of $2w$. But for $t\geq 3$, the Sylow $2$-subgroup of $\alt_{2^t}$ has size larger than $2^{t+1}$. It follows that $t\leq 2$, and hence $2w< 8$, so $w\leq 3$. From this, with our assumption that $D$ is quaternion and hence $D/Z$ is dihedral, we see that $|D/Z|$ must be $4$ or $8$. In the first case, we see $D \cong Q_8$.
If instead $|D/Z|=8$, then $D$ is isomorphic to a Sylow $2$-subgroup of $\wt\alt_{6}\cong \SL_2(9)$, which is indeed quaternion of size 16.   The  claim  about  the number of simple modules follows from the   first assertion since   if $w=2 $  or $3$, then $\CO\alt_n\bar b$  has three simple modules (see \cite[Prop.~12.9]{olsson93}).

The  assertions regarding the exceptional Schur covering groups $6.\alt_6$ and $6.\alt_7$  follow using \cite[Thm.~9.9]{N98}, together with the information in the GAP character table library \cite{GAP}.  
\end{proof}

\subsection{Sporadic Groups and Exceptional Schur Multipliers} \label{sporadic}
\begin{Proposition}\label{prop:sporadicgenquat}
   Let $\CO Gb$ be a $2$-block of a finite group $G$ with defect group $D$, and with a quasisimple normal subgroup $N$ of index dividing $2$ such that $G=DN$ and $N/\zent(N)$ is a sporadic simple group. Then $\CO Gb$ has quaternion defect groups precisely in the following situations:
   \begin{enumerate}
       \item $D \cong Q_8$ and $G \cong 2.M_{12}$, $2.J_2$ or $2.Ru$, where in each case $\CO Gb$ is the unique block with this defect group;
       \item $D \cong Q_{16}$ and one of the following:
       \begin{enumerate}
           \item $G=2.Co_1$ and $\CO Gb$ is the unique non-principal $2$-block. We have $\ell(\CO Gb)=3$;
           \item $G=2.Suz$ or $6.Suz$ and $\CO Gb$ is the unique $2$-block of $G$ with nonmaximal defect group. In the case of $2.Suz$, this is the unique non-principal $2$-block. We have $\ell(\CO Gb)=3$.
           \item $G/Z(G) \cong \Aut(J_2)$, where $|Z(G)|=2$, and $\CO Gb$ is the unique non-principal $2$-block. We have $\ell(\CO Gb)=2$.
       \end{enumerate}
   \end{enumerate}
\end{Proposition}
\begin{proof}

Let $\CO Nc$ be a block of $\CO N$ covered by $\CO Gb$. Note that since $G=DN$ we have that $\CO Nc$ is $G$-stable, and $D \cap N$ is a defect group for $\CO Nc$. Note also that $Z(G)=Z(N)$. Then either $O_2(Z(G))=1$ and $D \cap N$ is quaternion, or $|O_2(Z(G))|=2$ and $(D \cap N)/O_2(Z(G))$ is dihedral. By~\cite{landrock}, and taking into account Schur multipliers of the sporadic simple groups, we must have $|O_2(Z(G))|=2$, and $(D \cap N)/O_2(Z(G))$ is $C_2 \times C_2$ (and $N/Z(N)$ is $M_{12}$, $J_2$, $HS$, or $Ru$) or $D_8$ (and $N/Z(N)$ is $Suz$, $Co_1$, or the Baby Monster $B$). In the case of the Baby Monster $B$, we utilize the CTBlocks and CTblLib packages for GAP~\cite{GAP} to show that there are two conjugacy classes of involutions contributing elements of a defect group, which cannot then by quaternion (there are at least two conjugacy classes of involutions on which irreducible characters of the block do not vanish). In the other cases we  apply~\cite{GAP} to see that $N$ has a unique radical $2$-subgroup of order $|D \cap N|$, and that this is quaternion. We further use~\cite{GAP} to directly verify the number of simple modules in all cases except $2.Co_1$. In this last case, note that it suffices to determine the number of simple modules in the corresponding block of $Co_1$ with defect groups $D_8$. Using~\cite{GAP} the irreducible characters in this block are $\{\chi_{50},\chi_{66},\chi_{71},\chi_{93},\chi_{99}\}$, which have distinct degrees. But from the decomposition matrices in~\cite{er90} if the number of simple modules is two, then there is an irreducible Brauer character that lifts to two distinct irreducible ordinary characters. Hence $\ell(\CO Gb)=3$ in this case. Note that the block with defect group $Q_{16}$ of $6.Suz$ has $O_{2'}(Z(N))$ in its kernel.  

It remains to consider $2.Suz.2$, $2.M_{12}.2$, $2.J_2.2$ and $2.HS.2$. In each case the block of $N$ identified above is $G$-stable. In the first case potential blocks would have defect group of order $32$ and in the remaining of order $16$. We verify by~\cite{GAP} as above that in the case of $2.Suz.2$, $2.M_{12}.2$, and $2.HS.2$ each possess two conjugacy classes of involutions containing elements of a defect group, which then cannot be quaternion. Finally, in case $2.J_2.2$ we may use~\cite{GAP} to check that all radical $2$-subgroups of order $16$ are quaternion, and we are done.
\end{proof}

\begin{Proposition}\label{prop:excschur}
Let $\CO Gb$ be a $2$-block of a finite group $G$ with defect group $D$, and with a quasisimple normal subgroup $N$ of index dividing $2$ such that $G=DN$ and $N/\zent(N)$ is a simple group of Lie type with exceptional Schur multiplier, but is not isomorphic to an alternating group. Then $\CO Gb$ has quaternion defect groups precisely in the following situations:

\begin{enumerate}
\item the principal block of $G=2.\type{A}_2(2)$, which has $Q_{16}$ Sylow subgroups, or of $2.\type{A}_2(2).2$ which has $Q_{32}$ Sylow subgroups. In the former case $\ell(\CO Gb)=3$ and the latter $\ell(\CO Gb)=2$;
    \item one of three blocks of $G=6.\type{B}_3(3)$ or one block of $2.\type{B}_3(3)$ with $Q_{16}$ defect groups. In each case $\ell(\CO Gb)=2$.
\end{enumerate}

\end{Proposition}
\begin{proof}
 We use the same ideas as in Proposition \ref{prop:sporadicgenquat}. To deal with some of the groups with larger Schur multipliers, we also use the fact that if a block of $G$ has defect group $Q_{2^n}$, then it must dominate a block of $G/Z_2$ with defect group $Q_{2^n}/Z_2$, where $Z_2=O_2(\zent(G))$, together with the fact that $Q_{2^n}/\zent(Q_{2^n})$ is dihedral.
\end{proof}

\subsection{Groups of Lie type in defining characteristic} \label{sec:defining}
 \begin{Proposition}
     \label{prop:defchar}
Let $G$ be a quasi-simple group of Lie type defined in characteristic $2$ or let $G$ be the Ree group $\tw{2}\tG_2(3^{2m+1})$, and assume that $G$ is not isomorphic to one of the groups considered in Propositions \ref{prop:alternatinggenquat}, \ref{prop:sporadicgenquat},  or \ref{prop:excschur}. Then no block of $\CO G$ has quaternion defect groups $Q_{2^n}$ with $n\geq 3$. 
 \end{Proposition}
 \begin{proof}
 Since the Ree groups $\tw{2}\tG_2(3^{2m+1})$ have elementary abelian Sylow $2$-subgroups of size $8$, we may assume $G$ is defined in characteristic $2$. Then our assumptions yield that the Schur covering group of $G/\zent(G)$ is of the form $\bG^F$, where $\bG$ is a simple, simply connected algebraic group defined over $\overline{\mathbb{F}}_2$ and $F\colon \bG\rightarrow\bG$ is a Steinberg endomorphism. In this situation, $\zent(\bG^F)$ has odd order, so the defect groups of $G$ and $\bG^F$ are isomorphic by \cite[Thm.~9.9(c)]{N98}. Then we may assume without loss that $G=\bG^F$. 
 
 By a result of Dagger and Humphreys (see \cite{humphreys}), the only defect groups for $2$-blocks of $G$ are the trivial group and the Sylow $2$-subgroups. The group $U:=\mathbf{U}^F$, where $\mathbf{U}$ is the unipotent radical of an $F$-stable Borel subgroup $\mathbf{B}$ of $\bG$, is a Sylow $2$-subgroup of $G$. However, as argued in \cite[Prop.~4.9]{NRSV}, any generating set of $U$ has more than two elements, and hence $U$ is not quaternion.
\end{proof}

We remark that Propositions \ref{prop:structurequasisimplegenquat}, \ref{prop:alternatinggenquat},
  \ref{prop:sporadicgenquat},  \ref{prop:excschur}, and \ref{prop:defchar}, now give a description of the blocks of quasi-simple groups whose defect groups are quaternion.

\section{Proofs of the main results.}\label{sec:mainproofs}

\begin{proof}[Proof  of \thref{Main2}.] Let $H,c $ be  as in \thref{main-reduction}.  We may assume that $H$ is of type (i) or (ii). If $H$ (in case (i)) or $N$ (in case (ii)) is a finite group of Lie type in non-defining characteristic, then by \thref{lietype} $H$ (or $N$) is Morita equivalent to a principal block. Suppose that $H$ (or $N$) is a quasisimple  group  as in  Section \ref{sec:altspor}. If $l(\CO Hc)=3$, then, as described in \thref{Erdmann_class_plus}, $\CO Hc$ is Morita equivalent to a principal block. Otherwise, by Propositions \ref{prop:alternatinggenquat}, \ref{prop:sporadicgenquat},  \ref{prop:excschur} and \ref{prop:defchar} $l(\CO Hb)=2$ and we have the cases listed.
\end{proof}

\begin{proof}[Proof of \thref{Main1}.] The Morita--Frobenius number  of any principal  block  equals $1$. It then suffices to consider the non-principal  blocks  listed in \thref{Main2}. In all cases except the two non-principal blocks of maximal defect of $6.\alt_7$ and the two non-principal blocks of maximal defect of $6.\type{B}_3(3)$, either the number of simple modules is 3 and as in \thref{Erdmann_class_plus} the block is Morita equivalent to a principal block, or the block in question is the unique non-principal block with that defect group. In both cases the Morita--Frobenius number must then be 1. Finally, the $2$-blocks of $6.\alt_7$ and of $6.\type{B}_3(3)$ have Morita--Frobenius number 1 by~\cite[Prop.~7.2]{FK}. The last statement now follows, considering Remark \ref{rem:mfbound}.
\end{proof}

 \begin{proof}[Proof of \thref{tame-Donovan-Corollary}.]   For  blocks with dihedral (including Klein-$4$) or semidihedral defect groups, the result is contained in \cite{er90}. The  case of blocks with quaternion defect groups  follows  from \thref{Main1}.
 \end{proof}

 \begin{proof}[Proof of \thref{MainCorollary2}.]  This follows by combining \thref{Main2} and \cite[Thm.~1.1]{kl20}.\end{proof}

We close by recording the distribution of $k$-blocks in \thref{MainCorollary2} into classes in the notation of~\cite{er90}. 

\begin{Remark}
In the notation of \thref{MainCorollary2}:
\begin{enumerate}
    \item principal blocks of $\SL_2(q_1)$ belong to the class $Q(3A)_2$;
    \item principal blocks of $\SL_2(q_2)$ belong to $Q(3K)$;
    \item the principal block of $2.A_7$ belongs to $Q(3B)$;
    \item the blocks of $2.\Aut(J_2)$ and $2.S_7$ with $Q_{16}$ defect groups, and the two faithful blocks of $6.\type{B}_3(3)$ with $Q_{16}$ defect groups belong to $Q(2A)$;
    \item the nonprincipal blocks of maximal defect of $6.A_7$ and the unique block of $2.\type{B}_3(3)$ with $Q_{16}$ defect groups belong to $Q(2B)_1$.
\end{enumerate}
\end{Remark}


\end{document}